\documentclass[a4paper,twoside,reqno]{amsart}
\usepackage[utf8]{inputenc}

\usepackage{amsmath,amssymb,amsfonts,amsthm,enumerate,mathtools,tikz}

\usepackage{fullpage}
\usepackage{hyperref}
\renewcommand{\MR}[1]{\href{http://www.ams.org/mathscinet-getitem?mr=#1}{MR#1}}

\usepackage{color}

\numberwithin{equation}{section}

\newcommand{\Nor}{\mathrm{Nor}}

\newcommand{\cA}{\mathcal{A}}

\newcommand{\cN}{\mathcal{N}}

\newcommand{\cO}{\mathcal{O}}

\newcommand{\cG}{\mathcal{G}}

\newcommand{\cL}{{\mathcal{L}}}
\newcommand{\cR}{{\mathcal{R}}}

\newcommand{\asX}{\langle X \rangle}

\newcommand{\la}{{\triangleright}}
\newcommand{\ra}{{\triangleleft}}

\newcommand{\N}{\mathbb{N}}

\newcommand{\LL}{\mathcal{L}}

\newcommand{\Sym}{\mathop{{\rm Sym}}\nolimits}

\def\gkdim{\operatorname {GK dim}}
\def\gldim{\operatorname {gl\,dim}}

\def\id{{\rm id}}
\theoremstyle{plain}

\newtheorem{thm}{Theorem}[section]
\newtheorem{pro}[thm]{Proposition}
\newtheorem{lem}[thm]{Lemma}
\newtheorem{cor}[thm]{Corollary}

\theoremstyle{definition}

\newtheorem{ex}[thm]{Example}

\newtheorem{convention}[thm]{Convention}

\newtheorem{dfn}[thm]{Definition}
\newtheorem{que}[thm]{Question}
\newtheorem{rmk}[thm]{Remark}

\newcommand{\LM}{\mathbf{LM}}

\newcommand{\extd}{{\rm d}}
\newcommand{\tens}{\otimes}
\renewcommand{\k}{{\bf k}}
\newcommand{\del}{\partial}

\title{Quadratic algebras and idempotent braided sets}

\author{Tatiana Gateva-Ivanova$^*$}
\address{Max Planck Institute for Mathematics, Vivatsgasse 7, 53111 Bonn, Germany,
and American University in
Bulgaria, 2700 Blagoevgrad, Bulgaria} \email{tatyana@aubg.edu}

\author{Shahn Majid}
\address{School of Mathematical Sciences\\ 327 Mile End Rd, London E1 4NS, UK}
\email{s.majid@qmul.ac.uk}

\thanks{${}^*$ partially supported by the Max Planck Institute for Mathematics,
Bonn,
by the Abdus Salam International Centre for Theoretical Physics, Trieste, and
by the American University in Bulgaria.
}

\begin{document}
\date{\today. Ver2}

\begin{abstract} We study the Yang-Baxter algebras $\cA(\k,X,r)$ associated to finite set-theoretic solutions $(X,r)$ of the braid relations. We introduce an equivalent set of quadratic relations $\Re\subseteq\cG$,  where $\cG$ is the reduced Gr\"obner basis of $(\Re)$. We show that  if  $(X,r)$ is left-nondegenerate and idempotent then $\Re=\cG$ and the Yang-Baxter algebra is PBW. We use graphical methods to study the global dimension of PBW algebras in the $n$-generated case and apply this to Yang-Baxter algebras in the left-nondegenerate idempotent case.  We study the $d$-Veronese subalgebras for a class of quadratic algebras and use this to show that for $(X,r)$ left-nondegenerate idempotent, the $d$-Veronese subalgebra $\cA(\k,X,r)^{(d)}$ can be identified with $\cA(\k,X,r^{(d)})$, where  $(X,r^{(d)})$ are all left-nondegenerate idempotent solutions.  We determined the Segre product in the left-nondegenerate idempotent setting. Our results apply to a previous class of  `permutation idempotent' solutions, where we show that all their Yang-Baxter algebras for a given cardinality of $X$ are isomorphic and are isomorphic to their $d$-Veronese subalgebras.  In the linearised setting, we construct the Koszul dual of the Yang-Baxter algebra and the Nichols-Woronowicz algebra in the idempotent case, showing that the latter is quadratic. We also construct noncommutative differentials on some of these quadratic algebras.  \end{abstract}
\maketitle

\setcounter{tocdepth}{1}
\tableofcontents

\section{Introduction}
\label{Intro}

The linear braid or Yang-Baxter equation (YBE) for a map  $R : V\otimes V \to
V\otimes V$ on a vector space $V$ was extensively studied in the 1980s and solutions lead both to knot invariants in
nice case and
to quantum groups, such as the  coquasitriangular bialgebras
$A(R)$ and their Hopf algebra quotients, covariant quantum planes and other structures, see e.g. \cite{MajidQG,
FRT,Ma88}. Early
on, V.G. Drinfeld \cite{D}, proposed to also consider the parallel equations for  $r:X\times X\to X\times X$ where $X$
is a set,
and  by now numerous results in this setting have been found, particularly in the involutive case, e.g.
\cite{ESS,GI96,GI04,GI04s,GI12,GIM08,GIM11,GI18,GIVB, Vendramin16, rump}. Non-involutive or strictly braided
set-theoretic solutions here are
less well understood but of increasing interest, starting with \cite{soloviev, Lu2000}. They have been used to produce
knot and virtual knot
invariants\cite{nelson} and, more recently, certain non-involutive solutions have been shown to arise from skew braces
\cite{Gua_Ven}. Thus,
non-involutive solutions  and some of their related algebraic structures have attracted significant further attention,
see for instance
\cite{Gua_Ven,  ce_smok_vendr19, smok-vendr18, Bachiller18, Vendramin_prob, AL_vendr20, Br19, Colazzo22, Doikou21,
doikou_smok21,
doikou_smok23, doikou_ryb23, Lebed17, Lebed20, Stanovsky21, GI21, GIM08} and references therein.

On the algebra side, we will be particularly interested in quadratic `Yang-Baxter' algebras $\cA(\k , X, r)$  over a
field $\k$ proposed in
\cite[Sec 6]{GIM11}  as analogues
of the `quantum planes' in the linear R-matrix theory. In that work, the main results were for $r$ a multipermutation
(square-free) involutive solution of level two. More generally, it is known\cite{GIVB,GI12} that if $X$ is finite  and $r$ is nondegenerate
and
involutive then  $\cA(\k,X,r)$ has remarkable algebraic, homological and
combinatorial properties. Although in most cases not a PBW algebra,  it shares various good
properties of the commutative polynomial ring $\k [x_1, \cdots , x_n]$, such as  finite global dimension, polynomial
growth,
Cohen-Macaulay, Koszul, and being a Noetherian domain.  More recently, in \cite{AGG} another `almost commutative' class of quadratic PBW
algebras called `noncommutative projective spaces' was investigated and analogues of  Veronese subalgebras and Segre morphisms between them introduced and studied. The Veronese subalgebra $A^{(d)}$ of a quadratic algebra $A$ is
defined as the subalgebra of elements with degrees that are divisible by $d$. These and related Segre products were previously studied
in a noncommutative setting for general Koszul algebras  by Backelin and Froeberg in \cite{Backelin,Froberg}. The important case of these for Yang-Baxter algebras for involutive solutions appeared in \cite{GI23,GI_Veronese}.

Parallel results for $\cA(\k,X,r)$ when $r$ is strictly braided are very much unknown and are the topic of the present work. Our first result, Proposition~\ref{pro:main1}, is a new set of relations $\Re$ for the YB algebra that depends on an enumeration of $X$, coincides in the involutive case with the standard relations but is better behaved in the strictly braided case, being contained in the unique reduced Gr\"obner basis $\cG$ of the ideal  defining the algebra. A special case on which we will illustrate many of our results is the strictly braided case where $r$ is finite, left nondegenerate  and idempotent. We show in Section~\ref{sec:idemp} in this case that
$\Re$ is the reduced Gr\"obner basis of the ideal of relations and therefore $\cA(\k,X,r)$ is always a PBW algebra, see Theorem~\ref{thm:main1}. Idempotent solutions of YBE were studied in \cite{Lebed17, Lebed20, Stanovsky21, Colazzo22, Colazzo23}

Another starting point for the paper is some new results about general PBW algebras in Section~\ref{sec:graphs} using a pair of mutually dual graphs associated to these to show,  Theorem~\ref{thm:gldiminf},  that an
arbitrary $n$-generated PBW algebra $A$ with Gelfand-Kirillov dimension $< n$  has infinite global dimension. Another
result,
Theorem~\ref{thm:new}, provides (an exact) lower and an upper bound for the dimension of the grade 2 component of
$\cA(\k, X, r)$
in the case where this is PBW and $(X, r)$ is a left-nondegenerate idempotent quadratic set. Equivalently, the theorem
provides a lower
and an (exact) upper bound for the  number of relations in the reduced Gr\"{o}bner basis for this Yang-Baxter algebra. Corollary~\ref{cor:new} covers the case where $r$ solves the YBE.

Section~\ref{sec:perm} is devoted to the further results on a subclass of `permutation' idempotent left-nondegenerate solutions first studied in \cite[Prop.~3.15]{Colazzo22}. These depend on a bijection $f:X\to X$ and have the form $r_f(x,y)=(f(y),y)$ for all $x,y\in X$. Using our general results from Section~\ref{sec:idemp}, we give an explicit
 presentation
 of $\cA(\k , X, r_f)$ in terms of generators and $n(n-1)$ quadratic relations  forming the reduced Gr\"{o}bner
 basis, where $n=|X|$, see Proposition~\ref{pro:YBalgPermSol} and Corollary~\ref{cor:main}. Since these relations do not depend on $f$, our result is that all the  $\cA(\k , X, r_f)$ are isomorphic for a given $|X|$, indeed the same algebra for a given enumeration. By contrast, the number of
 \emph{non isomorphic} permutation
 solutions $(X, r_f)$ is the number of conjugacy classes in $\Sym(X)$ and hence the partition function $p(n)$,  a potentially large number.

Section~\ref{sec:ver} studies Veronese subalgebras and Segre products of $\cA(\k,X,r)$ particularly for left-nondegenerate idempotent braided sets $(X,r)$. We start in Section~\ref{sec:vergen} with a general result on the Veronese subalgebras $A^{(d)}$ for a well-behaved class of quadratic algebras, see Theorem~\ref{thm:vergen}. Section~\ref{sec:mon} provides some set-theoretic results that we next need about the set-theoretic solution on $(S,r_S)$ introduced in \cite{GIM08} on the monoid $S=S(X,r)$ and its conversion to an isomorphic braided monoid $(\cN,\rho)$ on the set of normal words. These restrict on each degree to isomorphic solutions $(S_d,r_d)$ and $(\cN_d,\rho_d)$ respectively. The latter have been called the normalised $d$-Veronese solutions associated to $(X,r)$ in \cite{GI23,GI_Veronese} in the involutive case and we build on the general strategy there. If $(X,r)$ is left-nondegenerate idempotent, we show that $\cN_d$ can be canonically identified with $X$ and referring $\rho_d$ back to this results in an infinite sequence of left-nondegenerate idempotent solutions $(X,r^{(d)})$.  Similarly  to the case of left-nondegenerate involutive case, we again have an algebra homomorphism $\cA(\k,\cN_d,\rho_d)\to \cA(\k,X,r)$ with image the $d$-Veronese subalgebra $\cA(\k,X,r)^{(d)}$, i.e. an analogue of the $d$-Veronese morphism in algebraic geometry. In the left-nondegenerate case, however, this turns out to be injective and appears equivalently as
\[ \cA(\k, X, r^{(d)})\hookrightarrow  \cA(\k,X,r),\]
with image $\cA(\k,X,r)^{(d)}$.  In the permutation idempotent case of $(X,r_f)$, we show that $(X,r_f^{(d)})= (X,r_{f^d})$ is again of permutation idempotent type but for $f^d$, which  means, in particular, that every $\cA(\k,X,r_f)$ contains in a natural way a copy of itself as a strict subalgebra (as the $m+1$-Veronese subalgebra if $m$ is the order of $f$).   We similarly show that the Segre products of YB-algebras in the left-nondegenerate idempotent case  are again Yang-Baxter algebras for the natural Cartesian product solution, as an analogue of the Segre morphism and in contrast to the involutive case in (\cite{GI23}) where the corresponding map has a large kernel.  The $d$-Veronese and Segre morphisms  play pivotal roles in classical algebraic geometry \cite{Harris} and in applications to other fields of mathematics.

A final Section~\ref{seclin} studies algebraic structures associated to R-matrices obtained by linearising idempotent solutions $(X,r)$, and to idempotent linear solutions of the Yang-Baxter or braid relations more generally. The transpose of a linearisation of a set-theoretic is  not necessarily a linearisation, giving a distinctly different `transpose' YB algebra $\cA^T(\k,X,r)$ associated a braided set $(X,r)$. Here $\cA(\k,X,r)=S_+(R)$ is a standard `covector' quantum plane and $\cA^T(\k,X,r)=S_+(R^T)$ is the standard `vector' version in terms of the coaction of the FRT bialgebra\cite{FRT}. We also compute the Koszul dual $S_+(R)^!$ in the idempotent case as a braided Hopf algebra in a certain (pre)braided category.  Likewise in the idempotent case, we show that the Nichols algebra $S^-_\Psi(V)$  associated to a finite-dimensional linear idempotent solution of the YBE is in fact quadratic and the Koszul dual of $S_+(R^T)$, see Theorem~\ref{nichols}. In the case where $R$ linearises $(X,r)$, this defines a natural `fermionic YB algebra' $\Lambda(\k,X,r)$. We also propose a  formulation for noncommutative differential structures $(\Omega^1,\extd)$ on  quadratic algebras and apply it to the permutation idempotent  YB algebra $\cA(\k,X,r_{\id})$ for $|X|=2$. We also provide a general construction of noncommutative differentials on $S^-_\Psi(V)$ in the idempotent case where, remarkably, the YB algebra $S_+(R)$ now appears as quantum differential forms.

\section{Preliminaries and first result}
\label{sec:pre}
This section provides basic algebraic preliminaries
for the paper as well introducing Definition~\ref{orbitsinG} of $r$-orbits and a preliminary new result,  Proposition~\ref{pro:main1}, about Gr\"obner bases for the YB algebra. These will be important for later sections but are not limited to the idempotent case.

Let $X$ be a non-empty set, and let $\k $ be a field.
We denote by $\asX$
the free monoid
generated by $X,$ where the unit is the empty word denoted by $1$, and by
$\k \asX$-the unital free associative $\k$-algebra
generated by $X$. For a non-empty set $F
\subseteq \k \asX$, $(F)$ denotes the two sided ideal
of $\k \asX$ generated by $F$.
When the set $X$ is finite, with $|X|=n \ge 2$, and ordered, we write $X= \{x_1,
\dots, x_n\},$ and fix the degree-lexicographic order $<$ on $\asX$, where $x_1< \dots
<x_n$.
As usual, $\N$ denotes the set of all positive integers,
and $\N_0$ is the set of all non-negative integers.

We shall consider associative graded $\k$-algebras.
Suppose $A= \bigoplus_{m\in\N_0}  A_m$ is a graded $\k$-algebra
such that $A_0 =\k $, $A_pA_q \subseteq A_{p+q}, p, q \in \N_0$, and such
that $A$ is finitely generated by elements of positive degree. Recall that its
Hilbert function is $h_A(m)=\dim A_m$
and its Hilbert series is the formal series $H_A(t) =\sum_{m\in\N_0}h_{A}(m)t^m$. For $m \geq 1$,  $X^m$ will denote the set of all words of length $m$ in $\asX$,
where the length of $u = x_{i_1}\cdots x_{i_m} \in X^m$
will be denoted by $|u|= m$.
Then
\[\asX = \bigsqcup_{m\in\N_0}  X^{m},\quad
X^0 = \{1\},\quad   X^{k}X^{m} \subseteq X^{k+m},\]
so the free monoid $\asX$ is naturally \emph{graded by length}. Similarly, the free associative algebra $\k \asX$ is
also graded by
length:
\[\k \asX
 = \bigoplus_{m\in\N_0} \k \asX_m,\quad
 \k \asX_m=\k  X^{m}. \]
A polynomial $f\in  \k \asX$ is \emph{homogeneous of degree $m$} if $f \in
\k  X^{m}$.

\subsection{Gr\"obner bases for ideals in the free associative algebra}
\label{sec:grobner}
We remind briefly some basics of the theory of noncommutative Gr\"obner bases, which we use throughout in the paper. In
this
subsection $X=
\{x_1,\dotsc,x_n\}$, we fix the degree lexicographic order $<$ on the free monoid $\asX$ extending $x_1 < x_2< \cdots
<x_n$ (we
refer to this as {\em deg-lex ordering}).
Suppose $f \in \textbf{ k}\asX$ is a nonzero polynomial. Its leading
monomial with respect to the deg-lex order $<$ will be denoted by
$\LM(f)$.
One has $\LM(f)= u$ if
$f = cu + \sum_{1 \leq i\leq m} c_i u_i$, where
$ c, c_i \in \k $, $c \neq 0 $ and $u > u_i$ in $\asX$, for all
$i\in\{1,\dots,m\}$. Given a set $F \subseteq \k  \asX$ of
non-commutative polynomials, we consider the set of leading monomials
 $\LM(F) = \{\LM(f) \mid f \in F\}.
 $
A monomial $u\in \asX$ is \emph{normal modulo $F$} if it does not contain any of
the monomials $\LM(f)$ as a subword.
 The set of all normal monomials modulo $F$ is denoted by $\cN(F)$. The set $F$ is \emph{reduced} if each  $f \in F$  is a linear combination
of normal monomials modulo $F\setminus\{f\}$.

Let  $I$ be a two sided graded ideal in $\k\asX$ and let $I_m = I\cap
\k X^m$.
We shall
assume that
$I$ \emph{is generated by homogeneous polynomials of degree $\geq 2$}
and $I = \bigoplus_{m\ge 2}I_m$. Then the quotient
algebra $A = \k  \asX/ I$ is finitely generated and inherits its
grading $A=\bigoplus_{m\in\N_0}A_m$ from $ \k  \asX$. We shall work with
the so-called \emph{normal} $\k$-\emph{basis of} $A$.
We say that a monomial $u \in \asX$ is  \emph{normal modulo $I$} if it is normal
modulo $\LM(I)$. We set
$\cN(I):=\cN(\LM(I))$.
In particular, the free
monoid $\asX$ splits as a disjoint union
\begin{equation}
\label{eq:X1eq2a}
\asX=  \cN(I)\sqcup \LM(I).
\end{equation}
The free associative algebra $\k  \asX$ splits as a direct sum of
$\k$-vector
  subspaces
\begin{equation}\label{Eq2kN} \k  \asX \simeq  \k  \cN(I)\oplus I,\end{equation}
and there is an isomorphism of vector spaces
$A \simeq \k  \cN(I).$

It follows that every $f \in \k \asX$ can be written uniquely as $f =
h+f_0,$ where $h \in I$ and $f_0\in {\k } \cN(I)$.
The element $f_0$ is called \emph{the normal form of $f$ (modulo $I$)} and denoted
by
$\Nor(f)$.
We define
\begin{equation}\label{eq:Nm}
\cN(I)_{d}=\{u\in \cN(I)\mid u\mbox{ has length } d\}.
\end{equation}
In particular, $\cN(I)_1 =X, \cN(I)_0=1$. Then
$A_m \simeq \k  \cN(I)_{m}$ for every $m\in\N_0$.

A subset
$G \subseteq I$
of monic polynomials is a \emph{Gr\"{o}bner
basis} of $I$ (with respect to the order $<$) if
\begin{enumerate}
\item $G$ generates $I$ as a
two-sided ideal, and
\item for every $f \in I$ there exists $g \in G$ such that $\LM(g)$ is a
    subword of $\LM(f)$, that is
$\LM(f) = a\LM(g)b$,  for some $a, b \in \asX$.
\end{enumerate}
A  Gr\"{o}bner basis $G$ of
$I$ is \emph{reduced} if  (i)  $G$ is a reduced set;  (ii) the set $G\setminus\{f\}$ is not a Gr\"{o}bner
basis of $I$, whenever $f \in G$.

It is known that every ideal $I$ of $\k  \asX$ has a {\em unique reduced
Gr\"{o}bner basis} $\cG=\cG(I)$ with respect to $<$, which may in general be countably
infinite.
 For more details, we refer the reader to \cite{Latyshev, Mo88, Mo94}. Its set of leading monomials
\begin{equation}
\label{eq:obstructions}
\textbf{W} =\{LM(f) \mid  f \in \cG\}
\end{equation}
is \emph{the set of obstructions} for $A= \k  \asX/ I$, in the sense of Anick\cite{Anickmon}.
There are equalities of sets  $\cN(I) = \cN(\cG) = \cN(\textbf{W}).$ Moreover, $\textbf{W}$ is the maximal antichain of words in the set $\asX \setminus \cN(I)$. Clearly, for every $f \in \k\asX$ the normal form $\Nor_I(f)$ of $f$ modulo $I$ coincides with $\Nor_{\textbf{W}}(f)$, the normal form of $f$ modulo $\textbf{W}$.

Note that if  $f\in \cG$ has a binomial form $f=u-v$ with $v<u$  and $u, v\in X^m$, then we can write $f=u- \Nor(u)$ since
by definition $u$ and $v$ are normal modulo  $\cG\setminus\{f\}$ so that in particular $v = \Nor(v)$. Hence, if we know that $\cG$ consists entirely of such binomials then we can write
\begin{equation}
\label{eq:Grbasibinomial}
\cG = \{w - \Nor(w) \mid w \in \textbf{W}\}
\end{equation}
as in a certain sense dual to (\ref{eq:obstructions}).

\begin{rmk}
\label{rmk:diamondlemma}  More generally, Bergman's Diamond lemma  \cite[Theorem 1.2]{Bergman} implies the following.
Let $G  \subset \k \asX$  be a set  of noncommutative polynomials. Let $I =
(G)$ and let $A = \k \asX/I.$ Then the following
conditions are equivalent.
\begin{enumerate}
\item
The set
$G$  is a Gr\"{o}bner basis of $I$.
\item Every element $f\in \k \asX$ has a unique normal form modulo $G$,
    denoted by $\Nor_G (f)$.
\item
There is an equality $\cN=\cN(G) = \cN(I)$, so there is an isomorphism of vector spaces
\[\k \asX \simeq I \oplus \k \cN.\]
\item The image of $\cN$ in $A$ is a $\k$-basis of $A$, we call it \emph{the normal $\k$-basis of $A$}.
In this case, one can define multiplication $\bullet$ on the $\k$-vector space $\k \cN$ as \[a\bullet b: =
\Nor(ab),\quad
\forall a,b \in \k \cN, \]
which gives the structure of a $\k$-algebra on $\k \cN(G)$  isomorphic to
$A$, see \cite[Section 6]{Bergman}. We shall often identify $A$ with the $\k$-algebra $(\k \cN(G), \bullet)$
 \end{enumerate}
In this case,  if $G$ consists entirely of homogeneous  binomials of arbitrary degrees,
\[G= \{u_i-v_i\mid \  u_i, v_i \in \asX,\  |u_i|=|v_i| \},\]
where $i$ runs over some set of indices, then $\Nor (w ) \in \cN$ for every $w \in \asX$ and $\bullet$ restricts to make $(\cN, \bullet)$ a graded monoid with $A$ isomorphic to the monoid algebra.

If the Gr\"{o}bner basis $G$ of $I$ is finite, then we say that $A$ is a {\em standard finitely presented (s.f.p.) algebra} with a standard finite presentation $A = \k \asX/(G).$
\end{rmk}

\subsection{Quadratic algebras}
\label{sec:Quadraticalgebras}
A quadratic  algebra is an associative graded algebra
 $A=\bigoplus_{i\ge 0}A_i$ over a ground field
 $\k $  determined by a vector space of generators $V = A_1$ and a
subspace of homogeneous quadratic relations $R= R(A) \subset V
\otimes V.$ We assume that $A$ is finitely generated, so $\dim A_1 <
\infty$. Thus, $ A=T(V)/( R)$ inherits its grading from the tensor
algebra $T(V)$.

In this paper, we consider finitely presented quadratic algebras   $A= \k  \langle X\rangle /(\Re)$,
where by convention $X$ is a fixed finite set of generators of
degree $1$, $|X|=n \geq 2,$ and $(\Re)$ is the two-sided
ideal of relations, generated by a {\em finite} set $\Re$ of
homogeneous polynomials of degree two. In particular, $A_1 = V= \k X$.
 \begin{dfn}
\label{def:PBW}
A quadratic algebra $A$ is
\emph{a  Poincar\'{e}-Birkhoff-Witt type algebra} or shortly
\emph{a PBW algebra} if there exists an enumeration $X=\{x_1,
\cdots, x_n\}$ of $X,$ such that the quadratic relations $\Re$ form a
(noncommutative) Gr\"{o}bner basis with respect to the
deg-lex order $<$ on $\asX$.
In this case, the set of normal monomials
(mod $\Re$) forms a $\k$-basis of $A$ called a \emph{PBW
basis}
 and $x_1,\cdots, x_n$ (taken exactly with this enumeration) are called \emph{
 PBW-generators of $A$}. 
\end{dfn}

These are precisely s.f.p. algebras where all the elements of G are quadratic. PBW algebras in general were introduced by Priddy, \cite{priddy} and form an important class of Koszul algebras.
 A PBW basis  is a generalization of the classical
 Poincar\'{e}-Birkhoff-Witt basis for the universal enveloping of a  finite
 dimensional Lie algebra. The interested reader can find information on quadratic algebras and, in
  particular, on Koszul algebras and PBW algebras in \cite{PoPo}.

\begin{lem}
\label{lem:pre}
Let $X=\{x_1,\cdots,x_n\}$, $n\ge 2$ and $A = \k\asX/I$ a quadratic algebra with $I= (I_2)$ and suppose that there exists a set $ \textbf{N}\subseteq X^2$ of order $|\bf N|=\dim A_2$ such that for every $xy \in \textbf{W}:=X^2\setminus{\bf N}$ there exist unique monomial $N(xy) \in \textbf{N}$, such that
\begin{equation}\label{lem:pre_assume} xy - N(xy) \in I_2,\quad xy > N(xy).\end{equation}
\begin{enumerate}
\item
$\Re= \{f_{xy}= xy -N(xy)\mid xy \in \textbf{W}\}$  is a $\k$-basis of the graded component $I_2$ of $I.$
\item
 $\Re$ is a set of defining relations for $A$, i.e.,  $(\Re) = I$ so that $A=\k\asX/(\Re)$.
 \item
 $\textbf{N}= \cN(\Re)_2= \cN(I)_2$ is the set of normal words of length $2$  modulo $I$ and  for all $xy\in X^2$,
\begin{equation*}
\Nor(xy) = \begin{cases} xy & \text{if}\  xy \in \textbf{N}\\  N(xy) & \text{if}\ xy \in \textbf{W}.\end{cases}\end{equation*}
\end{enumerate}
\end{lem}
\begin{proof}

(1)  By the assumption (\ref{lem:pre_assume}), we have that $\Re \subseteq I_2$.  Since  $\Re$   consists of pairwise distinct elements $f_{xy}$ for each $xy\in W$ and by the assumption on $|\bf N|$, we have $|\Re|=|\textbf{W}|= n^2 -|\mathbf{N}|= n^2-\dim A_2=\dim I_2$. Hence we need only show that the set $\Re$ is linearly independent to conclude the proof of (1).

Indeed, $\Re$ contains $n^2 - \dim A_2$ elements of $\k X^2$ all of which have pairwise distinct highest monomials.
Suppose we have a linear relation of the form
\begin{equation}
\label{eq:lincomb06}
0 = \sum _{f_{xy} \in \Re} c_{f_{xy}} f_{xy} = \sum_{xy \in \textbf{W}} c_{f_{xy}} (xy-N(xy));\quad c_{f_{xy}}  \in \k,
\end{equation}
($x, y$ are used to specify $f$). Let $U$ and $V$ be the subspaces of $\k X^2$ defined by
\[U= {\rm span}_\k \{N(xy)\mid xy \in \textbf{W}\}, \quad V = {\rm span}_\k \textbf{W}, \]
where $U \subseteq \k \textbf{N}$ so that $U \cap V  \subseteq \k \textbf{N} \cap \k \textbf{W} = \{0\}$.
 The relation (\ref{eq:lincomb06}) implies the equality of elements
\[V\ni\sum_{xy\in \textbf{W}} c_{f_{xy}} xy = \sum_{xy\in \textbf{W}}c_{f_{xy}} N(xy) \in U,\]
so that $\sum_{xy\in \textbf{W}} c_{f_{xy}} xy =0$.
The spanning set   $\textbf{W}$ of $V$  consists of distinct monomials of $\k \asX$, which are necessarily linearly independent.  Hence $c_{f_{xy}} = 0$ for all $f_{xy} \in \Re$. Hence $\Re$ is  linearly independent.

  (2)  $I=(I_2)$ by assumption, so by $\Re=I$ by part (1). Hence, $A$ has the finite presentation as stated.

 (3) The description of $\Re$ implies that a monomial $tz\in X^2$ is normal modulo $I$ \emph{iff} $tz \notin \textbf{W}$ or, equivalently, \emph{iff} $tz \in \textbf{N}$. Hence,  $\textbf{N}= \cN(\Re)_2= \cN(I)_2$. It follows from the definition of normal forms, see Section~\ref{sec:grobner},   and the form of  $\Re$ that given a monomial $xy \in X^2$, its normal form $\Nor(xy)$ (modulo $I$) is the minimal element $zt\in X^2$ such that $xy = zt$ modulo $I$, which implies (3).
  \end{proof}
If $\Re$ is a Gr\"obner basis w.r.t. deg-lex ordering then from the definition, $A$ will be PBW. For completeness, we understand this appropriately in the case $A=\k\langle X\rangle$ where $I=\{0\}, N=X^2$ and $\Re,\bf W$ are empty.

\subsection{Set-theoretic solutions of the Yang-Baxter equation and their Yang-Baxter algebras}

The notion of \emph{a quadratic set} was introduced in \cite{GI04}, see also \cite{GIM08},  as a
set-theoretic analogue of a quadratic algebra. Here we generalize it by not assuming that the map $r$ is bijective.
\begin{dfn}\cite{GI04}
\label{def:quadraticsets_All} Let $X$ be a nonempty set (possibly
infinite) and let $r: X\times X \longrightarrow X\times X$ be a
 map. In this case,  we refer to  $(X, r)$ as a \emph{quadratic set}. The image of $(x,y)$ under $r$ is
written as
\[
r(x,y)=({}^xy,x^{y}).
\]
This formula defines a ``left action'' $\cL: X\times X
\longrightarrow X,$ and a ``right action'' $\cR: X\times X
\longrightarrow X,$ on $X$ as: $\cL_x(y)={}^xy$, $\cR_y(x)=
x^{y}$, for all $x, y \in X$.

There are several cases of interest in the paper. (i) $(X, r)$ is \emph{left nondegenerate}, (respecively, \emph{right nondegenerate}) if
the map $\cL_x$ (respectively, $\cR_x$) is bijective for each $x\in X$.
$(X,r)$ is nondegenerate if both maps $\cL_x$ and $\cR_x$ are bijective.
(ii) $(X, r)$ is \emph{involutive} if $r^2 = \id_{X\times X}$.
(iii) $(X, r)$ is \emph{idempotent} if $r^2 = r$.
(iv) $(X, r)$ is \emph{a set-theoretic
solution of the Yang--Baxter equation} (YBE) if  the braid
relation
\[r^{12}r^{23}r^{12} = r^{23}r^{12}r^{23}\]
holds in $X\times X\times X,$  where  $r^{12} = r\times\id_X$, and
$r^{23}=\id_X\times r$. In this case, we also refer to  $(X,r)$ as
\emph{a braided set}.
\end{dfn}

\begin{rmk}
    \label{rmk:YBE1}
Let $(X,r)$ be quadratic set.
Then $r$ obeys the YBE, i.e., $(X,r)$ is a braided set {\em iff}
\[
\begin{array}{cccc}
 {\bf l1:}&{}^x{({}^yz)}={}^{{}^xy}{({}^{x^y}{z})},
 \quad\quad
 {\bf r1:}&
{(x^y)}^z=(x^{{}^yz})^{y^z},
 \quad \quad{\rm\bf lr3:}&
{({}^xy)}^{({}^{x^y}{z})} \ = \ {}^{(x^{{}^yz})}{(y^z)}
\end{array}\]
for all $x,y,z \in X$: The map $r$ is idempotent,   $r^2=r$,
\emph{iff}
\[
\textbf{pr:} \quad {}^{{}^xy}{(x^y)}={}^xy,\quad ({}^xy)^{x^y}= x^y, \quad \forall x, y \in X.
\]
\end{rmk}

\begin{convention}
\label{conv:convention1} As a notational tool, we  shall  identify the sets $X^{\times
m}$ of ordered $m$-tuples, $m \geq 2,$  and $X^m,$ the set of all
monomials of length $m$ in the free monoid $\asX$.
Sometimes for simplicity we
 shall write $r(xy)$ instead of $r(x,y)$. \end{convention}

\begin{dfn} \cite{GI04, GIM08}
\label{def:algobjects} To each finite quadratic set $(X,r)$ we associate
canonically algebraic objects generated by $X$ with quadratic
relations $\Re =\Re(r)$ naturally determined as
\[
xy=y^{\prime} x^{\prime}\in \Re(r)\quad \text{iff}\quad
 r(x,y) = (y^{\prime}, x^{\prime})\ \&\
 (x,y) \neq (y^{\prime}, x^{\prime})
\]
 as equalities in $X\times X$. The monoid
$S =S(X, r) = \langle X ; \; \Re(r) \rangle$
 with a set of generators $X$ and a set of defining relations $ \Re(r)$ is
called \emph{the monoid associated with $(X, r)$}. For an arbitrary fixed field $\k$,
\emph{the} $\k$-\emph{algebra associated with} $(X ,r)$ is
defined as
\[\begin{array}{c}
 \cA(\k ,X,r) = \k \langle X  \rangle /(\Re_{\cA})
\simeq \k \langle X ; \;\Re(r)
\rangle;\quad
\Re_{\cA} = \{xy-y^{\prime}x^{\prime}\mid \  xy=y^{\prime}x^{\prime}\in \Re
(r)\}.
\end{array}
\]
Usually, we shall fix an enumeration $X= \{x_1, \cdots, x_n\}$ and extend it to the degree-lexicographic order $<$ on
$\asX$.
In this case we require the relations of $\cA$ to be written as
\begin{equation}\label{eq:Algebra}
\Re_{\cA} = \{xy-y^{\prime}x^{\prime}\mid \  xy>y^{\prime}x^{\prime}\ \&\  r(xy)= y^{\prime}x^{\prime}\  \text{or}\;
r(y^{\prime}x^{\prime})= xy\}. \end{equation}
Clearly, $\cA(\k,X,r)$ is a quadratic algebra generated by $X$
 with defining relations  $\Re_{\cA}$, and is
isomorphic to the monoid algebra $\k S(X, r)$.  When $(X,r)$ is a solution of YBE, we defer to
$\cA(\k,X,r)$ is the associated \emph{Yang-Baxter algebra} (as in \cite{Ma88} for the linear case) or \emph{YB algebra}
for short,
and to $S(X, r)$ as the associated Yang-Baxter monoid.
\end{dfn}

If $(X,r)$ is a finite quadratic set then $\cA(\k ,X,r)$ is \emph{a connected
graded} $\k$-algebra (naturally graded by length),
 $\cA=\bigoplus_{i\ge0}\cA_i$, where
$\cA_0=\k ,$
and each graded component $\cA_i$ is finite dimensional.
Moreover, the associated monoid $S(X,r)$ \emph{is naturally graded by
length}:
\begin{equation}
\label{eq:Sgraded}
S = \bigsqcup_{i \geq 0}  S_{i};\quad
S_0 = 1,\; S_1 = X,\; S_i = \{u \in S \mid\;   |u|= i \}, \; S_i.S_j \subseteq
S_{i+j}.
\end{equation}
In the sequel, by `\emph{a graded monoid} $S$', we shall
mean that $S$ is a monoid generated  by $S_1=X$ and graded by length.
The grading of $S$ induces a canonical grading of its monoid algebra
$\k S(X, r).$
The isomorphism $\cA\cong \k S(X, r)$ agrees with the canonical gradings,
so there is an isomorphism of vector spaces $\cA_m \cong \k S_m$.

\begin{dfn}
\label{orbitsinG}  Let $(X,r)$ be a finite quadratic set. We define an equivalence relation on $X^2$ by
 \[xy \sim zt\; \text{if and only if}\; r^k(xy) = r^m(zt) \; \text{for some}\; m, k\geq 0.\]
 We define an $r$-orbit $\cO$ as an equivalence class with respect to this. Here $r^0$ denotes the identity map.
\end{dfn}

Clearly $X^2$ is a  disjoint union of $r$-orbits and we denote by $O_{xy}$ the $r$-orbit containing $xy\in X^2$.  Also note that  the equality $xy=zt$ holds in $S=S(X,r)$ if and only if $xy\sim zt$, so each $r$-orbit corresponds to a distinct element of $S$ of grade 2.

\begin{rmk}  $\sim$ can be extended in a natural way to an equivalence relation on the free monoid $\asX$, such that
 $\sim$  is a congruence on $\asX$ which (by definition) satisfies:
 \[\text{if} \; u, v, a, b \in \asX, \text{ and}\;  u \sim v, \;\text{then}\;  aub \sim avb.\]
 It then follows from the definition that $S$ is isomorphic (as a monoid) to the quotient monoid modulo the congruence,
 $\asX /\sim$.
  In particular two words $w_1, w_2 \in \asX$ are equal in $S$ \emph{iff} $w_1$, and $w_2$ belong to the same congruence
  class in $\asX$.

 There is a practical way given $w_1, w_2 \in \asX$, how to determine if $w_1 = w_2$ in $S$. Namely,
 $w_1, w_2 \in \asX$ are equal
 in $S$
  if they have equal lengths  $|w_1| = |w_2|\geq 2$ and there exists a monomial $w_0$ such that each  $w_i, i=1,2$ can
  be
  transformed to
  $w_0$ by a finite sequence of
  replacements (they are also called \emph{reductions} in the literature) each of the form
  \[a(xy)b \longrightarrow a(zt)b,\]
where $xy=zt$ is an equality in $S$, $xy>zt$  in $X^2$ and $a,b \in \asX$.  Clearly, every such replacement preserves monomial length, which therefore descends to $S(X,r)$. Furthermore, replacements coming from the defining relations are possible only
on monomials of length $\geq 2$, hence $X \subset S(X,r)$ is an inclusion.
\end{rmk}

\subsection{Special set of relations for the YB algebra $\cA (\k, X, r)$ of a finite braided set $(X,r)$ }

The canonical Definition~\ref{def:algobjects} is via a natural set of defining relations $\Re_{\cA}$ which does not depend on any enumeration of $X$ but may not be convenient for explicit computations in $\cA$. In this section we introduce a new presentation of  $\cA(\k, X, r)$ for any finite braided set $(X,r)$ with a possibly different set of relations $\Re$ depending on the enumeration, generating the same two sided ideal of relations $I =(\Re_{\cA})=(\Re)$ and with better properties such as $\Re\subseteq \cG(I)$, the reduced Gr\"{o}bner basis of $I$. We will later use this for example to find $\k$-bases of the graded components $\cA_2$ and $I_2$.

We fix an enumeration $X= \{x_1, \cdots, x_n\}$ on $X$ and extend it to deg-lex ordering $<$ on the monoid $\asX$.
Then each $r$-orbit $\cO$ has a unique minimal element $ab\in \cO$. In fact $\cO$ is uniquely determined
by its minimal element, and each $xy\in \cO, xy \neq ab$ satisfies $xy>ab$.  Enumerate these minimal elements lexicographically
\[a_1b_1 < a_2b_2 < \cdots < a_Mb_M,\]
where $M$ is the number of orbits and enumerate the orbits $\cO_1, \cdots \cO_M$ so that  $a_ib_i \in \cO_i$. Clearly,
 \[1 \leq |\cO_i| \leq n^2,\quad  \sum_{1 \leq i \leq M} |\cO_i| = n^2.\]

\begin{pro}
\label{pro:main1}
 Suppose $(X,r)$ is a finite braided set, $X= \{x_1, \cdots, x_n\}$ an enumeration with $n\ge 2$ and $a_ib_i\in \cO_i$ the minimal elements of the enumerated $r$-orbits as above.
 \begin{enumerate}
\item[(1)]
The YB algebra $\cA=\cA(\k,X,r)$ has a finite presentation
\begin{equation}
\label{eq:rels11pro}
\cA= \k \asX /(\Re);
\quad
\Re = \{xy - a_ib_i\ |\  1 \leq i \leq M,\quad  xy \in \cO_i\setminus \{a_ib_i\}\}
\end{equation}
where $\Re$ consists  of precisely $n^2- M$ linearly independent
relations as shown such that
\begin{enumerate}
\item[(i)] every monomial  $xy\in X^2\setminus \{a_ib_i\mid 1\leq i \leq M\}$
occurs exactly once in $\Re$ and is the highest monomial of the uniquely determined element $xy-a_ib_i \in \Re$;
\item[(ii)]  for every integer $i, 1 \leq i \leq M$,  there are exactly $|\cO_i|-1$ distinct polynomials
of the shape $xy- a_ib_i \in \Re$.
\end{enumerate}

\item[(2)]  The two-sided ideal $I=(\Re)$ of $\k \asX$ generated by $\Re$ has $\Re$ as a basis in degree 2, $I_2=\k\Re$. Moreover, $\dim \cA_2=M$ and the set
\[\cN (I)_2 = \cN(\Re)_2 =\{a_1b_1 < a_2b_2 < \cdots < a_Mb_M\},\]
is a (normal)  $\k$-basis of $\cA_2$.
\item[(3)] If $\Re$ is a Gr\"{obner basis} of $I$ w.r.t. deg-lex ordering on $\asX$ then $\cA$ is a PBW algebra with PBW generators $x_1, \cdots, x_n$ and $\Re=\cG(I)$. In this case the set of normal words
$\cN(I) = \cN(\Re)$
is the normal $\k$-basis of $\cA$.
\end{enumerate}
 \end{pro}
\begin{proof}
The Yang-Baxter algebra in Definition~\ref{def:algobjects} amounts in our case to relations (\ref{eq:Algebra}). We set ${\bf N}=\{a_ib_i\ |\ 1\le i\le M\}$ for $M=\dim A_2$ and for every $xy\in \cO_i\setminus\{a_ib_i\}$ we assign $N(xy)=a_ib_i$. Moreover, $xy - N(xy)=xy-a_ib_i\in \Re$ is identically zero in $\cA$,
since both $xy, a_ib_i$ are in the same orbit $\cO_i$ so $xy= a_ib_i$ holds in the monoid $S = S(X,r)$, and therefore in the algebra $\cA = \k S$. Hence $xy-N(xy)\in I_2$. The set $\bf W=X^2\setminus\bf N$ is the union of such $xy$ for each $1\le i\le n$ and each $a_ib_i$ is minimal in $\cO_i$, so (\ref{lem:pre_assume}) holds.

(1) We can now apply Lemma~\ref{lem:pre} with the same  $|{\bf W}|=n^2-M$ elements $\Re$ as stated in (1), where each $xy\in X^2 \setminus \{a_ib_i \mid 1 \leq i \leq M\}$ occurs exactly once in $\Re$, since it is an element of the unique orbit $\cO_i$. Part (1)(i) is clear from the description of $\bf W$ and for part (ii), each orbit $\cO_i$ of order  $|\cO_i|\geq 2$ contributes exactly $|\cO_i|-1$ distinct elements of $\bf W$ and of $\Re$. An $r$-orbit of order $|\cO_j|= 1$ contains unique element $a_jb_j$ (fixed under $r$) and does not imply a nontrivial relation in $\Re$, here $|\cO_j|-1=0$.

(2) Follows immediately from Lemma~\ref{lem:pre}(2)(3).

(3) By Definition~\ref{def:PBW}, if $\Re$ is a Gr\"{o}bner basis of $I= (\Re)$ w.r.t. deg-lex ordering on $\asX$ then  $\cA$ is a PBW algebra with PBW generators $x_1, \cdots, x_n$. We show that $\Re$ in this case meets the requirements for a reduced Gr\"{o}bner basis of $I$ as defined in Section~\ref{sec:grobner}.
(i) The description of $\Re$ implies that each polynomial $f \in \Re$ is in normal form modulo $\Re\setminus \{f\}$,
since $f = xy - a_jb_j$, and both $xy$ and $a_jb_j$ are not `divisible' by any highest monomial $HM(g)$ of some $g \in \Re\setminus \{f\}$.
(ii) Suppose $f\in \Re$ as in (\ref{eq:rels11pro}), so $f = xy - a_jb_j$, for some $j, 1 \leq j \leq M$, where $xy > a_jb_j$ and $HM (f)= xy$. Clearly, $f\in I$, but since the monomial $xy$
occurs exactly once in $\Re$, $HM(f) = xy$ is not `divisible' by the remaining highest monomials in $\{HM(g)\mid g \in \Re\setminus \{f\}\}$, and therefore $\Re \setminus \{f\}$ is not a Gr\"{o}bner
basis of $I$.  \end{proof}

We understand this appropriately for the identity solution $r={\rm id}_{X\times X}$ where $A=\k\langle X\rangle$ with the convention that $\cG(\{0\})$ is the empty set.

\begin{rmk}
\label{rmk:main2}
We note that  $\Re=\Re_{\cA}$ \emph{iff}  $r$ is involutive, $r^2=\id$. But in the noninvolutive case, the  proposition gives a potentially new set of relations $\Re$ which depends on the enumeration but where now $\Re\subseteq \cG=\cG(I)$, the reduced Gr\"{o}bner basis of $I$, since $\Re$ is a reduced set of generators of $I$ (in contrast to  $\Re_{\cA}$ which is not necessarily contained in $\cG$). Moreover, since  $\Re$ generates $I$ and consists of binomials, one can show that  $\cG$ itself consists of homogeneous binomials $f_j= u_j-v_j,$ where
 $\LM(f_j)= u_j > v_j$, and $u_j, v_j \in X^m$, for some $m\geq 2$. Hence, were are in the setting where (\ref{eq:Grbasibinomial}) applies. In this case the homogeneous elements of $\cG$  of degree $\leq d$ can be found inductively  (using a standard strategy for constructing a Gr\"{o}bner
basis) and consequently the set of normal monomials $\cN_{d}(I)$  can be found
inductively for all $d\ge 1$. Moreover, as explained after Remark~\ref{rmk:diamondlemma}, the  binomial form of the elements of $\cG$  in our case implies that $\bullet$ restricts to a monoid $(\cN,\bullet)$ which we can identify with  $S=S(X,r)$.
 \end{rmk}

\section{Left nondegenerate idempotent solutions and their Yang-Baxter algebras}
\label{sec:idemp}
\begin{def}
\label{def:graph1}
Let $(X, r)$ be a finite quadratic set. We define an associated finite directed graph $\mathfrak{G}= \mathfrak{G}(X,r)$
by vertex set $V(\mathfrak{G})= X^2$ and the set of directed edges (arrows) $E(\mathfrak{G})$ according to
\[xy \longrightarrow zt \; \text{iff}\;   r(xy) = zt.\]
\end{def}
Our directed graphs include the possibility of a self-arrow. It is clear that there is a bijective map
\[ r-\text{orbits in}\;  X^2  \Longleftrightarrow \text{connected components of }\;
\mathfrak{G}, \]
where we recall that an $r$-orbit is an equivalence class as in Definition~\ref{orbitsinG}. For each $xy\in X^2$, the $r$-orbit $\cO_{xy}$ corresponds to the unique connected component  $ \mathfrak{G}_{xy}$ which contains the vertex $xy$. Henceforth (by a slight abuse of notation), we will identify $r$-orbits with connected components of $\mathfrak G$ and vice versa. In particular,  $xy = zt$ holds in $S$ if the two words $xy, zt\in X^2$ are in the same connected component.   We will also be interested in fixed points under $r$, i.e. $xy\in X^2$ such that $r(xy)=xy$, which in the directed graph $\mathfrak{G}$ means vertices $xy \circlearrowleft$ with a self-arrow.  Our first topic is to study the specific properties of $r$-orbits when the quadratic set $(X,r)$ is {\em idempotent} in the sense $r^2=r$.

\begin{lem}
\label{lem:idemporb}
Suppose $(X,r)$ is an idempotent quadratic set, i.e. $r^2=r$. Then

(1) Every $r$-orbit has a unique fixed point and every arrow in the $r$-orbit points to it.

(2) $xy,zt\in X^2$ lie in the same $r$-orbit  (equivalently $xy=zt$ holds in $S(X,r)$) if and only if $r(xy)=r(zt)$ in $X^2$.
\end{lem}
\begin{proof}
(1) Since $r^2=r$, all positive powers $r^p, p\geq 2,$ can be reduced to $r^p=r$  and therefore distinct elements of $X^2$ are in the same $r$-orbit if and only if there is an arrow from one to the other or if they both point to a common vertex. Moreover, if $\cO$ is an $r$-orbit and $xy\in \cO$ then $r(xy)\in \cO$ is a fixed point since $r^2=r$, so every orbit has a fixed point. If $zt,z't'\in\cO$ are two fixed point then either $zt=r(zt)=z't'$ or $z't'=r(z't')=zt$ or $zt=r(zt)=r(z't')=z't'$, so $\cO$ has exactly one fixed point.  Any other element $vw\in \cO$ must be related to $zt$ by $r(vw)=zt$ or by $r(zt)=vw$, which is not possible as $r(zt)=zt$, or by $zt=r(zt)=r(vw)$. So all arrows point to the fixed point.

(2) Let $xy,zt \in X^2$. If $xy = zt$ holds in $S$ then $xy, zt$ are in the same $r$-orbit. If neither is the unique fixed point of the orbit then they each point to it, hence $r(xy)=r(zt)$. If both are the unique fixed point then $r(xy)=xy=zt=r(zt)$. If $xy$ is the unique fixed point and $zt$ is not then $r(zt)=xy=r(xy)$, similarly if $zt$ is the unique fixed point and $xy$ is not.  Conversely, if $r(xy)=r(zt)$ holds in $X^2$ then both $xy, zt$ point to the same vertex, hence are in the same $r$-orbit, hence  $xy=zt$ holds in $S$.
\end{proof}

We shall need the following lemma.

\begin{lem}
\label{lem:orbitsgeneral}
Suppose $(X,r)$ is a left nondegenerate quadratic set.
\begin{enumerate}
\item
\label{pro:orbits1}
For $x,y,z \in X$, if  $r(xy) = r(xz)$ holds in $X^2$ then $y = z$  (i.e., the map $r$ is left $2$-cancellative).

\noindent Suppose in addition that $r$ is idempotent. Then the following further hold
\item
\label{pro:orbits2} For $x, y, z \in X$, if  $xy = xz$ holds in $S(X,r)$ then $y = z$ (i.e., the monoid $S(X,r)$ is left cancellative on monomials of length $2$).

\item
\label{pro:orbits3}   $xy\in X^2$ is a fixed point if and only if ${}^xy = x$.
Moreover, for every $x\in X$ there exists unique $y\in X$, such that
$xy$ is a fixed point.

\item
\label{pro:orbits4}
There are exactly $n=|X|$ $r$-orbits in $X^2$ (each with one fixed point).
\end{enumerate}
\end{lem}
\begin{proof}
(\ref{pro:orbits1})
Suppose $r(xy) = r(xz)$ holds in $X^2$. Then the following are equalities of words in $X^2$:
\[r(xy) = ({}^xy)(x^y) = ({}^xz)(x^z) = r(xz). \]
Therefore ${}^xy = {}^xz$, and since $r$ is left nondegenerate, $z=y.$ In particular, the map $r$ is $2$-cancellative on
the left.

(\ref{pro:orbits2})  For the rest of the proof we  assume that $(X,r)$ is left nondegenerate and additionally that $r^2=r$. Assume that $xy=xz$ holds in $S$ for some $x,y,z \in X$.
Then by Lemma ~\ref{lem:idemporb}
 \[({}^xy)(x^y) = r(xy) = r(xz) = ({}^xz)(x^z)\]
holds in $X^2$. It follows that ${}^xy = {}^xz$, and then left nondegeneracy of $r$ implies $y = z.$

(\ref{pro:orbits3}) Clearly, if $xy$ is a fixed point, then $xy = r(xy) =({}^xy)(x^y)$ is an equality of words in $X^2$, hence    ${}^xy =
x$. Conversely, assume that  ${}^xy = x$. Then
\[r(xy) = ({}^xy)(x^y)= xx^y\quad \text{holds in}\; X^2.\]
Therefore $xy = xx^y$ is an equality in $S$, and part (\ref{pro:orbits2}) implies that $y=x^y$. Hence $r(xy)=xy$. Also, if $x\in X$ then by left nondegeneracy there exists a unique  $y \in X$ such that ${}^{x}{y} = x$ and hence a unique $y$ such that $xy$ is a fixed point by the previous step.

 (\ref{pro:orbits4}) The  fixed points as we vary $x$ in part (3) are distinct, since $r(xy)=({}^xy)(x^y)=x x^y$ as an equality in $X^2$ (so if $x'y'$ is another such fixed point then $x'=x$). This construction generates $n$ fixed points and again by (3), all fixed points are of this form. By Lemma~\ref{lem:idemporb}(1), there is exactly one fixed point per $r$-orbit, so there are exactly $n$ $r$-orbits also.  \end{proof}

\begin{lem}
\label{lem:orbits}
Suppose $(X,r)$ is a left nondegenerate idempotent quadratic set, where $X = \{x_1, \cdots, x_n\}$.
\begin{enumerate}
\item
\label{pro:orbits6}
For every $k, 1 \leq k \leq n,$ we let $\cO_k$ denote the $r$-orbit containing $x_1x_k$. Then $X^2$ is a disjoint union
\begin{equation}
\label{eq:orbits}
X^2 = \bigcup_{k=1}^n \cO_k.
\end{equation}

\item
\label{pro:orbits7} Every orbit $\cO_k$ contains exactly $n$ distinct elements of $X^2$.

\item
\label{pro:orbits8} For every pair $i,j$ with $2\leq i \leq n, 1 \leq j \leq n,$ there exists a unique integer $k_{ij}\in \{1, \cdots, n\}$  such that $x_ix_j \in \cO_{k_{ij}}$, or, equivalently such that $x_ix_j = x_1x_{k_{ij}}$ holds in $S(X,r)$. Here, $k_{ij}$ is characterised by ${}^{x_1}x_{k_{ij}}={}^{x_i}{x_j}$.
 \end{enumerate}
\end{lem}
\begin{proof}
(\ref{pro:orbits6}) We have to show that if $k\ne l$ then $\cO_k$ and $\cO_l$ are distinct orbits. If not, they must coincide, so $x_1x_k, x_1x_l\in \cO_k$. By Lemma~\ref{lem:idemporb}(2) this implies that  $r(x_1x_k)=r(x_1x_l)$, which contradicts left $r$-cancellation.

 (\ref{pro:orbits7})
Note first that each orbit $\cO_k$ contains at most $n$ distinct elements $xy\in X^2$, otherwise
$\cO_k$ would contain at least two words of the shape $xy$, $xz$, $y\neq z$, and therefore
$xy= xz$ in $S$,
 which contradicts Lemma~\ref{lem:orbitsgeneral} (\ref{pro:orbits2}).
 The equality  (\ref{eq:orbits}) then implies that
\[n^2 = |X^2| = \sum_{1\leq k \leq n} |\cO_k| \leq n^2, \quad\text{since}\; |\cO_k|\leq n, \;\text{for}\; 1 \leq k \leq
n.\]
It follows that each orbit $\cO_k$ contains exactly $n$ distinct elements of $X^2$ (including $x_1x_k$).

(\ref{pro:orbits8}) It is immediate that for every pair $i, j, \; 2\leq i\leq n, \; 1 \leq j \leq n,$ there exists a unique orbit $\cO_k\owns x_ix_j$, where $k$ depends on $i, j$. By Lemma~\ref{lem:idemporb}(2) this means a unique $k$ such that $r(x_ix_j)=r(x_1x_k)$ or equivalently such that $x_ix_j=x_1x_k$ in $S(X,r)$ as stated. Next observe that $r(x_ix_j) = {}^{x_i}x_jx_i^{x_j}$ and
 $r(x_1x_k) = {}^{x_1}x_k{x_1}^{x_k}$ in $X^2$, hence for these to be equal we need ${}^{x_i}x_j={}^{x_1}x_k$. Conversely, if this condition holds then  $r(x_ix_j)=xy$ and $r(x_1x_k)=xy'$ are fixed points with the same $x$ for some $y,y'$, but then by the uniqueness in Lemma~\ref{lem:idemporb}(3) $y=y'$ so that $r(x_ix_j)=r(x_1x_k)$. The condition stated uniquely determines $k$ by left nondegeneracy. \end{proof}

\begin{rmk}
\label{rmk:Scancellative} In fact, the left cancellation mentioned in the course of Lemma~\ref{lem:orbitsgeneral} holds in all degrees in
$S(X,r)$ when $(X,r)$ is a finite left-nondegenerate idempotent braided set. This is because in this case,  by \cite[p. 450]{cedo21}, $S(X,r)$ is a regular submonoid of the semidirect product $A(X,r)\rtimes \mathfrak{g}(X,r)$, where $A(X,r)$ is the derived monoid of $(X,r)$ (see the definition in \cite{cedo21}) and $\mathfrak{g}(X,r)= gr\langle \cL_x\mid x \in X \rangle$ is the
the  permutation group generated by the left actions. If $(X,r)$ is left-nondegenerate and idempotent,
then the derived monoid $A(X,r)$ is left cancellative, see \cite[p. 5466]{Colazzo23} and hence
$A(X,r)\rtimes \mathfrak{g}(X,r)$ is left cancellative. It follows that its submonoid $S(X,r)$ is also left cancellative.
\end{rmk}

We are now ready to study the YB algebra $\cA(\k,X,r)=\k S(X,r)$ for left-nondegenerate idempotent braided sets.

 \begin{thm}
\label{thm:main1}
Suppose $(X,r)$ is a left-nondegenerate idempotent solution of the YBE on
$X= \{x_1, \cdots, x_n\}$, and let $\cA= \cA(\k, X, r)$ be its YB-algebra.
\begin{enumerate}
\item[(1)]
The algebra $\cA$ is a PBW algebra with PBW generators $x_1, \cdots, x_n$ and a standard finite presentation
\begin{equation}
\label{eq:rels1}
\cA= \k \asX /(\Re);
\quad
\Re = \{f_{ij}= x_ix_j - x_1x_{k_{ij}}\ |\  2 \leq i \leq n, 1 \leq j \leq n\}
\end{equation}
where $\Re$ consists  of precisely $n(n-1)$ linearly independent
relations as shown such that
\begin{enumerate}
\item[(i)] every monomial  $x_ix_j$  with $2 \leq i \leq n, 1 \leq j \leq n$  occurs exactly once in $\Re$;
\item[(ii)]  for every integer $k, 1 \leq k \leq n$,  there are exactly $n-1$ distinct polynomials
$f_{ij} \in \Re$ with $k_{ij}=k$.
\item[(iii)] $\Re$ is the reduced Gr\"{o}bner basis of the ideal $I=(\Re)$ w.r.t. deg-lex ordering on $\asX$.
\end{enumerate}

\item[(2)]
\begin{equation}
\label{eq:normalbasis}
\cN = \cN(\Re) =\{1\} \cup \{x_1^{d-1}x_p \mid d \geq 1, 1  \leq p \leq n  \}
\end{equation}
is the normal $\k$-basis of $\cA$;
\item[(3)] $\dim \cA_d = n$  for all $d \geq 1$ resulting in the Hilbert series
\begin{equation}\label{Hidemp} H_\cA (t) = \frac{1 + (n-1)t}{1 - t}.\end{equation}
\end{enumerate}
\end{thm}
\begin{proof}
\textbf{(1)}
  Lemmas ~\ref{lem:orbitsgeneral} and
Lemma~\ref{lem:orbits}
imply
that $X^2$ splits into $n$ disjoint  $r$-orbits $\cO_k, 1 \leq k \leq n$, each of which is determined uniquely by its minimal element
$x_1x_k\in\cO_k$ and contains exactly $n$ distinct elements:
\[\cO_k=\{x_1x_k\}\cup\{x_ix_j \mid\   2 \leq i\leq n, 1 \leq j \leq n,\  {}^{x_i}{x_j} ={}^{x_1}{x_k}\}.\]
Hence, the set of quadratic polynomials  $\Re$ in  Proposition~\ref{pro:main1} has exactly the prescribed form given in
(\ref{eq:rels1}) with $M=n$ and (i),(ii) there reduce to conditions (i) and (ii) as stated. In particular, $\Re$ consists of $n(n-1)$ relations and the Yang-Baxter algebra in Definition~\ref{def:algobjects}, which amounts in our case to relations
\begin{equation}
\label{eq:Algebra1}
\Re_{\cA}
 = \{xy-y^{\prime}x^{\prime}\mid  \; r(xy)=y^{\prime}x^{\prime} \neq xy\},
\end{equation}
now has the finite presentation (\ref{eq:rels1}).

For part (iii), we note  that any new relation induced by a nonsolvable ambiguity coming from $\Re$
must have the form $F = u - v \in I$, where $u$ and $v$ are monomials of length 3, each of which is  normal mod $\Re$, therefore
each of which must belong to $\cN_3$. In other words, a possible new relation must have the form
$x_1x_1 x_p  - x_1x_1x_q \in I$ for $p > q$ say, where $1\le p,q\le n$. Then  $x_1x_1 x_p  = x_1x_1x_q$ holds in $\cA$ and hence holds in $S$. This is impossible, since $S$  is left cancellative by Remark~\ref{rmk:Scancellative}. One can also see this directly in the required degree 3 case by viewing $x_1x_1 x_p =x_1x_1x_q$ as 36 possible identities in $X^3$, where on each side we apply all possible braids built from $r$ (there are only 6 due to $r$ being idempotent). One can then show that  $p=q$ in each case.  Having shown that all ambiguities of overlap induced by $\Re$ are solvable, this  implies that $\Re$ is a Gr\"{o}bner basis of the ideal $I= (\Re)$ w.r.t. deg-lex ordering on $\asX$. Now Proposition~\ref{pro:main1} (3) implies that $\Re$ is the reduced Gr\"{o}bner basis of $I$.

\textbf{(2)} This follows from Proposition~\ref{pro:main1}(3), but is also clear from the form of the highest monomials of $\Re$. A word in $w\in \asX$ is normal modulo $\Re$ if and only if $w$ does not contain as a subword any
$x_ix_j$ with $i \geq 2$, $1 \leq i,j \leq n$, which implies $\cN$ as stated.
Moreover, since $\Re$ is a Gr\"{o}bner basis of $I$ it follows from the Diamond Lemma that $\cN(\Re)$ is a $\k$-basis of $\cA$.
In particular, for every integer $d \geq 1$ the set
\begin{equation}
\label{eq:N}
\cN_d = \{x_1^{d-1}x_k \mid  1  \leq k \leq n  \}
\end{equation}
forms a $\k$-basis of the graded component $\cA_d$ so that
$\dim \cA_d = n$ , for all $d \geq 1$.
It follows that the Hilbert series is $H_\cA (t) = 1 +nt + nt^2 + \cdots$, which sums as stated in (3).
\end{proof}

\begin{ex}
\label{ex1}
Let  $X=\{x_1, x_2, x_3\}$ and $r(x_ix_j)= f(x_j)x_j$ for all $1 \leq i, j \leq 3$, where $f =(x_1\;x_2\; x_3)$ is a 3-cycle
(this will be an example of a class of left-nondegenerate `permutation' idempotent solutions studied in Section \ref{sec:perm}). Explicitly,
\[
\begin{array}{lll}
r(x_3x_1) = x_2x_1, & r(x_1x_1) = x_2x_1, & r(x_2x_1)=x_2x_1, \\
r(x_2x_2) = x_3x_2, & r(x_1x_2) = x_3x_2, & r(x_3x_2)=x_3x_2, \\
r(x_3x_3) = x_1x_3, & r(x_2x_3) = x_1x_3, & r(x_1x_3)=x_1x_3. \\
\end{array}\]

The  associated Yang-Baxter algebra  $\cA= \k \langle x_1, x_2, x_3  \rangle /(\Re_{\cA})$ has the set of  6 relations
\begin{align*}
 \Re_{\cA} = \{f_1& = x_3x_1- x_2x_1, \ f_2= x_2x_1- x_1x_1,\  f_3= x_3x_2-x_2x_2, \\
 f_4&= x_3x_2-x_1x_2, \
 f_5= x_3x_3 - x_1x_3,\
f_6 = x_2x_3 -x_1x_3\},\end{align*}
which is not reduced.
For example, $f_{4}$ can be reduced using $f_3$,
so we can take instead $f_{22}=f_4 - f_3= x_2x_2 -x_1x_2 \in (\Re_\cA)$,
next reduce $f_3$ using $f_{32}$ to obtain $f_{32}=f_3 + f_{22}= x_3x_2 -x_1x_2\in (\Re_\cA)$.
Similarly, the sum $f_{31}= f_1+f_2 = x_3x_1-x_1x_1 \in (\Re_\cA)$.
Now we have a new set of relations
\begin{align*}
\Re = \{f_{31}&= x_3x_1- x_1x_1,\  f_{21}= x_2x_1- x_1x_1,\\
  f_{32}&= x_3x_2-x_1x_2,\  f_{22}=x_2x_2-x_1x_2,\\
 f_{33}&= x_3x_3 - x_1x_3,\  f_{23}=x_2x_3 -x_1x_3\}
\end{align*}
as the $n(n-1)$ monomials  in Theorem~\ref{thm:main1} organised according to the three $r$-orbits as shown in Figure~\ref{graphex12}(a). It is easy to check that there is an equality of ideals $(\Re_{\cA})= (\Re)$. Moreover, the set $\Re$ is the unique reduced Gr\"{o}bner basis of the ideal $I=(\Re)$,  as required by Theorem~\ref{thm:main1}. It is indeed known from \cite{Colazzo22},  that the algebra $\cA$ for examples of this particular type are PBW, while the reduced Gr\"{o}bner basis of the ideal $I$ given here is the new feature.
 \end{ex}

\begin{figure}
\[ \includegraphics[scale=1.3]{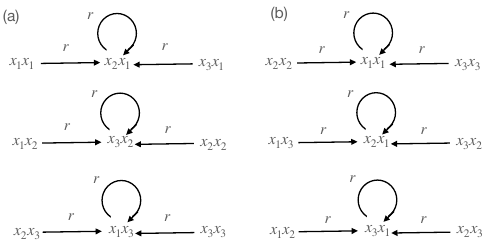}\]
\caption{Graph $\mathfrak G$ and $r$-orbits for Example~\ref{ex1} and Example~\ref{ex2} respectively.\label{graphex12}}
\end{figure}

We also give a second  example of a left nondegenerate idempotent solution $(X,r)$ of
order $3$ whose algebra $\cA$ is PBW and has additional ``good" properties.

\begin{ex}
\label{ex2} Let $(X, r)$ be the solution on the set $X=\{x_1, x_2, x_3\}$
defined by
\[
\begin{array}{lll}
r(x_3x_3) = x_1x_1, & r(x_2x_2) = x_1x_1, & r(x_1x_1)=x_1x_1, \\
r(x_3x_2) = x_2x_1, & r(x_1x_3) = x_2x_1, & r(x_2x_1)=x_2x_1, \\
r(x_2x_3) = x_3x_1, & r(x_1x_2) = x_3x_1, & r(x_3x_1)=x_3x_1. \\
\end{array}\]
We leave the reader to check that $r$ is a solution of YBE.
It is clear that $r^2=r$, and left actions defined by $r$ are transpositions
\[
\mathcal{L}_{x_1}= (x_2 \; x_3),\quad \mathcal{L}_{x_2}= (x_1 \; x_2),\quad \mathcal{L}_{x_3}= (x_1 \; x_3),
\]
so that $r$ is left nondegenerate. As with the previous example, $r$ is not right-nondegenerate (in the present case all right actions result simply in $x_1$), but this time $S(X,r)$ will nevertheless end up cancellative from both sides.

The associated Yang-Baxter algebra $\cA$  has the set of 6 defining relations
\begin{align*}
\Re_{\cA} = \{f_1 &= x_3x_3- x_1x_1,\  f_2= x_2x_2- x_1x_1,\  f_3= x_2x_3-x_3x_1, \\
 f_4&= x_1x_2-x_3x_1,\  f_5= x_3x_3 - x_2x_1,\
f_6 = x_1x_3 -x_2x_1\}.\end{align*}
 Theorem~\ref{thm:main1} then provides
 quadratic relations $\Re$ which generate the same ideal as $
(\Re_{\cA})$, namely
\begin{align*}
 \Re = \{f_{33}&=f_1 = x_3x_3- x_1x_1,\  f_{22}= f_2= x_2x_2- x_1x_1,\\
       f_{23}&= x_2x_3-x_1x_2,\   f_{31}= x_3x_1-x_1x_2,\\
         f_{32}&= x_3x_2 - x_1x_3,\   f_{21} = x_2x_1 -x_1x_3\}
\end{align*}
coming from the three $r$-orbits  in Figure~\ref{graphex12}(b). Using standard (noncommutative) Gr\"{o}bner basis technique one checks if all ambiguities of overlap coming from the highest
monomials of $f_{ij}$, $2 \leq i \leq 3, 1 \leq j, k\leq 3=$ are solvable. These ambiguities are
\[\begin{array}{c}
x_3x_3x_3,\  x_3x_3x_2,\  x_3x_3x_1,\  x_3x_2x_3,\  x_3x_2x_2,\  x_3x_2x_1, \\
x_2x_3x_3,\  x_2x_3x_2,\ x_2x_3x_1,\  x_2x_2x_3,\  x_2x_2x_2,\  x_2x_2x_1
\end{array}
\]
and one can see (using reductions) that they are all solvable. Therefore the set $\Re$ is a Gr\"{o}bner basis of the ideal $I= (\Re)$ so that the rest of Theorem~\ref{thm:main1} applies. In particular, $\Re$ is reduced and the Yang-Baxter algebra $\cA$ is a PBW algebra with a
set of PBW generators $x_1, x_2, x_3$.
The normal $\k$- basis of $\cA$ is
\[\cN =  \{x_1^i, x_1^ix_2, x_1^ix_3 \mid i\geq 0\},\]
where $x_1^0=1$. It is clear that
\[{}^{x_1}{x_1}= {}^{x_2}{x_2}= {}^{x_3}{x_3} = x_1. \]
In follows from \cite[Thm 3.12]{Colazzo23} that the monoid $S(X,r)$ is cancellative and $\cA$ is right Noetherian.
 \end{ex}

\begin{rmk}
\label{rmk:GKdim1} We also saw in our explicit description of the  $\k$-basis of $\cN (I)$ in Theorem~\ref{thm:main1}  for a finite  left-nondegenerate   idempotent solution $(X,r)$  that  $\cA=\cA(\k,X,r)$  has $\dim \cA_d =n$ for all $d\ge 1$.
This provides a direct route to the known result\cite[Corollary 3.15]{Colazzo23} (1) that the YB algebra in this case has Gelfand-Kirillov dimension $\gkdim\cA = 1$. Recall that a graded algebra $A= \oplus_{d\geq 0}A_d$ has finite Gelfand-Kirillov dimension $\gkdim A = p$ if there exists a polynomial $f(x)$ of degree $p$ with integer coefficients such that $\sum_{0 \leq i \leq d} \dim A_i \leq f(d)$ for all $d \geq 2$, and the integer $p\geq 0$ is minimal. In our case,
\[ d+1\le \sum_{0 \leq i \leq d} \dim \cA_i = 1 +nd, \]
so the sum is bounded above by  $f(d) = 1 +n d$ and we do not have a bound by a constant polynomial in $d$, hence $\gkdim\cA = 1$.
\end{rmk}

We can also use the explicit Gr\"obner bases of $\cA$ in Theorem~\ref{thm:main1}  to give a stronger version of a previous result \cite[Proposition 3.10]{Colazzo22} that the YB algebra $\cA= \cA(\k, X, r)$ of a finite idempotent left nondegenerate solution $(X,r)$ is a finite left module over some $\k$-subalgebra isomorphic to a polynomial algebra in one variable.

\begin{cor}
\label{cor:free} If  $(X,r)$ be a left-nondegenerate idempotent solution with
$X= \{x_1, \cdots, x_n\}$ then $\cA(\k,X,r)$ is a rank $n$ free left module over the polynomial subalgebra $\k [x_1]$, with a set of free generators $1, x_2, \cdots, x_n$.
\end{cor}
\begin{proof}
Let $R= \k [x_1]\subseteq \cA$ and consider the left $R$-module
$\mathfrak{M} =R + \sum_{j=2}^n Rx_j \subseteq \cA$ generated by $1, x_2, \cdots, x_n$. By Theorem~\ref{thm:main1}, the set  $\Re$ there is a Gr\"obner basis w.r.t the deg-lex ordering and $\cN$ given in (\ref{eq:N}) is the normal $\k$-basis of $\cA$.
 By the Diamond Lemma, the algebra $\cA$ is identified with the algebra $(\k \cN,
\bullet)$. More precisely, if $a\in \cA$ is not a constant in $\k $ then there exists
 a  unique presentation of $a$ as a finite linear combination of normal words from in $\cA$:
\[a = a_0 +\sum _{k=1}^q\sum_{j=1}^n \alpha_{kj} x_1^{k-1}x_{j};\quad   a_0,\alpha_{kj} \in \k , \]
 which can be written as
\[a = a_0+f_1(x_1)x_1+ f_2(x_1)x_2 + \cdots +f_n(x_1)x_n \in \mathfrak{M};\quad f_j(x_1)=\sum _{k=1}^{m_j} \alpha_{kj} x_1^{k-1}, \alpha_{kj}, a_0\in
\k .\]
It follows that $\cA =  \mathfrak{M}$.
Moreover, $1, x_2, \cdots, x_n$ is a set of free generators (a left basis) of $\mathfrak{M}$ over $\k [x_1]$.
Indeed, assume there is a relation of the form
\[g_1.1 + g_2x_2 + \cdots +g_nx_n= 0;\quad  g_i = \sum_{k=0}^{m_j} \beta_{ik}x_1^k \in  \k [x_1],\quad  1\leq i \leq n.\]
This implies
\begin{equation}
\label{eq:combination}
  \sum_{k=0}^{m_1} \beta_{1k}x_1^k+ \sum_{k=0}^{m_2} \beta_{2k}x_1^kx_2+ \cdots +\sum_{k=0}^{m_n} \beta_{nk}x_1^k x_n = 0,
 \end{equation}
which is a relation involving only distinct monomials from the normal basis $\cN$. This implies that all coefficients
$\beta_{sk}$ occurring in (\ref{eq:combination}) equal zero, and therefore $g_1(x_1)= g_2(x_1)=  \cdots= g_n(x_1)=0$. It follows
that the set
$1, x_2, \cdots , x_n$ is a left basis of the left $\k [x_1]$-module  $\mathfrak{M}$, so $\mathfrak{M}$ is a free left $\k [x_1]$-module of rank $n$.
 \end{proof}

This corollary also recovers straightforwardly that $\cA$ is left Noetherian as known from \cite[Proposition 3.10]{Colazzo23}(2). Moreover, by   \cite[Thm.~2.2]{SSW85}, each affine (finitely generated) $\k$-algebra $A$ with ${\rm GK}\dim A= 1$ is necessarily PI.

\section{PBW algebras, dual graphs and global dimension}
\label{sec:graphs}

In this section, we obtain new results about general PBW algebras $A$ and investigate a correlation between
the Gelfand Kirillov
dimension ${\rm GK}\dim A$  and  its global dimension $\gldim A$.  Recall that a graded algebra $A$ has \emph{finite global dimension $d$, $\gldim A=d$} if  each graded A-module has a free resolution of length at most  $d$, and $d$ is a minimal such integer. Specifically, Theorem~\ref{thm:gldiminf}  shows that an $n$-generated PBW algebra $A$  has
infinite global dimension whenever  ${\rm GK}\dim A =m
< n$.  The proof depends on new results about graphs of normal words and obstructions associated to
a PBW algebra in \cite[Section 3]{GI12}, where Lemma~\ref{lem:max_graphgkdim1} gives information about these graphs in the case of  ${\rm GK}\dim A= 1$ and $n(n-1)$ quadratic relations (or equivalently, $\dim A_2 = \binom{n}{2}+1$).   We then apply our results to the case of the YB algebra $\cA=\cA(\k,X,r)$ of a general idempotent left-nondegenerate solution, see Corollary~\ref{cor:new}, and more generally in Theorem~\ref{thm:new} to the case where $r$ need not be a solution.

 Let $A = \k \asX/(\Re)$ be a PBW algebra with a set of PBW-generators $X=\{x_1, \cdots , x_n\}$ $n \geq 2$, where $\Re$ is the
 reduced
Gr\"{o}bner basis of the ideal $I =(\Re)$ and let $\textbf{W} = \{LM (f) \mid f\in \Re\}$ be the set of obstructions as in (\ref{eq:obstructions}).
Then the set of normal words $\cN$ modulo $I$ coincides with the set of normal words modulo $\textbf{W}$, $\cN(I) =
\cN(\textbf{W})= \cN(\Re)$ and in this section we let
\[\textbf{N} = \cN_2\]
denote the set of normal words of length $2$. Here,  $X^2$ splits as a disjoint union
\begin{equation}
\label{eq:duality1}
X^2 = \textbf{W} \cup \textbf{N};\quad \textbf{N} = X^2 \setminus \textbf{W},\quad \textbf{W} = X^2 \setminus \textbf{N}.
\end{equation}
Each PBW algebra $A$ has a canonically associated \emph{monomial algebra} $A_{\textbf{W} }= \k \asX/(\textbf{W})$. As
a quadratic monomial algebra, $A_{\textbf{W}}$ is also PBW. In fact, the set of monomials  $\textbf{W}$ is a quadratic Gr\"{o}bner basis of
the ideal $J =(\textbf{W})$ with respect to any (possibly new) enumeration of $X$.
Both algebras $A$ and $A_{\textbf{W}}$ have the same set of obstructions $\textbf{W}$
and therefore they have the same normal $\textbf{k}$-basis $\cN$, the same Hilbert series and the same
growth. It follows from results of Anick \cite{Anickhom} that they share the same global dimension
 \[\gldim  A = \gldim  A_{\textbf{W}}.\]
 More generally, the set of
obstructions $\textbf{W}$ determines uniquely the Hilbert series, the growth (Gelfand-Kirillov dimension) and the global dimension
for
the whole family of PBW algebras $A$ sharing the same obstruction set $\textbf{W}$.
In various cases, especially when we are interested in the type of growth or the global dimension of a PBW algebra $A$,  it is more
convenient to work with the corresponding monomial algebra $A_{\textbf{W}}$. Following \cite[Sec.~3]{GI12}, we define graphs as follows.

\begin{dfn}
\label{dfn:graphM}
 Let $X$ be a set and $M \subseteq X^2$ a subset of monomials of length $2$. We define the graph $\Gamma_M$
corresponding to $M$ to be a directed graph with a set of vertices $V (\Gamma_M) = X$ and directed
edges (arrows) $E = E(\Gamma_M)$ defined as
\[
x \longrightarrow y \in  E\quad \text{iff}\quad  x, y \in X,\  xy \in  M.
\]
Denote by $\widetilde{M}$ the complement $X^2 \setminus M$ and define the `dual' graph $\Gamma_{\widetilde{M}}$ by
$x \longrightarrow y \in  E(\Gamma_{\widetilde{M}})$ \emph{iff} $x,y\in X$ and $x \longrightarrow y$ is not an arrow of $\Gamma_M$. Note that we allow self-arrows a.k.a. elementary loops.

In particular, given a PBW algebra $A$ with a set of PBW-generators $X= \{x_1, \cdots, x_n\}$, let $\bf N$ be the set of normal words of length 2 and $\bf W$ the complementary set of obstructions. Then  $\Gamma_{\textbf{N}}$ is the {\em graph of normal words} and $\Gamma_{\textbf{W}}$ the {\em graph of obstructions} dual to it in the above sense.
\end{dfn}

By the order of a graph $\Gamma$ we will mean the order of its vertex set, $|V (\Gamma)|$. Hence $\Gamma_{\textbf{N}}$ is a
graph of order $|X|$, which has exactly  $|E(\Gamma_{\textbf{N}})|= |\textbf{N}|= \dim A_2$ arrows, since by Remark~\ref{rmk:growth} there is a one-to-one correspondence between the words in $\textbf{N}$ and the arrows of $\Gamma$.  By a path of length $k-1$, $k \geq 2$ in $\Gamma$, we mean a sequence of arrows $v_1 \longrightarrow v_2\longrightarrow
\cdots \longrightarrow v_k$, where
$v_i \longrightarrow v_{i+1} \in  E$. A \emph{cycle} of length $k$ in $\Gamma$ is a path of the shape
$v_1 \longrightarrow v_2\longrightarrow \cdots \longrightarrow v_k\longrightarrow v_1$, where $v_1, \cdots , v_k$ are distinct
vertices. A cycle of length $1$ just means an elementary loop or  self-arrow $x \longrightarrow x\in E$.  Hence, the graph $\Gamma_{\textbf{N}}$ contains a self-arrow $x \longrightarrow x$ whenever $xx \in  \textbf{N}$ and a cycle of length two $x \longrightarrow y \longrightarrow
x$ whenever
$xy, yx \in  \textbf{N}$. We recall that in a directed graph, a {\em bidirected edge} means arrows $x \longrightarrow y, y \longrightarrow x$ in both directions as here.  A directed graph having no bidirected edges is \emph{an oriented graph}. An oriented graph without cycles is \emph{an acyclic oriented graph}.

\begin{ex}\rm Let $(X,r)$ with $X=\{x_1, \cdots,x_n\}$ be in the setting of Theorem~\ref{thm:main1} where $\Re$ is a Gr\"obner basis of $(\Re)$ so that the YB-algebra  $\cA(\k,X,r)$ is PBW. We saw that
set of normal words of length $2$ is $\textbf{N}= \{x_1x_j\mid 1 \leq j \leq n\}$, so the graph $\Gamma_{\textbf{N}}$ has order
$n$, one self-arrow $x_1\longrightarrow x_1$  and exactly $n-1$ additional directed edges $x_1 \longrightarrow x_j$, $2 \leq  j \leq n$ as shown in  Figure~\ref{figGamNmain1}.  \end{ex}
\begin{figure}
\[\includegraphics[scale=1.3]{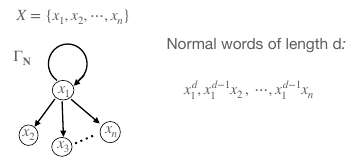}\]
\caption{Graph $\Gamma_{\bf N}$ in the PBW case where $\Re$ is a Gr\"obner basis in Theorem~\ref{thm:main1}.\label{figGamNmain1}}
\end{figure}

In general, the graph of normal words  $\Gamma_{\textbf{N}}$ of a given PBW algebra is a directed graph which may contain
bidirected edges, so it is not necessarily an oriented graph. Also observe that any directed graph $\Gamma$ with a set of vertices $X
= \{x_1 \cdots, x_n\}$  uniquely a quadratic monomial algebra $A$. Indeed, we let $\textbf{N}=\{xy \in X^2 \mid x\longrightarrow y \in E(\Gamma)\}$ and let $\textbf{W}= X^2 \setminus
\textbf{N}$. Then the monomial algebra $A=\k \asX/(\textbf{W})$ has $x_1, \cdots, x_n$ as a set of PBW-generators, $\textbf{W}$ as
a set of obstructions, $\textbf{N}$ as a set of normal words of length $2$ and $\Gamma=\Gamma_{\textbf{N}}$. The graph of normal words $\Gamma_{\textbf{N}}$
was introduced in a more general context by Ufnarovski and the following remark is a particular case of a more general result of
\cite{ufnarovski}.
\begin{rmk}
\label{rmk:growth}
Let  $A$ be a PBW algebra and let $\cN$ be its set of normal words, with $\textbf{N}= \cN_2$. Then:
\begin{enumerate}\item[(i) ]For every $m\geq 1$, there is a one-to-one correspondence between the set $N_m$ of normal
words of length $m$ and the set of paths of length $m-1$ in the graph $\Gamma_{\textbf{N}}$. The path
$a_1 \longrightarrow a_2 \longrightarrow \cdots \longrightarrow a_m$ (these are not necessarily distinct vertices) corresponds
to the word $a_1a_2 \cdots a_m\in N_m$. In particular,  each word $xy \in \textbf{N}$ is represented by a path $x\longrightarrow y$ of length $1$, possibly a self-arrow, and each $x\in X$ is represented by a path of length $0$.
\item[(ii)]  $A$ has exponential growth \emph{iff} the graph $\Gamma_{\textbf{N}}$ has two intersecting cycles.
\item[(iii)] $A$ has polynomial growth of degree $m$ (${\rm GK}\dim A=m$) \emph{iff} $\Gamma_{\textbf{N}}$ has no intersecting cycles
    and $m$ is the largest
number of (oriented) cycles occurring in a path of $\Gamma_{\textbf{N}}$.
\end{enumerate}
\end{rmk}

The graph of obstructions $\Gamma_{\textbf{W}}$ can be used to determine explicitly the global dimension of a PBW algebra. It was shown in  \cite[Sec.~3]{GI12} that a PBW algebra $A$ has finite global dimension $d$ \emph{iff} $\Gamma_{\textbf{W}}$ is an acyclic oriented graph and $d-1$ is the maximal length of a path occurring in $\Gamma_{\textbf{W}}$. An immediate consequence of this  which we will need, is the following.
\begin{cor}
\label{cor:gldim}
A PBW algebra $A$ has infinite global dimension \emph{iff} the graph of obstructions $\Gamma_{\textbf{W}}$ has a cycle.
\end{cor}
Algorithmic methods for the computation of global dimension of the more general case of standard finitely presented algebras with polynomial growth can be found in \cite{GI88}.

Next, a complete oriented graph $\Gamma$ is called a \emph{tournament} (or a \emph{tour}). In other words, a tournament is a
directed graph in which each pair of vertices is joined by a single arrow having a unique direction.
Clearly, a complete directed graph without cycles (of any length) is an acyclic tournament.  The
following is well known in graph theory.
\begin{rmk}
\label{rmk: graph1}
\begin{enumerate}
\item
An
acyclic oriented graph $\Gamma$ with $n$ vertices is a tournament \emph{iff}  $\Gamma$ has exactly $\binom{n}{2}$ arrows (directed
edges).
\item
 Let $\Gamma$ be an acyclic tournament of order $n$. Then the set of its vertices $V = V (\Gamma)$
can be labeled $V = \{y_1, y_2, \cdots , y_n \}$ so that the set of arrows is
\begin{equation}
E(\Gamma ) = \{y_i \longrightarrow y_j \mid  1 \leq i < j \leq n\}.
\end{equation}
Analogously, the vertices can be labeled $V = \{z_1, z_2, \cdots , z_n \}$ so that
\[E(\Gamma) = \{z_j \longrightarrow z_i \mid  n \geq j > i \geq 1\}.\]
\end{enumerate}
\end{rmk}

The proof of the following lemma was kindly communicated by Peter Cameron.
\begin{lem}
\label{lem:peter}
Suppose $\Gamma$ is an acyclic directed graph with a set of vertices $V= \{x_1, \cdots, x_n\}$.  Then $\Gamma$ is a subgraph of an
acyclic tournament $\Gamma_0$ with the same set of vertices.
\end{lem}
\begin{proof}
We claim that one can add new directed edges to connect every two vertices in $V$ which are not connected in such a way that the
resulting graph $\Gamma_0$ is an acyclic tournament. This can be proved by induction on the number of missing directed edges. So all we have
to do for the inductive step is to add one directed edge. Suppose that $x, y \in V$ are not joined. Then we claim that we can put a directed edge
between them without creating a cycle.
Suppose this is false. Then adding $x\longrightarrow y$ would create a cycle $x \longrightarrow y \longrightarrow
u_1\longrightarrow\cdots \longrightarrow u_r \longrightarrow x$, and adding $y\longrightarrow x$ would create a cycle
$y \longrightarrow x \longrightarrow v_1\longrightarrow\cdots \longrightarrow v_s \longrightarrow y.$
 But then there is a cycle
 \[y \longrightarrow u_1\longrightarrow \cdots \longrightarrow u_r \longrightarrow x \longrightarrow v_1\longrightarrow \cdots
 \longrightarrow v_r \longrightarrow y,\]
 contradicting that we start with an acyclic directed graph.
\end{proof}
\begin{lem}
\label{lem:max_graphgkdim1} Let $A= A_{\textbf{W}}$ be a quadratic monomial algebra generated by $X=\{x_1, \cdots, x_n\}$ and
presented as $A_W = \k \langle x_1, \cdots, x_n\rangle/(\textbf{W})$, where $\textbf{W }\subset X^2$ is a set of monomials of
length $2$. Let
$\textbf{N}$ be the set of normal words of length $2$ and assume that $x_1x_1 \in \textbf{N}$ , and that each vertex $x_j$ in the
graph $\Gamma_{\textbf{N}}$ is connected with $x_1$ by a path. Then the following are equivalent:
\begin{enumerate}
\item The algebra $A$ has Gelfand-Kirillov dimension ${\rm GK}\dim A=1$ and $\dim A_2 =\binom{n}{2}+1$;
\item The graph $\Gamma_{\textbf{N}}$ is formed out of an acyclic tournament $\Gamma_1$ with vertices $V(\Gamma_1) = X=
    V(\Gamma_{\textbf{N}})$ to which a single self-arrow $x_1\longrightarrow x_1$ is added, so $E(\Gamma_{\textbf{N}})=
    E(\Gamma_1)\bigcup \{x_1 \longrightarrow x_1\}$.
    \item There is a (possibly new) enumeration $X = \{y_1\cdots, y_n\}$, such that
   \begin{equation}
    \label{eq:normalwords2}
   \textbf{N} =\{y_iy_j\mid 1 \leq i < j\leq n\}\cup \{yy\}
    \end{equation}
for some fixed $y\in X$.  \end{enumerate}
Moreover, suppose $B$ is a monomial algebra generated by $X= \{x_1, \cdots, x_n\}$ with ${\rm GK}\dim B =1$, and such that $x_1x_1$ is
a normal word for $B$. Then
\[
\dim B_2 \leq \dim A_2=\binom{n}{2} +1.
\]
\end{lem}
\begin{proof}
Let $\Gamma_1$  be the subgraph of $\Gamma_{\textbf{N}}$ obtained by `erasing' the arrow $x_1 \longrightarrow x_1$,
so $E(\Gamma_{\textbf{N}})= E(\Gamma_1)\bigcup \{x_1 \longrightarrow x_1\}$, and $|E(\Gamma_1)|= |E(\Gamma_{\textbf{N}})|- 1$ .
There are equalities
\begin{equation}
\label{eq:numberedges1}
\dim A_2 = |\textbf{N}| = |E(\Gamma_{\textbf{N}})|.
\end{equation}

(1) $\Longrightarrow$ (2). Assume $A$ satisfies (1). Then ${\rm GK}\dim A =1$ implies that the graph $\Gamma_{\textbf{N}}$  does not
have two cycles connected with a path, or passing through a vertex, see Remark~\ref{rmk:growth}. Moreover, the assumption that
every vertex $x_j$  is connected with $x_1$ by a path implies that the only cycle of $\Gamma_{\textbf{N}}$ is the self-arrow
$x_1\longrightarrow x_1.$ It follows that the subgraph $\Gamma_1$ is an  acyclic directed graph  with exactly $\binom{n}{2}$
arrows. Now Remark~\ref{rmk: graph1} (1) implies that $\Gamma_1$ is an acyclic tournament and therefore the graph
$\Gamma_{\textbf{N}}$ has the desired shape.

(2) $\Longrightarrow$ (3). Follows from Remark~\ref{rmk: graph1}, part (2).

(3) $\Longrightarrow$ (1). Assume that after a possible relabeling of the vertices  $X = \{y_1\cdots, y_n\}$ of
$\Gamma_{\textbf{N}}$, the set of arrows satisfies (\ref{eq:normalwords2}).
Clearly, $\Gamma_{\textbf{N}}$ has exactly  $\binom{n}{2} +1$ arrows, hence $\dim A_2= \binom{n}{2} +1$. Moreover,
$\Gamma_{\textbf{N}}$ contains exactly one cycle  and therefore, by Remark~\ref{rmk:growth}, ${\rm GK}\dim A=1$.
\end{proof}

Observe that part (1) of the Lemma also holds  if the graph $\Gamma_{\textbf{N}}$ is formed out of an acyclic tournament
$\Gamma_1$ with vertices $V(\Gamma_1) = X= V(\Gamma_{\textbf{N}})$ to which is added an arrow $x\longrightarrow z,$ for some $x, z
\in X, x\neq z$,  so  $E(\Gamma_{\textbf{N}})= E(\Gamma_1)\bigcup \{x \longrightarrow z\}$. In this case $\Gamma_{\textbf{N}}$
has unique cycle $ x \longrightarrow z \longrightarrow x$.

\begin{thm}
\label{thm:gldiminf}
If $A$ is a PBW algebra with a set of PBW-generators $x_1, \cdots, x_n$, $n \geq 2$, and ${\rm GK}\dim A=m < n$, then
$A$ has infinite global dimension, $\gldim A = \infty$.
\end{thm}
\begin{proof}
Consider the graph $\Gamma_{\textbf{N}}$  of normal words. Two cases are possible:

(a) There exists a vertex $x_i\in X$ without a self-arrow $x_i \longrightarrow x_i$ in $\Gamma_{\textbf{N}}$. Then the graph of
obstructions $\Gamma_{\textbf{W}}$
contains the self-arrow $x_i \longrightarrow x_i$ , and therefore, by Corollary~\ref{cor:gldim},  $\gldim A = \infty.$

(b)
The graph $\Gamma_{\textbf{N}}$ contains $n$ self-arrows $x_i \longrightarrow x_i$, $1 \leq i \leq n$, then by Remark~\ref{rmk:growth},
$\Gamma_{\textbf{N}}$ does not have additional cycles (otherwise $A$ would have exponential growth). We shall prove  that $\Gamma_{\textbf{N}}$
has two vertices $x, y \in X, x \neq y$ which are not connected with an arrow.  Assume on the contrary, that every two vertices
are connected with an arrow in $E(\Gamma_{\textbf{N}})$. Consider the subgraph graph $\Gamma_1$ of  $\Gamma_{\textbf{N}}$ obtained
by `erasing' all self-arrows, so
$\Gamma_1$ has set of arrow $E(\Gamma_1) = E(\Gamma_{\textbf{N}}) \setminus \{x_i \longrightarrow x_i\mid 1 \leq i \leq n\}$. By
our assumption every two vertices of $\Gamma_1$ are connected with an arrow and therefore $\Gamma_1$  is an acyclic tournament of
order $n$.
Then by Remark~\ref{rmk: graph1},
 the set of its vertices $V (\Gamma_1)= X$
can be relabeled $\{y_1, y_2, \cdots , y_n \}$, so that the set of arrows is
\begin{equation}
E(\Gamma_1) = \{y_i \longrightarrow y_j \mid  1 \leq i < j \leq n\}.
\end{equation}
This implies that the graph $\Gamma_{\textbf{N}}$ contains a  path with $n$-self-arrows.
\[\includegraphics{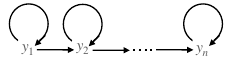}\]
It follows from  Remark~\ref{rmk:growth}  that ${\rm GK}\dim A\geq n,$ which contradicts the hypothesis ${\rm GK}\dim A< n$.
Therefore, there are two vertices $x, z \in X, x \neq z$ which are not connected with an arrow in $\Gamma_{\textbf{N}}$, so
 the obstruction graph $\Gamma_{\textbf{W}}$
contains the cycle  $x\longrightarrow z\longrightarrow x$. Corollary~\ref{cor:gldim} then implies that $\gldim A = \infty.$
\end{proof}

\begin{cor}
\label{cor:new}
If $(X,r)$ is a finite left nondegenerate idempotent solution of order $|X|=n \ge 2,$ then the YB algebra $\cA(\k,X,r)$ has infinite global dimension, $\gldim \cA(\k,X,r) =\infty.$
\end{cor}
\begin{proof}
We know that the YB algebra of every  left-nondegenerate idempotent solution is PBW and has Gelfand-Kirillov
dimension
${\rm GK}\dim \cA = 1$, see Remark~\ref{rmk:GKdim1}, or \cite[Corollary 3.15]{Colazzo23}. Therefore, by Theorem~\ref{thm:gldiminf}, $\cA$ has infinite global dimension.\end{proof}

The following lemma is about general idempotent quadratic sets $(X,r)$. We do not assume any kind of nondegeneracy, nor that
$(X,r)$ is a solution of YBE.
\begin{lem}
\label{lem:2cancel}
Suppose $(X,r)$ is a left nondegenerate quadratic set, where $r^2 = r$, and $\cA= \cA(X, \k , r)$ is the corresponding quadratic
algebra. Assume that an enumeration $X = \{x_1, \cdots, x_n\}$ is fixed and, as usual,
    consider the deg-lex ordering on $\asX$. Then the words  $x_1x_1, x_1x_2, \cdots, x_1x_n$  are  normal and distinct in
    $\cA$, hence $\dim \cA_2 \geq n.$
\end{lem}
\begin{proof}
For the quadratic algebra $A$, a word $xy\in X^2$ is not normal \emph{iff} $xy- zt$ is in the ideal of relations of $A$, where
$zt \in X^2$ and $xy > zt$ in the deg-lex ordering on $\asX$. It is clear that $x_1x_1\in \textbf{N}$. Suppose $x_1x_j$ is not
normal for some $j>1$, then $x_1x_j - ab$ is in the ideal of relations of $\cA$, where $x_1x_j  >ab$. This implies $a= x_1$, and
$b=x_i$ with $1\leq i < j$. Therefore the equality $x_1x_j = x_1x_i$ holds in $\cA$, and also in the monoid $S= S(X,r)$, but this contradicts Lemma~\ref{lem:orbitsgeneral} (\ref{pro:orbits2}). Hence all monomials $x_1x_j, 1 \leq j \leq n$ are
normal, and therefore $\dim \cA_2 \geq n$.
\end{proof}

\begin{thm}
\label{thm:new}
Suppose that $(X,r)$ is a left nondegenerate idempotent quadratic set, with $X= \{x_1, \cdots, x_n\}$
and $\cA= \cA(\k, X, r)= \k \asX/ I$ is its associated quadratic algebra.
\begin{enumerate}
\item
\label{thm:new1}
 $n\leq \dim \cA_2$ is an exact  bound  (i.e., $I$ is generated by at least $n(n-1)$ linearly independent quadratic relations and this bound can be attained).
\item \label{thm:new2} Suppose that in addition  $\cA$ is a PBW algebra with PBW generators $x_1, \cdots, x_n$.
\begin{enumerate}[(i)]
\item  If $\dim A_2 = n$
then
\[\dim \cA_d= n,\quad \forall\  d \geq 1.\]
\item
If ${\rm GK}\dim\cA=1$ then
\[
\dim \cA_2  \leq \binom{n}{2}+1.
\]
\end{enumerate}
\end{enumerate}
\end{thm}
\begin{proof}

(\ref{thm:new1}) It follows from Lemma~\ref{lem:2cancel} that the words $x_1x_1, x_1x_2, \cdots, x_1x_n$  are normal (modulo $I$) and distinct in $\cA$, so the set of normal words of length $2$ has at least $n$ elements:
\[\textbf{N} \supseteq \{x_1x_1, x_1x_2, \cdots, x_1x_n\},\]
and therefore $\dim \cA_2 \geq n.$ Moreover, this lower bound is exact since it is attained by the YB algebra $\cA$ of any left nondegenerate idempotent solution $(X,r)$ has $\dim\cA_2=n$ by Theorem~\ref{thm:main1}(2).

(\ref{thm:new2}) Assume for the remainder that $\cA$ is a PBW algebra with PBW generators $x_1, \cdots, x_n$. For part (i),  the ideal $I$ of relations has a quadratic Gr\"{o}bner basis, say $G$.
Then $\dim \cA_2= n$ holds \emph{iff}
\[\textbf{N} = \{x_1x_1, x_1x_2, \cdots, x_1x_n\},\]
or equivalently,  the graph $\Gamma_{\textbf{N}}$ has exactly $n$ arrows
\[ E(\Gamma_{\textbf{N}}) = \{x_1\longrightarrow x_1,  x_1 \longrightarrow x_j, 1 < j < n\}.\]
By Remark~\ref{rmk:growth}, there is a bijective correspondence between the set of normal words (modulo $G$) of length $d, d \geq 1$
and the set of paths of length $d-1$ in $\Gamma_{\textbf{N}}$. (Clearly, for $d= 1$, each $x\in X$ is represented by a path of length $0$). There are exactly $n$ such paths:
\[\underbrace{x_1 \rightarrow x_1 \rightarrow \cdots \rightarrow x_1}_{d-2\ \text{arrows}} \rightarrow x_j, 1 \leq j \leq n,\]
corresponding to the words $x_1^{d-1}x_j$ for $1\leq j \leq n$.  Hence  the set of normal words modulo $G$ of length $d$, is exactly
$\cN_d (G) = \cN_d (I) = \{x_1^{d-1}x_j, 1 \leq j \leq n\}$,
so $\dim \cA_d = n$ for all $d \geq 1$.

For part (ii), suppose that $\gkdim \cA= 1$. We shall prove that $\dim \cA_2  \leq \binom{n}{2}+1$, where  $\dim \cA_2=|E(\Gamma_{\textbf{N}})|$.  Observe that the graph $\Gamma_{\textbf{N}}$ has a self-arrow $x_1\longrightarrow x_1$, and every vertex $x_i$  is connected with $x_1$
by an arrow.
Then  Remark~\ref{rmk:growth} and ${\rm GK}\dim \cA= 1$  imply that the graph $\Gamma_{\textbf{N}}$ has no additional cycles. It
follows that the subgraph $\Gamma_1$ obtained from $\Gamma_{\textbf{N}}$ by `erasing' the self-arrow $x_1\longrightarrow x_1$ is an
acyclic directed graph with a set of vertices  $\{x_1, \cdots, x_n\}.$
Now Lemma~\ref{lem:peter} implies that $\Gamma_1$ is a subgraph of an acyclic tournament $\Gamma_0$ with the same set of
vertices. Therefore
the number of its arrows satisfies the inequality
\[|E(\Gamma_1)|\leq |E(\Gamma_0)|=\binom{n}{2}.\]
But the number of arrows of $\Gamma_{\textbf{N}}$ is $\dim \cA_2=|E(\Gamma_{\textbf{N}})|= |E(\Gamma_1)|+1$, proving the stated bound.
\end{proof}

 We end the section with an open question.
\begin{que}
\label{que:conjecture}
Suppose $(X,r)$ is a left nondegenerate idempotent quadratic set on $X= \{x_1, \cdots, x_n\}$ for which the associated quadratic algebra $\cA=\cA(\k,X,r)$ is PBW
with PBW-generators
the elements of $X$ taken with this fixed enumeration, and assume that $\cA$ has Gelfand Kirillov dimension $1$.
What is the exact upper bound for $\dim \cA_2$, i.e. the minimal possible number of relations of $\cA$?
\end{que}

\section{Permutation idempotent solutions and their Yang-Baxter algebras}
\label{sec:perm}

In this section, we focus on a class of concrete left nondegenerate idempotent  solutions $(X, r_f)$ which we call `permutation idempotent solutions', where $f\in \Sym
(X)$. Such solutions appeared in \cite[Prop.~3.15]{Colazzo22}. Our aim in this section is to provide new results on their associated Yang-Baxter
algebra $\cA(\k,X,r_f)$, see Proposition~\ref{pro:YBalgPermSol}, Corollary~\ref{cor:main} and Corollary~\ref{cor:normalforms}. We first start with an abstract characterisation of the solutions among quadratic sets of a certain form. After this, until the end of the paper, $X$ will be assumed to be of finite order $n
\geq 2$.

\begin{pro}
\label{pro:important}
Let $X$ be a nonempty set of arbitrary cardinality, and let
$r: X\times X\longrightarrow X\times X$ be a map such $(X,r)$ is left nondegenerate and
\begin{equation}
\label{eq:r(x,y)}
r(x,y) = ({}^xy, y), \;  \; \forall \; x,y \in X.
\end{equation}
Then the following are equivalent:
\begin{enumerate}
\item
\label{solution}
$(X,r)$ is a solution of YBE;
\item
\label{l1}
$(X,r)$ satisfies condition \textbf{l1} in  Remark~\ref{rmk:YBE1};
\item
\label{leftequivalent}
There exists a bijection $f \in \Sym(X)$, such that
\[r(x,y) = (f(y), y),\quad\forall\ x,y \in X.\]
\end{enumerate}
In this case $(X, r)$ is an idempotent solution.
\end{pro}
\begin{proof}
(\ref{l1}) $\Longrightarrow$ (\ref{leftequivalent}). It follows from \textbf{l1} and  (\ref{eq:r(x,y)}) that
\begin{equation}
\label{eq:l1a}
{}^{(xy)}a  = {}^{x}{({}^ya)}= {}^{{}^xy}{({}^ya)},\quad\forall\  a, x,y \in X.
\end{equation}
Let $t\in X$ be an arbitrary element. By the left nondegeneracy there exists an $a\in X$ such that ${}^ya= t$.
Therefore
\begin{equation}
\label{eq:l1b}
{}^{x}{t}= {}^{{}^xy}{t},\quad\forall\  x,y, t \in X.
\end{equation}
Let $z\in X$ be an arbitrary element. By the left nondegeneracy again, there exists an $y\in X$, such that ${}^xy=z$.
This together with (\ref{eq:l1b})  implies
\begin{equation}
\label{eq:l1c}
{}^{x}{t}= {}^zt, \quad\forall\  x,z, t \in X.
\end{equation}
Therefore $\LL_x = \LL_z,$ for all $x,z \in X$. In particular, there exists a bijection $f \in \Sym(X)$, such that
 $\cL_x= f,$ for all $x\in X$. This proves part (\ref{leftequivalent}).

 (\ref{leftequivalent}) $\Longrightarrow$ (\ref{solution})
   Assume (\ref{leftequivalent}). We shall prove that $(X,r)$ is a solution. Let $xyz \in X^3$ be an arbitrary monomial. The
   `Yang-Baxter diagrams',
\begin{equation}
\label{ybediagram}
\begin{array}{l}
 xyz \longrightarrow^{r^{12}} \; f(y)y z\longrightarrow^{r^{23}} \; f(y)f(z) z \longrightarrow^{r^{12}} \;f^2(z)f(z) z\\
 \\
 xyz \longrightarrow^{r^{23}}\; xf(z) z\longrightarrow^{r^{12}}\; f^2(z)f(z) z \longrightarrow^{r^{23}} \;f^2(z)f(z) z\\
 \end{array}
 \end{equation}
show that \[r^{12}r^{23}r^{12} (xyz) =r^{23}r^{12}r^{23}(xyz),\]
for every monomial $xyz\in X^3$, and therefore $(X,r)$ is a solution of YBE. The implication (\ref{solution}) $\Longrightarrow$  (\ref{l1}) follows straightforwardly from Remark~\ref{rmk:YBE1}. This proves the equivalence of (1), (2) and (3).

Finally, condition (\ref{leftequivalent}) implies that
\[r^2(x, y)= r(r(x,y)) = r(f(y), y) = (f(y), y) = r(x,y), \]
for all $x, y \in X$ so that $(X,r)$ is idempotent. \end{proof}

\begin{dfn}\label{defperm}
Let $X$ be a nonempty set and $f \in \Sym(X)$.
We refer to the left nondegenerate solution $(X,r_f)$ where
\[ r_f: X\times X \longrightarrow X\times X,\quad r_f(x,y) = (f(y), y)\]
as a  \emph{permutation idempotent solution}. We denote by $P_n$ the class of all permutation idempotent solution on sets $X$ of finite
order $n$ up to isomorphism.
\end{dfn}

Note that Colazzo at al \cite{Colazzo22} studied finite nondegenerate idempotent solutions of YBE and introduced an example $(X,r)$ in which $r(x, y) = (\lambda (y), y)$,  where $\lambda : X \longrightarrow X$ is a permutation. Proposition~\ref{pro:important} puts this in context as all solutions among quadratic sets of a certain form as stated. That $\cA(\k,X,r_f)$ for such solutions is PBW is known from \cite[Prop.~3.15]{Colazzo22} but the reduced Gr\"{o}bner basis and a PBW $\k$-basis which we next provide are new, as an application of Theorem~\ref{thm:main1}.

\begin{pro}
\label{pro:YBalgPermSol}
 Let $(X, r_f)$ be a permutation idempotent solution with $X= \{x_1, x_2, \cdots, x_n\}$ and $f \in \Sym (X)$.
 \begin{enumerate}
  \item
  \label{pro:YBalgPermSol1}
  The YB algebra $\cA=\cA(X,r_f)$ is  PBW  with a standard finite presentation
\begin{equation}
\label{eq:rels1p}
\begin{array}{c}
\cA= \k \asX /(\Re);\quad \Re = \{x_ix_j- x_1x_j\mid 2 \leq i \leq n, 1 \leq j\leq n\},
\end{array}
\end{equation}
where $\Re$ consists of $n(n-1)$ linearly independent relations as shown and is the reduced Gr\"{o}bner basis of $I= (\Re)$  w.r.t. deg-lex order on $\asX$.
\item
\label{pro:YBalgPermSol2}
The set of normal monomials
\[
\cN =\cN(\Re)  = \{1\}\cup\{x_1^{d-1}x_p\mid d\ge 1,\ p \in \{1, 2, \cdots, n\}\}
\]
is a PBW $\k$-basis of $\cA$. Moreover,  $\cA$ is isomorphic as a graded algebra to $(\k  \cN, \bullet )$,
where for each $d\geq 1$, the graded component $\cA_d$ has a $\k$-basis
\begin{equation}
\label{eq:N_d}
\cN_d= \{w_1 = x_1^d < w_2 = x_1^{d-1} x_2 < \cdots <w_n = x_1^{d-1} x_n \},
\end{equation}
the set of normal monomials of length $d$.
\item
\label{pro:YBalgPermSol3}
The Hilbert series $H_\cA$ of $\cA$ is (\ref{Hidemp}).
 \end{enumerate}
\end{pro}
\begin{proof}
Clearly, every permutation idempotent solution is left nondegenerate, so we can apply Theorem~\ref{thm:main1} and the lemmas leading up to it. Since $r(x_ix_j) = (f(x_j)x_j)=r(x_1x_j)$ for all $1 \leq i \leq n, 1 \leq j \leq n$, $x_ix_j$ lies in the $r$-orbit $\cO_j$  containing $x_1x_j$ (i.e., $k_{ij}=j$ in Lemma~\ref{lem:orbits}). Clearly the entire $r$-orbit is exactly $n$ distinct elements
\[
\cO _j = \{x_1x_j, x_2x_j, \cdots, x_nx_j\},\quad\forall\ 1\le j\le n.
\]
Theorem~\ref{thm:main1} then implies that the algebra $\cA$ has a finite presentation (\ref{eq:rels1p}) as stated and that $\Re$ is a Gr\"{o}bner basis of the two sided ideal $I = (\Re)$ w.r.t. deg-lex ordering.  We also know from Theorem~\ref{thm:main1} that  $\cN=\cN(\Re) = \cN(I)$ is described explicitly by (\ref{eq:normalbasis}) and is (a PBW) $\k$-basis of $\cA$ leading to the Hilbert series as there. Bergman's Diamond lemma (and Remark~\ref{rmk:main2}) also implies
that
if we consider the space $\k  \cN$
endowed with multiplication defined by
 \[f \bullet g := \Nor (fg),\quad \forall\  f,g \in \k  \cN\]
then
$(\k  \cN, \bullet )$
has a well-defined structure of a graded algebra, and there is an isomorphism of
graded algebras as also stated in part (2).  \end{proof}

In view of the proposition, we shall often identify the algebra  $\cA$ with $(\k \cN, \bullet )$. Also note that  for a fixed enumeration $X = \{x_1, \cdots, x_n\}$, Proposition~\ref{pro:YBalgPermSol} provides an
explicit standard finite presentation (\ref{eq:rels1}) of $\cA(\k,X,r_f)$ which {\em does not depend on the permutation $f$}, hence (for a given enumeration) these are the same algebra with the same explicitly given  PBW $\k$-basis  denoted $\cN$. We therefore have following important corollary.
\begin{cor}
\label{cor:main}
All permutation idempotent solutions $(X,r_f)$ where $f\in \Sym(X)$ have isomorphic Yang-Baxter algebras $\cA(\k,X,r_f)$. For a fixed enumeration $X= \{x_1, \cdots, x_n\}$, these
algebras are the same, with the same standard finite presentation given in (\ref{eq:rels1p}) and the same $\k$-bases $\cN$ of normal words
given explicitly in (\ref{eq:normalbasis}).
\end{cor}

Another consequence of Proposition~\ref{pro:YBalgPermSol} is the following.

\begin{cor}\label{cor:normalforms}
Let $(X, r_f)$ be a permutation idempotent solution with YB-algebra $\cA= \cA(\k, X. r_f)$ .
Then for every $d\geq 2$ the equalities
\begin{equation*}
y_1y_2\cdots y_{d-1}x_q = (x_1)^{d-1}x_q, \quad\forall\  y_i\in X,\  q \in \{1, \cdots, n\}
\end{equation*}
hold in the monoid $S(X,r_f)$ and hence in $\cA(\k,X,r_f)$.
\end{cor}
\begin{proof} We have proven that the set  $\Re$ in (\ref{eq:rels1p}) is a Gr\"{o}bner basis of $I = (\Re)$, the normal $\k$- basis $\cN$ of $\cA$ is given in (\ref{eq:normalbasis})
and the normal $\k$- basis $\cN_d$ of $\cA_d$ is in (\ref{eq:N_d}). It follows from the Diamond Lemma that every word
$u \in \asX$ has a unique normal form (modulo $\Re$) denoted by $\Nor (u)\in \cN$. Moreover, since $\Re$ consists of homogeneous quadratic relations which  preserve the grading,  $\Nor (u)\in \cN_d,$ whenever $u$ has length $|u|= d$.
The Diamond Lemma also implies that
\[
\Nor (uv) = \Nor(\Nor(u)\Nor(v)), \quad\forall\  u, v \in \asX.
\]
Now the form of the relations (\ref{eq:rels1p}) implies that
\[
\Nor(x_jx_p) = x_1x_p,\quad\forall\ 1 \leq j, p \leq n.
\]
Applying these two rules, and induction on $d$ yields that
\begin{equation*}
\Nor(y_1y_2\cdots y_{d-1}x_q) = (x_1)^{d-1}x_q,\quad\forall\  y_i\in X,\  q \in \{1, \cdots, n\},
\end{equation*}
which implies the equalities stated. \end{proof}

\section{Veronese subalgebras and Segre products for left nondegenerate idempotent solutions}
\label{sec:ver}

In this section we will determine the $d$-Veronese subalgebras $\cA^{(d)}$ of $\cA= \cA(\k,X,r)$ for any finite left-nondegenerate idempotent solution $(X,r)$ and an  analogue of Veronese morphisms. We will also look at Segre products and an analogue of Segre morphisms. Our main result, Theorem~\ref{thm:ver}, is that $\cA^{(d)}$ is then isomorphic to the YB-algebra of another left-nondegenerate idempotent solution on a set of the same order as $X$. In the permutation idempotent case we find, see Corollary~\ref{cor:verperm}, that these are again permutation idempotent solutions.  Our proofs will depend on a new result, Theorem~\ref{thm:veridemp}, that the restriction $r_d$ of the solution $r_S$ on $S(X,r)$ for any braided set $(X,r)$ in \cite{GIM08} is idempotent for all $d$ if  $r$ is idempotent. Note that while our general strategy  follows the lines of \cite{GI_Veronese}, which mainly treated the involutive case, the particular characteristics and our results in the idempotent case are significantly different. Similarly, Section~\ref{sec:seg} is in line with the approach of \cite{GI23} in the involutive case but with a very different proof needed in our case and resulting in an isomorphism.

\subsection{Veronese subalgebras for a class of quadratic algebras}\label{sec:vergen}

This section has a general result for a class of quadratic algebras, which we then specialise to the case of Yang-Baxter algbras. We first recall the general definition of the $d$-Veronese subalgebra of a graded algebras, as in the text \cite[Sec.~3.2]{PoPo}.
\begin{dfn}
Let $A= \bigoplus_{m\in\N_0}A_{m}$ be a graded $\k$-algebra. For any integer $d\geq
2$, the \emph{$d$-Veronese subalgebra} of $A$ is the graded algebra
\begin{equation}
\label{eq:A^d}
A^{(d)}=\bigoplus_{m\in\N_0} A_{md}.
\end{equation}
\end{dfn}
Here,  $A^{(d)}$ is necessarily a subalgebra of $A$ but the
embedding is not a graded algebra morphism. It is known  \cite[Prop.~2.2, Chap.~3]{PoPo}
that if $A$ is a one-generated quadratic Koszul algebra then its Veronese
subalgebras are also one-generated quadratic and Koszul.
Moreover, \cite[Prop.~4.3, Chap.~4]{PoPo} implies  that if $x_1, \cdots, x_n$ is a set of PBW-generators of a
PBW algebra $A$, then the elements of its
PBW-basis of degree $d$, taken in lexicographical order, are PBW-generators of the Veronese subalgebra $A^{(d)}$.

Now suppose
$A= \k \asX /(\Re)$ is a quadratic algebra, with a set of one generators $X= \{x_1, \cdots, x_n\}$,
and a set of quadratic relations $\Re$ of the form
\begin{equation}\label{Refi}
\Re= \{f_i= x_iy_i - a_ib_i\ |\  1 \leq i \leq M\},\end{equation}
where $\Re$ consists  of precisely $M$ linearly independent
binomials given above such that
 for every $i, 1 \leq i \leq M$ there is an inequality   $x_iy_i > a_ib_i$, and the monomial $x_iy_i$ occurs once in $\Re$. Denote
 \[\textbf{W}: = \{HM(f_i)\mid 1 \leq i \leq M\} = HM(\Re).\]

We assume without loss of generality that $\Re$ is a reduced set. Otherwise, using a standard technique for reductions we can find a reduced set
\[\Re_0= \{g_i = x_iy_i -x_i^{\prime}y_i^{\prime}| x_iy_i > x_i^{\prime}y_i^{\prime}, 1\leq i \leq M\},\]
 generating the same ideal
$(\Re) = (\Re_0)$ and such that
$HM(\Re _0) = HM(\Re) = \textbf{W}$. In particular, every monomial
$x_iy_i\in HM(\Re_0)$ occurs exactly once in $\Re_0$.

Proceeding with  $\Re$ reduced, we recall that there exists a unique reduced Gr\"{o}bner basis $\cG=\cG(I)$ of the ideal $I=(\Re)$ w.r.t. deg-lex ordering on $\asX$, that $\Re \subseteq \cG$ and the latter is finite or countably infinite. It follows from the above form of $\Re$ that  $\cG$ consists of homogeneous binomials $F_j= u_j-v_j,$ where
 $\LM(F_j)= u_j > v_j$, and $u_j, v_j \in X^m$, for some $m\geq 2$. Thus, we are in the setting of (\ref{eq:Grbasibinomial}).
As usual, we denote by $\cN = \cN(I)= \cN(G)$ the set of normal monomials modulo $I$.
Due to the form of $\Re$, we also have a monoid
\[S= \langle X \mid x_iy_i = a_ib_i,\  {\rm if}\   x_iy_i - a_ib_i\in \Re,  1\leq i \leq M\rangle\]
with $A\cong \k S$, the monoid algebra. As explained after Remark~\ref{rmk:diamondlemma}, the product $f\bullet g:= \Nor(fg), \forall f,g \in \k \cN$ restricts in our case to a monoid $(\cN, \bullet) \cong S$ so that $A\cong \k\cN$ its monoid algebra.

\emph{We fix an integer $d\geq 2$.} The set $\cG^{s}$ of all elements of $\cG$  of degree $\leq s$ can be found inductively  (using a standard strategy for constructing a Gr\"{o}bner
basis) for every $s\leq 2d$. Consequently the set of normal monomials  can be found
inductively for all $1 \leq s \leq 2d$.
Note that there is an equality of sets
\[\cN(I)_s = (\cN(\cG^{2d}))_s ,\quad \forall  s \leq 2d. \]
Moreover,
for every monomial $w \in \asX$ of length $s=|w|\leq 2d$ the normal forms modulo $I$ and modulo $\cG^{2d}$ respectively, coincide,
\[\Nor_I(w)= \Nor _{\cG^{2d}}(w),\]
and (once we know $\cG^{2d}$) the normal form $\Nor _{\cG^{2d}}(w)$ can be found
explicitly. In this case we use notation
\begin{equation}
\label{eq:normalwords33}
\begin{array}{c}
\Nor (w): = \Nor _{\cG^{2d}}(w), \quad\forall w \in \asX,\quad  |w|\leq 2d.
\end{array}
\end{equation}

\begin{thm}\label{thm:vergen}
Let
$A= \k \asX /(\Re)$ be a quadratic algebra presented by a reduced set of relations of the form (\ref{Refi}). Suppose $\cN_{d}$ has order $p$ and enumerate its elements in lexicographic order,
\[\cN_d= \{v_1 < v_2< \cdots < v_p\}.\]
\begin{enumerate}
\item[(1)] The $d$-Veronese subalgebra $A^{(d)}$ of $A$ is a quadratic algebra with a set of one-generators $\cN_d$ and a finite presentation $A^{(d)}= \k\langle v_1,v_2, \cdots,  v_p\rangle/ (\cR^d)$ with
\begin{equation}
\cR^d =\{v_iv_j - v_i^{\prime} v_i^{\prime} \mid  v_i, v_j \in \cN_d, \; v_iv_j \notin \cN_{2d},\; v_i^{\prime}, v_i^{\prime}\in \cN_{d}, \; v_i^{\prime} v_i^{\prime} =\Nor(v_iv_j) \}.
\end{equation}

\item[(2)] The set $\cR^d$ is linearly independent in $\k \cN_d^2$
and forms a basis of the graded component $J_2$ of the two sided ideal $J= (\cR^d)$ of
$\k \langle \cN_d \rangle$.

\item[(3)] If the set of relations $\Re$ forms a Gr\"{o}bner basis of the ideal $I= (\Re)$ w.r.t. deg-lex ordering on $\asX$, then
    the $d$-Veronese subalgebra $A^{(d)}$ is a PBW algebra with PBW generators $v_1 < \cdots < v_p$.
\end{enumerate}
\end{thm}
\begin{proof} (1) and (2). We know from \cite[Prop.~2.2, Chap.~3]{PoPo}
that  $A^{(d)}$ is a quadratic algebra. Here  $A^{(d)}_1 = A_d=\k \cN_d$ in our case, and we can take the stated $p$ linearly independent elements of $\cN_d$ as a set of one generators of $A^{(d)}$. We can therefore write $A^{(d)} = \k\langle \cN_d\rangle/ J$ for a two sided ideal $J=(J_2)$ generated by homogeneous polynomials of degree $2$ in $v_1, \cdots, v_p$.

We prove that each of the elements of $\cR^d$ as stated is a quadratic relation for $A^{(d)}$. Observe first that each of the elements $f_{ij}= v_iv_j - v_i^{\prime} v_i^{\prime} \in \cR^d$ is identically zero in $A$. Indeed, in the definition of $\cR^d$ we have $v_i^{\prime} v_i^{\prime}= \Nor(v_iv_j)$ and since the equality $v_iv_j = \Nor(v_iv_j)$ always hold in $A$, one has
 $f_{ij}=  v_iv_j - v_i^{\prime} v_i^{\prime}= v_iv_j - \Nor(v_iv_j)= 0$ in $A$. Hence $\cR^d \subseteq J_2$.

Moreover, since the set $\cN_{2d}$ is a basis for $A_{2d} = A^{(d)}_2 $
we have $\dim  A^{(d)}_2 = |\cN_{2d}|$. We now apply Lemma~\ref{lem:pre}, with the algebra $A$ there replaced by $A^{(d)}$ and $X,\Re$ there replaced by $\cN_d,\cR^d$ respectively. We set
\[ \textbf{N}:=\cN_{2d} = \{v_iv_j\mid v_i, v_j \in \cN_d, \quad \Nor_I(v_iv_j) = v_iv_j\}\]
\[\textbf{W}:= \cN_d^2 \setminus \textbf{N}= \{v_iv_j \mid  v_i,v_j \in \cN_d, \; v_iv_j \notin \cN_{2d}\} \]
and set $N(uv): = \Nor(uv)$ for all $u, v \in \cN_d, uv \in \textbf{W}$. Then we have
\[\cR^d = \{v_iv_j - N(v_iv_j)\mid v_i, v_j \in \cN_d,\   v_iv_j \in\textbf{ W}\}\]
and by Lemma~\ref{lem:pre}, the elements of $\cR^d \subseteq \k \cN_d^2 $ are linearly independent. Moreover,
\[\dim J_2 = p^2 - \dim  A^{(d)}_2 = |\textbf{W}| = |\cR^d|\]
and hence $\cR^d$ is a basis of $J_2$. Since $A^{(d)}$ is quadratic, its ideal of relations $J$
is generated by $J_2$ and therefore $J= (\cR^d)$ and $A^{(d)}$ has the stated presentation.

(3)  Assume now  $\Re$ forms a Gr\"{o}bner basis of the ideal $I= (\Re)$ w.r.t. deg-lex ordering on $\asX$ (so that $A$ is PBW). Then \cite[Prop.~4.3, Chap.~4]{PoPo} implies  that the set $\cN_d$ ordered lexicographically
is a set of PBW-generators of the Veronese subalgebra $A^{(d)}$.
\end{proof}

This result can be applied in the YB algebra of a braided set $(X,r)$ as in Proposition~\ref{pro:main1}. We focus on the case of $(X,r)$  left-nondegenerate idempotent, where we know by Theorem~\ref{thm:main1} that the case in part  (3) above also applies. This gives us the following corollary as preparation for our main theorem in Section~\ref{sec:vermain}.

\begin{cor}
\label{cor:VerIdemp1}
Let $(X,r)$ be a finite left nondegenerate idempotent solution on $X= \{x_1, \cdots, x_n\}$, $n\ge 2$ with YB algebra
$\cA= \k \asX /(\Re)$ as presented in Theorem~\ref{thm:main1} with relations  (\ref{eq:rels1}). Let $\cN_d$ be  the set of normal words of length $d$ enumerated in lexicographical order,
\begin{equation}
\label{eq:V}
 \cN_d = \{v_1 = x_1^d < v_2= x_1^{d-1}x_2< \cdots <x_1^{d-1}x_n   \}.
\end{equation}
\begin{enumerate}
\item
The $d$-Veronese subalgebra $\cA^{(d)}$ of $\cA$ is a quadratic algebra with $\cN_d$ a set of one-generators and a finite presentation $\cA^{(d)}= \k\langle v_1,v_2, \cdots,  v_n\rangle/ (\cR^d)$ with
\begin{equation}
\label{eq:relgen_ver}
\cR^d =\{v_iv_j - v_1 v_{p_{ij}} \mid i\geq 2,\  1 \leq i,j, \leq n\},
\end{equation}
a set of exactly $n(n-1)$ linearly independent relations. Here $p_{ij}$ is defined by $\Nor(v_iv_j)=v_1 v_{p_{ij}}$.
\item
The set $\cR^d$ is a Gr\"{o}bner basis of the ideal $J= (\cR^d)$ w.r.t. deg-lex ordering on the free monoid $\langle v_1, \cdots, v_n\rangle$.
\end{enumerate}
\end{cor}
\begin{proof}
(1) By Theorem~\ref{thm:main1},  we know that the set $\cN_{2d}$ has the form
\[
\cN_{2d} = \{x_1^{2d-1}x_k = v_1v_k\mid 1 \leq k \leq n  \}
\]
from which Theorem~\ref{thm:vergen}  straightforwardly implies part (1).

(2)  Assume that $\cR^d$ is not a Gr\"{o}bner basis of the ideal $J=(\cR^d)$ of  $\k\langle v_1,v_2, \cdots,  v_n\rangle$. Then $\dim \cA^{(d)}_3 < n$ or equivalently $\dim J_3> n(n-1)$. (Here the grading is by length of the elements of $\langle v_1,v_2, \cdots,  v_n\rangle$).
However, by the definition the $d$-Veronese subalgebra $\cA^{(d)}$ its  graded components are $\cA^{(d)}_m =\cA_{dm}$ hence
\[  \dim \cA^{(d)}_m =\dim \cA_{dm}= n, \quad \forall m \geq 1,\]
so an inequality $\dim \cA^{(d)}_3 < n$ is impossible.
\end{proof}

\subsection{Restricted solutions $(S_d,r_d)$ and the normalised braided monoid $(\cN,\rho)$.}
\label{sec:mon}

We first recall the basic facts and notations that we will need from \cite[Thm. 3.6, Thm. 3.14]{GIM08}, where it was shown that the monoid $S=S(X,r)$ itself carries a solution $r_S$ if  $(X,r)$ is a braided set. The key steps in that construction are as follows.

\begin{enumerate}
\item
The left and
right actions $X\times X\to X$ given by $(x,y)\mapsto {}^x y$ and $(x,y)\mapsto x^y$
for all $x,y\in X$ as defined via $r$ can be extended in a unique way to left and
right actions $S\times S\to S$ denoted $(a,b)\mapsto {}^ab$  and $(a, b) \mapsto  a^b$ for all $a,b$
such that
\begin{equation}
\label{eq:braided_monoid}
\begin{array}{lclc}
{ML0 :}\quad & {}^a1=1,\quad  {}^1u=u,\quad &{MR0:} \quad &1^u=1,\quad a^1=a,
\\
 {ML1:}\quad& {}^{(ab)}u={}^a{({}^bu)},\quad& {MR1:}\quad  & a^{(uv)}=(a^u)^v,
 \\
{ML2:}\quad & {}^a{(u v)}=({}^au)({}^{a^u}v),\quad &{MR2:}\quad &
(a b)^u=(a^{{}^bu})(b^u),\\
{M3:}\quad &{}^uvu^v=uv& &
\end{array}
\end{equation}
hold in $S$ for all $a, b, u, v \in S$.  M3 are the {\em braided-commutativity relations} and ML2/MR2 are the {\em matched pair property}.

\item These actions define a  map \[r_S: S\times S
\longrightarrow S\times S, \quad  r_S(a, b) := ({}^ab, a^vb)\]
which obeys the Yang-Baxter equation and restricts to $r$ in degree 1.  Here, $r_S$ is  bijective \emph{iff} $r$ is, and (say) left nondegenerate {\em iff} $r$ is.
\item The construction of the actions is by induction on the degree, hence they respect  the grading $S = \bigsqcup_{d\in\N_0}  S_{d}$  in (\ref{eq:Sgraded}),
\begin{equation}
\label{eq:braided_monoid2}
|{}^ab|= |b|= |b^a|,\quad  \forall   \; a,b \in S.
\end{equation}
\end{enumerate}

Clearly, $r_S$  immediately restricts to a solution $r_d: S_d\times S_d\to S_d\times S_d$  for all $d\ge 1$. It also follows by restriction that  $r_d$ is left nondegenerate {\em iff} $r$ is. These restricted solutions $(S_d,r_d)$ are connected with the $d$-Veronese subalgebra of $\cA(\k,X,r)$ as we will see later, building on ideas in \cite{GI_Veronese}. It is known that $r_S$ is not necessarily idempotent even if $r$ is (instead it obeys a cubic property $(r_S)^3 =r_S$ in \cite[Prop.~2.2]{Colazzo22})  but we now show that every $r_d$ is.

\begin{lem}\label{lem:rdidemp} Let $(X,r)$ be an idempotent braided set. Then
\[ (1)\quad{}^{ {}^x(yz)}(x^{yz})={}^{xy}z,\quad  (2)\quad {}^{ {}^{xy}z}((xy)^z)= {}^x(yz),\quad (3)\quad ({}^{xy}z)^{(xy)^z}=x^{yz}, \quad (4)\quad ({}^x(yz))^{x^{yz}}=(xy)^z,\]
\[(5)\quad          {}^{{}^{xy}(zt)}((xy)^{zt})={}^{xy}(zt),\quad (6)\quad  \Big( {}^{xy}(zt)\Big)^{(xy)^{zt}}=(xy)^{zt}\]
hold in $S_2$ and $(S_2,r_2)$ is idempotent.
\end{lem}
\begin{proof} We assume the conditions {\bf pr} stated in Remark~\ref{rmk:YBE1} for $r$ to be idempotent and only prove (2)-(3) (parts (1),(4)  proceed similarly by symmetry and we omit them). We have
\begin{align*}
{}^{ {}^{xy}z}((xy)^z)&={}^{ {}^{xy}z}((x^{{}^yz})y^z)=({}^{ {}^x({}^yz)}(x^{{}^yz}))({}^{ ({}^x({}^yz))^{x^{{}^yz}}} (y^z))=({}^{xy}z)({}^{x^{{}^yz}   }(y^z))={}^x({}^yz y^z)={}^x(yz)\\
({}^{xy}z)^{(xy)^z}&=\big(({}^x({}^yz))^{x^{{}^yz}}\big)^{y^z}= (x^{{}^yz})^{y^z}=x^{{}^yz y^z}=x^{yz}
\end{align*}
using the matched pair properties, the relations in $S_2$ and (for 3rd and 2nd equalities respectively) the assumptions {\bf pr}.  Using these results, we prove (6),
\begin{align*}\Big( {}^{xy}(zt)\Big)^{(xy)^{zt}}&=\Big( {}^{xy}z {}^{(xy)^z}t\Big)^{((xy)^z)^t}=\Big(  ({}^{xy}z)^{({}^{ {}^{(xy)^z}t  }  ((xy)^z)^t     )}\Big) ({}^{(xy)^z}t)^{((xy)^z)^t}=\Big(  ({}^{xy}z)^{({}^{x^{{}^yz}}(y^z t) ) }\Big) (x^{{}^yz})^{y^z t}\\
&=\Big( {}^x({}^yz)(x^{{}^yz}\Big)^{y^z t}=(x({}^yz))^{y^z t}=\Big( (x^{{}^{{}^yz}(y^z)}) ({}^y z)^{y^z}\Big)^t=\Big((x^{{}^yz})y^z\Big)^t=(xy)^{zt}
\end{align*}
where the first equality is the match pair property, the 2nd equality is the matched pair property again from the other side. The third quality requires some explanation: we expand $(xy)^z=(x^{{}^y z})( y^z)=x'y'$ by the matched pair property and apply parts (2)(3) with $x',y',t$ in the role of $x,y,z$ there. The 4th equality is the matched pair property for the right action of $y^zt$. The 5th equality is the relations of $S_2$, the 6th is the matched pair identity for the right action of $y^z$. The 7th equality is {\bf pr} and the 8th is the match pair identity for the right action of $z$. By symmetry we also have the other side
(5), where at the 3rd equality where we expand ${}^y(zt)=({}^yz)( {}^{y^z}t)=z't'$ and apply (1),(4) with $x,z',t'$ in the role of $x,y,z$. But (5),(6) are the conditions {\bf pr} for $(S_2,r_2)$, so this  is idempotent. \end{proof}

\begin{thm}\label{thm:veridemp} If $(X,r)$ is an idempotent braided set then so are $(S_d,r_d)$ for all $d\ge 1$.
\end{thm}
\begin{proof} We proceed by induction following the same steps as in the  Lemma~\ref{lem:rdidemp}. Assume that we have {\bf pr} for $(S_{d-1},r_{d-1})$ and consider $x,z\in S_{d-1}$, $y\in X$ in the proofs of (2),(3). Then $x, {}^yz\in S_{d-1}$ and we can apply $\bf pr$ for $(S_d,r_d)$ at the 3rd and 2nd equalities respectively. Similarly when looking at the proofs of (1),(2) we will have $x^y,z\in S_{d-1}$ and can apply $\bf pr$ for $(S_{d-1},r_{d-1})$ at the corresponding points for these also. The matched pair and action properties used in the proof of the lemma automatically hold for all degrees and the $S_2$ relations likewise hold in all degrees as the M3 or braided-commutativity property. Hence (1)-(4) all hold for $x,z\in S_{d-1}$ and $y\in X$.

We similarly consider the proof of (6) assuming now that $y,z\in X$ and $x,t\in S_{d-1}$. Then for the 3rd equality we have $y'\in X$ and $x',t\in S_{d-1}$ which can be used in the extended version of (2)-(3) with $x',y',t$ in the role of $x,y,z$. The 7th equality is still the original $\bf pr$ since $y,z\in X$. It follows that (6) holds for all $xy$ and $zt$ of degree $S_d$. Similarly considering the proof for (5), we have for  the 3rd equality that $z'\in X$ and $x,t'\in S_{d-1}$ so we can use the extended version of (1)-(2) with $x,z',t'$  in the role of $x,y,z$. Hence $\bf pr$ holds for $(S_d,r_d)$.
\end{proof}

Next, given any finite braided set $(X,r)$ and $S=S(X,r)$, we can transfer restricted solution $(S_d, r_d)$ over to an isomorphic solution
$(\cN_{d},
\rho_d)$, where we know from (\ref{eq:Nm}) that $\cA_d=\k S_d\cong \k \cN_d$ and that the latter is  implemented by a bijection ${\rm Nor}:S_d\to \cN_d$.  The fact that $\cN_{d}$ is ordered lexicographically makes this
solution convenient for our description of the  relations of the $d$-Veronese subalgebra, but note that as $\cN_{d}$ is a subset of the
set of normal monomials $\cN$, this description will depend on the initial enumeration of $X$.

In practical terms, it is useful to note that the construction of the left and right actions of $S$ in \cite{GIM08} also gives monomials ${}^ab\in X^q$ and $a^b$ in $X^p$ associated to any monomials $a = a_1a_2 \cdots a_p \in X^p$ and $b=
b_1b_2\cdots b_q \in X^q$, as defined inductively by
\begin{equation}
\label{eq:effective_actions}
\begin{array}{lll}
{}^c{(b_1b_2\cdots b_q)}=({}^cb_1)({}^{c^{b_1}}b_2) \cdots ({}^{(c^{({b_1}\cdots
b_{q-1})})}b_q)),\quad \forall\ c \in X,\\
&&\\
{}^{(a_1a_2 \cdots a_p)}b= {}^{a_1}{({}^{(a_2 \cdots a_p)}b)}.
\end{array}
\end{equation}
and similarly for the right action. We then need the following two facts from \cite[Lemma 4.7]{GI_Veronese}, where details can be found. Suppose $a, a_1\in X^p$ and $b, b_1 \in X^q$ with $a_1 = a, b=  b_1$ in $S$. Then:
\begin{enumerate}
\item The following are equalities of words in the free monoid $\asX$:
\begin{equation}
  \label{eq:Noractions}
  \begin{array}{ll}
  \Nor ({}^{a_1}{b_1}) = \Nor ({}^{a}{b}), & \Nor ({a_1}^{b_1}) = \Nor
  ({a}^{b}).\\
  \Nor ({}^{a}{b}) = \Nor ({}^{\Nor(a)}{\Nor(b)}),& \Nor(a^b)= \Nor
  ({\Nor(a)}^{\Nor(b)}).
  \end{array}
\end{equation}
In particular, ${}^{a_1}{b_1} ={}^ab$ and $a_1^{b_1}= a^b$ in $S$.
 \item  The following are equalities in the monoid $S$:
 \begin{equation}
  \label{eq:M3Noractions}
  ab = {}^{a}{b}a^b =  \Nor ({}^{a}{b})\Nor ({a}^{b}).
  \end{equation}
\end{enumerate}
Given these facts we have that the following is well-defined:

\begin{dfn}
\label{def:rho}
Define  left  and right `actions' on $\cN$ and the map $\rho$ by
\begin{equation}
  \label{eq:actions}
\la : \cN \times \cN \longrightarrow \cN,\quad  a \la  b := \Nor
({}^ab);\quad   \ra : \cN \times \cN \longrightarrow \cN,\quad   a \ra  b := \Nor(a^b),\end{equation}
 \begin{equation}
  \label{eq:rho}
\rho: \cN \times \cN \longrightarrow \cN \times \cN,
 \quad  \rho(a,b) := (a \la  b, a \ra  b).
\end{equation}
for all $a, b \in \cN$.  We also define the restriction $\rho_d=\rho|_{\cN_d\times\cN_d}$ as a map $\rho_d: \cN_{d} \times \cN_{d} \longrightarrow \cN_{d} \times
\cN_{d}$.
\end{dfn}

One can verify that  $(\cN,\rho)$ is a solution of the YBE isomorphic under $\Nor:S\to \cN$ to $(S,r_S)$. Likewise $(\cN_d,\rho_d)$ is a solution isomorphic to $(S_d,r_d)$.  More details are in  \cite[Prop. 4.10]{GI_Veronese}, stated there in the involutive case but applying for any braided set.  In line with that work, we call $(\cN,\rho)$ the {\em normalised braided monoid} and $(\cN_d,\rho_d)$ the  {\em $d$-Veronese solutions} associated to $(X,r)$ and an enumeration of $X$.  These normalised versions are better adapted for the results in the next section.

\subsection{Isomorphism theorem for $\cA(\k,X,r)^{(d)}$ in the left nondegenerate idempotent case}\label{sec:vermain}

Let $A=\cA(\k,X,r)$ and in this section let $(X,r)$ be left-nondegenerate idempotent and $X=\{x_1,\cdots,x_n\}, n\ge 2$. It follows from Theorem~\ref{thm:veridemp} and other facts in the preceding section that in this case $(\cN_d,\rho_d)$ is another left-nondegenerate idempotent solution, and meanwhile the enumeration by elements $v_i=x_1^{d-1}x_i$ in  (\ref{eq:V}) gives a natural identification of $\cN_d$ with $X$ as enumerated sets by $v_i\leftrightarrow x_i$.

\begin{dfn} Given a left-nondegenerate braided set $(X,r)$, we define its {\em Veronese prolongation} as the sequence of left-nondegenerate idempotent braided sets $(X,r^{(d)})$ corresponding to $(\cN_d,\rho_d)$ under the identification $x_i\leftrightarrow v_i$. By convention we understand $r^{(1)}=r$ itself.
\end{dfn}

\begin{thm}
\label{thm:ver}
  Let  $(X,r)$ be a finite left-nondegenerate idempotent solution on $X= \{x_1, \cdots, x_n\}$, $n,d\ge 2$. Then $\cA(\k,X,r)^{(d)}$ in Corollary~\ref{cor:VerIdemp1} can be identified with $\cA(\k,X,r^{(d)})$ as PBW algebras. \end{thm}
\begin{proof}
We let $\cA=\cA(\k,X,r)$. By Corollary~\ref{cor:VerIdemp1}, we can write  $d$-Veronese  subalgebra $\cA^{(d)}$ as a quadratic algebra with a set of one generators $\cN_d$ ordered lexicographically, and a set of quadratic relations in $v_1, \cdots, v_n$. Meanwhile, we let  $B = \cA(\k, \cN_d, \rho_d)$ and Theorem~\ref{thm:main1} gives us its standard finite presentation
\begin{equation}
\label{eq:relsB}
B = \k\langle v_1, \cdots, v_n\rangle/(\Re_B);\quad \Re_B = \{F_{ij}= v_iv_j - v_1v_{q_{ij}}\ |\  2 \leq i \leq n, 1 \leq j \leq n\},
\end{equation}
where the indices $k_{ij}$ in (\ref{eq:rels1}) are now denoted $q_{ij}$. We know that there are $n$ such $\rho_d$-orbits $\cO_i, 1\le i\le n,$ in $\cN_d\times\cN_d$ and that $q_{ij}$ are determined by  $v_iv_j \in \cO_{q_{ij}}$.

We shall prove that every polynomial $F_{ij}= v_iv_j - v_1v_{q_{ij}}\in \Re_B$ is identically $0$ in the algebra $\cA$, and therefore in $\cA^{(d)}\subseteq \cA$ when we view $v_i\in \cA$ according to  $v_i= x_1^{d-1}x_i$ for all $1 \leq i\leq n$. Here, $q_{ij}$ is the unique $q$ such that $v_iv_j$ is in the same $\rho_d$-orbit as $v_iv_q$. As in Section~\ref{sec:idemp} applied in our case, this happens if and only if we have  equality of the two elements of $\cN_d\times\cN_d$,
\[
\begin{array}{lll}
\rho_d(v_i,v_j)&=& (v_i\la v_j, v_i\ra v_j) = (\Nor({}^ {v_i}{v_j}),\Nor(v_i^{v_j})),\\
\rho_d(v_1,v_q)&=& (v_1\la v_q, v_q\ra v_1) = (\Nor({}^ {v_1}{v_q}),\Nor(v_1^{v_q})).
\end{array}
\]
Comparing each component, this amounts to a pair of equalities  in $\cN_q$ and hence in $\langle X\rangle$.
But $\Nor(w) = w$ holds automatically $S$ (and in $\cA$) for all $w \in\asX$, so $ {}^ {v_i}{v_j} = {}^ {v_1}{v_q}$ and $v_i^{v_j}= v_1^{v_q}$ hold in $S$. In this case, working in $S$ and using the matched pair (M3) identity, we have that
\[ v_iv_j= ({}^ {v_i}{v_j})(v_i^{v_j})= ({}^ {v_1}{v_q})(v_1^{v_q})=v_1v_q\]
holds in $S$ and hence in $\cA$. Thus,  $F_{ij}= v_iv_j -v_1v_q = 0$ in $\cA$ and hence in $\cA^{(d)}$.

Moreover, the normal form $\Nor (v_iv_j)$ modulo $I$, considered in $\k\asX$, gives that $v_iv_j =\Nor (v_iv_j) = v_1v_{p_{ij}}$ are equalities in $S$ and in $\cA$, for some $p=p_{ij}$ as in Corollary~\ref{cor:VerIdemp1}.  Hence $v_1v_q= v_iv_j= v_1v_p$ holds in $S$. Hence,  $x_1^{2d-1}x_q=x_1^{2d-1}x_p$  which is an equality of two normal words in $\cN_{2d}$. Hence, $q_{ij}=p_{ij}$. Hence $\cA(\k,\cN_d,\rho_d)$ and $\cA^{(d)}$ as generated by the $v_i$ coincide and the algebras are isomorphic with the same presentations. Finally, we identify $x_i\leftrightarrow v_i$ to give the equivalent statement for $(X,r^{(d)})$. \end{proof}

We recall that in algebraic geometry a $d$-Veronese morphism means an algebra homomorphism $B\to A$ with image $A^{(d)}$, where $B$ is of the same type of graded algebra as $A$. In our case we have such a map
\[ \nu_d: \cA(\k,\cN_d,\rho_d)\to  \cA(\k, X,r),\quad  \nu_d(v_i)=x_1^{d-1}x_i\]
where on the left $v_i$ is as an element of the set $\cN_d$ and on the right it is the corresponding normal word. This follows the same form as in the involutive case in \cite{GI_Veronese}, except that in our case the Veronese map $\nu_{d}$ is injective while in the involutive case it has a large kernel. We were also able in our case to refer everything uniformly back to $(X,r^{(d)})$ which is not possible in the involutive case. We conclude with some examples.

\begin{ex}
\label{ex:N2} Let $(X, r)$ be the solution on the set $X=\{x_1, x_2, x_3\}$ given in Example~\ref{ex2}. For $d=2$ one has
\[\cN_2 = \{v_i = x_1x_i\ |\ 1\le i\le 3\}\]
and we calculate
\[ {}^{x_1x_i}(x_1 x_j)=({}^{x_1 x_i}x_1)({}^{x_1x_1}x_j)=({}^{x_1x_i}x_1)x_j=({}^{x_1}x_i)x_j\]
from the form of the actions for $(X,r)$. This results in
\[
\mathcal{L}_{v_1}= \id_{\cN_2},\quad \mathcal{L}_{v_2}= (v_1 \; v_2\;v_3),\quad \mathcal{L}_{v_3}= (v_1 \; v_3\; v_2) = \mathcal{L}_{v_2}^{-1}
\]
and all right actions giving $v_1$. These give the  $2$-Veronese solution $(\cN_2, \rho_2)$,
\[
\begin{array}{lll}
\rho_2(v_3v_3) = v_2v_1, & \rho_2(v_1v_2) = v_2v_1, & \rho_2(v_2v_1) = v_2v_1, \\
\rho_2(v_3v_2) = v_1v_1, & \rho_2(v_2v_3) = v_1v_1, & \rho_2(v_1v_1) = v_1v_1, \\
\rho_2(v_2v_2) = v_3v_1, & \rho_2(v_1v_3) = v_3v_1, & \rho_2(v_3v_1) = v_3v_1. \\
\end{array}\]
The $2$-prolongation $(X,r^{(2)})$ looks the same with $x_i$ in place of $v_i$ and is not isomorphic to the original $(X,r)$ due to the very different form of the left actions.

The associated Yang-Baxter algebra $B = \cA(\k, \cN_2, \rho_2)$ has the following relations $\Re_B$ provided by  Theorem~\ref{thm:main1}.
\[
\begin{array}{lll}
 \Re_B &=& \{ f_{33}= v_3v_3 - v_1v_2, \;f_{21}  = v_2v_1 -v_1v_2\\
      & &   f_{32}= v_3v_2  -v_1v_1,  \; f_{23}= v_2v_3 -v_1v_1\\
      &&    f_{31} = v_3v_1- v_1v_3,  \;f_{22}= v_2v_2- v_1v_3\}
 \end{array}
 \]
according to the three $\rho_2$-orbits. It follows from Theorem~\ref{thm:main1} that the relations $\Re_B$ form a reduced Gr\"{o}bner basis of the two-sided ideal
$J= (\Re_B)$ of $\k\langle v_1, v_2, v_3\rangle$ with respect to the deg-lex order on $\langle v_1, v_2, v_3 \rangle$
extending $v_1 < v_2 < v_3$.
The normal $\k$-basis of $B$ is
\[\cN =  \{v_1^i, v_1^iv_2, v_1^iv_3 \mid i\geq 0\},\]
where $x_1^0=1$. It is clear that the two PBW algebras $\cA= \cA (\k, X, r)$  and $B = \cA(\k, \cN_2, \rho_2)$ are not isomorphic, in contrast to the isomorphism of the latter with $\cA^{(2)}$ by our general results.

For higher $(\cN_d,\rho_d)$, we let $v_i=x_1^{d-1}x_i$ then
\[
{}^{x_1 v_i}(x_1 v_j)=({}^{x_i v_i}x_1^d)({}^{x_1^{d+1}}x_j)={}^{x_1}({}^{v_i} x_1^d)({}^{x_1^{d+1}}x_j)={}^{x_1}(x_1^{d-1}x_i)({}^{x_1^{d+1}}x_j)=x_1^{d-1}({}^{x_1} x_i)\begin{cases} {}^{x_1}x_j & d\ {\rm even}\\ x_j & d\ {\rm odd,}\end{cases}\]
where we use the matched pair properties and that the right action on $X$  always produces $x_1$. For the 3rd equality that ${}^{v_i}v_1=v_i$, which is proven by induction, being true for both $r$ and $r^{(2)}$. As the end we use that $x_1$ left acts on $X$ as a transposition. In the odd case we have the same expression aside from $x_1^{d-1}$ out front as for the $r^{(2)}$ calculation above. For the even case, one can show by the matched pair and M3 properties that
\[ ({}^{x_1}x_i)({}^{x_1}x_j)={}^{x_1}(x_ix_j)={}^{x_1}(({}^{x_i}x_j)x_1)=x_1({}^{x_i}x_j)\]
so that aside from $x_1^{d-1}$ out front, we have the left action for $r$. The right action on $v_i$ is always  $v_1$, which is as for both solutions. Hence, the prolongation sequence $(X,r^{(d)})$ alternates between $r$ and $r^{(2)}$. \end{ex}

For further examples, we turn to the  permutation idempotent solutions $(X,r_f)$ of order $n$, where $f\in{\rm Sym}(X)$ is a permutation. We already know (as proven more formally in Corollary~\ref{cor:normalforms}) that every element of $S=S(X,r_f)$ of degree $d$ can be put in a standard form
\[ y_1\cdots y_{d-1} x_i= f^{d-1}(x_i)\cdots f(x_i)x_i = {x_1}^{d-1}x_i=v_i\]
for all $y_i\in X$ (here, we simply derived this by iterating the relation $xy={}^x y x^y=f(y)y$ in $S$, as ${}^xy=f(y)$ and $x^y=y$ for all $x,y\in X$). The
 map $\Nor: S\to \cN$ sends a general element of $S_d$ to the same element of the algebra but in normal form in $\cN_d$. We enumerate its  distinct elements as $\{v_1,\cdots,v_n\}$.

\begin{lem}
\label{lem:rhoperm} Let $(X,r_f)$ be a permutation idempotent solution with $X = \{x_1, \cdots, x_n\}$. The associated monoid $(\cN,
\bullet)$ is a graded braided monoid with a braiding operator
\begin{equation}
\label{eq:broperator}
\rho : \cN\times \cN \longrightarrow \cN\times \cN,\quad \rho (x_1^{d-1}x_i, x_1^{m-1}x_j)= (x_1^{m-1} f^d(x_j),
x_1^{d-1}x_j),\quad \forall\ d, m \geq 2.
\end{equation}
Here, $(\cN,  \rho)$ is a left nondegenerate solution and  $\rho^3=\rho$, but $\rho^2\neq  \rho$ if $n\ge 2$. Moreover the Veronese prolongation is $(X,r^{(d)})=(X,r_{f^d})$ as braided sets.
\end{lem}
\begin{proof} Iterating $x^y=y$, we have immediately  for all $x,z_i\in X$,
\[  x^{z_1\cdots z_{m-1} x_j}=(x^{z_1})^{z_2\cdots z_{m-1} x_j}=(z_1)^{z_2\cdots z_{m-1} x_j}=\cdots =z_{m-1}^{x_j}=x_j\]
and by the matched pair condition
\[ (y_1\cdots y_{d-1}x_i)^{z_1\cdots z_{m-1} x_j} = (y_1\cdots y_{d-1})^{{}^{x_i}(z_1\cdots z_{m-1} x_j)} x_i^{z_1\cdots z_{m-1} x_j} = t_1
\cdots t_{d-1} x_j\]
for some elements $t_1,\cdots,t_{d-1}\in X$ since the right action preserves degree of a monomial.

Similarly iterating ${}^xy=f(y)$, we have for all $x,y_i\in X$,
\[ {}^{y_1\cdots y_{d-1}x_i}x= f^{d}(x)\]
and hence using the matched pair conditions
\[ {}^{y_1\cdots y_{d-1}x_i}(z_1\cdots z_{m-1}x_j)=({}^{y_1\cdots y_{d-1}x_i}(z_1\cdots z_{m-1}))({}^{(y_1\cdots y_{d-1}x_i)^{z_1\cdots z_{m-1}}}x_j)=s_1\cdot s_{m-1}f^d(x_j)\]
for some elements $s_1,\cdots,s_{m-1}\in X$ since the left action preserves degree of a monomial.

We apply these formulae to the normal forms on $\cN_d,\cN_m$ to give the formulae as stated for $\rho$. It is easy to check from these formulae that $\rho^3=\rho$, but $\rho^2\neq\rho$, as remarked for $(S,r_S)$ in \cite{Colazzo22}. Moreover, setting $m=d$ and restricting to $\cN_d$, we have
\[ \rho_d(v_i,v_j)=(f^d(v_j),v_j)\]
where $f^d(v_j):=x_1^{d-1} f^d(x_j)$ is the elements of $\cN_d$ that corresponds to $f^d(x_j)$. Hence $(\cN_d,\rho_d)$ is isomorphic under the identification stated with $(X,r_{f^d})$,  where $r_{f^d}(x_i x_j)=f^d(x_j)x_j$. \end{proof}

\begin{cor}
\label{cor:verperm}
Given $(X,r_f)$  and $d \geq 2$ an integer, the  $d$-Veronese subalgebra $\cA(\k,X,r_f)^{(d)}\cong \cA(\k, X, r_{f^d})$.
\end{cor}
\begin{proof} We apply our general result that $\cA(\k,X,r_f)^{(d)}\cong\cA(\k,\cN_d,\rho_d)$ and identify $(\cN_d,\rho_d)$  as above.\end{proof}

Note that by Section~\ref{sec:mon}, we can define a kind of Veronese prolongation for any  finite braided set $(X,r)$ by the sequence $(X,r)^{(d)}=(\cN_d,\rho_d)$, but the underlying sets in general vary with $d$. In the left-nondegenerate involutive case\cite{GI_Veronese}, the order of the sets grows rapidly with $d$. In the left-nondegenerate idempotent case we saw that we can do better and uniformly refer all of the sets back to $X$. In this case the Veronese prolongation  $(X,r^{(d)})$ must have repetitions (since $X$ is finite).

 \begin{que} (1) Can we characterise a braided set  $(X,r)$ by  properties of its prolongation sequence $(X,r)^{(d)}$? (2) In the left-nondegenerate idempotent case, what is the {\em stopping distance} defined as $d-1$ for the smallest $d$ such that the $d$-prolongation $(X,r^{(d)})$ returns to $(X,r)$  and what are the number of distinct solutions that arise?
 \end{que}

For  example, we have just seen in the permutation idempotent case that prolongation sequence, and hence the sequence of the $\cA(\k,X,r_f)^{(d)}$, is periodic according to the order of $f$.

\subsection{Segre products for left nondegenerate idempotent solutions}\label{sec:seg}

We first recall the  Segre product of graded algebras as in the text \cite[Sec.~3.2]{PoPo}.
\begin{dfn}
\label{dfn:segre}
    Let $A=\oplus_{m\in \N_0}A_m$ and $B=\oplus_{m\in \N_0}B_m$ be graded algebras with $A_0=B_0=\k 1$.
The \emph{Segre product} of $A$ and $B$ is the $\N_0$-graded algebra
    \[A\circ B:=\bigoplus_{i \geq 0}A_i\tens B_i \]
 as a subalgebra of $A\tens B$. The latter is bigraded and $A\circ B$ is the diagonally graded component.
\end{dfn}
 As well as the seminal work  of
Fr\"{o}berg and Backelin\cite{Backelin,Froberg} in the noncommutative case, more recent results include the Segre product of specific Artin-Schelter regular algebras\cite{kristel} and twisted Segre products\cite{twistedSegre}, although we do not consider the latter here.  When $A, B$ are quadratic algebras with linear subspaces of relations $R_A\subset A_1\tens A_1$ and $\cR_B\subset B_1\tens B_1$ then it is known\cite{kristel} that $A\circ B$ is a quadradtic algebra with linear subspace of relations
 \begin{equation}\label{segrel} R_{A\circ B}=\sigma_{23}( R_A \otimes B_1\otimes B_1 +A_1\otimes A_1\otimes R_B),\end{equation}
where $\sigma_{23}$ is the flip map acting in the 2nd and 3rd tensor factors (and the identity on the other tensor factors). Similarly to the notion of a $d$-Veronese morphism, a Segre morphism is an algebra homomorphism  $C\to A\tens B$ from an algebra of the same type as $A,B$ with image $A\circ B$.

\begin{lem}\label{lem:seg} We assume $\k$ is not of characteristic 2. Let $R_A$ and $R_B$ be given as the image of $\id-\Phi$ and $\id-\Psi$, respectively for linear maps $\Phi,\Psi$. Let $A*B$ be the quadratic algebra defined with the same generating space $A_1\tens B_1$ and relations $R_{A*B}$ defined as the image of $\id-\sigma_{23}(\Phi\tens\Psi)\sigma_{23}$. Then the identity map on $A_1\tens B_1$ extends to a {\em Segre morphism}
\[ \mathfrak{s}: A* B\to A\tens B\]
with image $A\circ B$ where it becomes a map of quadratic algebras.
\end{lem}
\begin{proof}  We write
\[\id-\sigma_{23}(\Phi \tens\Psi)\sigma_{23}={1\over 2}\sigma_{23}( (\id-\Phi)\tens(\id+\Psi)+(\id+\Phi)\tens (\id-\Psi))\sigma_{23}\]
from which we see that $R_{A*B}\subseteq R_{A\circ B}$. As both quadratic algebras have the same space $A_1\tens B_1$ in degree 1, this gives a well-defined surjective morphism of graded algebras $A*B\to A\circ B\subset A\tens B$. \end{proof}

In particular, our quadratic algebras of interest are of the required type so that this Lemma applies.

\begin{cor}\label{cor:seg} If $A=\cA(\k,X,r)$ and $B=\cA(k,Y,s)$ for quadratic sets $(X,r)$ and $(Y,s)$ then the Segre homorphism is
\[ \mathfrak{s}:A*B=\cA(\k, X\times Y, r*s)\to A\tens B,\quad \mathfrak{s}(z_{ia})=x_i\tens y_a,\]
where and $(X\times Y,r*s)$ is the `cartesian product'   with
\[ r*s= \sigma_{23}(r\times s)\sigma_{23}:  (X\times Y) \times (X\times Y) \to (X\times Y) \times (X\times Y).\]
\end{cor}
\begin{proof} This is immediate  by linearising $r,s$ to $\Phi,\Psi$ on $A_1=\k X$ and $B_1=\k Y$ respectively. Such linearisation is discussed more fully in Section~\ref{seclin}.\end{proof}

The term `cartesian product' here follows \cite{GI23} and indeed the set is the cartesian product while the map $r*s$ is  the cartesian product  with $\sigma_{23}$ inserted for the factors of $X,Y$ to fall in the correct place. For practical purposes, let $X=\{x_1,\cdots,x_n\}$ and $Y=\{y_1,\cdots,y_m\}$ be enumerations.  We enumerate $X\times Y$ lexicographically by elements  $z_{ia}:=(x_i,y_a)$. Then the actions underlying $r*s$ are
\[ {}^{z_{ia}}z_{yb}= ({}^{x_i}{x_j},  {}^{y_a}{y_b}),\quad   z_{ia}{}^{z_{yb}}= (x_i^{x_j}, y_a^{y_b})\]
in terms of the actions of $r$ and $s$ operating independently and in parallel. In this case it is clear that if $r,s$ obey YBE or  have other  properties defined via the left and right actions alone, then these apply also for $r*s$. In particular, if $(X,r)$, $(Y,s)$ are braided/invertible/involutive/left-nodegenerate/idempotent,  then so is their cartesian product. Most of these were already noted  in \cite{GI23}.

In our case we assume henceforth that $(X,r),(Y,s)$ are left-nondegenerate and idempotent braided sets. Then so is $(X\times Y, r*s)$ and by Theorem~\ref{thm:main1} applied to this, we know that $A*B=\cA(\k, X\times Y, r*s)$ is PBW with a canonical set of relations determined by the $r*s$-orbits in $(X\times Y)^2$.

\begin{lem} If $(X,r),(Y,s)$ are left-nondegenerate and idempotent braided sets then $z_{ia}z_{yb}, z_{kc}z_{ld}$  are in the same  $r*s$-orbit  if and only if $x_ix_j, x_k x_l$ are in the same $r$-orbit and $y_ay_b,y_cy_d$ are in the same $s$-orbit.
\end{lem}
\begin{proof} We use Lemma~\ref{lem:idemporb} applied to $(X\times Y, r*s)$. So $z_{ia}z_{yb}, z_{kc}z_{ld}$ are in the same orbit iff $r(z_{ia}z_{yb})=r( z_{kc}z_{ld})$. But from the independent form of the actions, this happens iff $r(x_ix_j)=r(x_k x_l)$ and $s(y_ay_b)=s(y_cy_d)$, which corresponds to the claim, again via Lemma~\ref{lem:idemporb}. \end{proof}

This means the graph of orbits is the product graph. There are $nm$ orbits where $\cO_{jb}$ contains $z_{11}z_{jb}$ and $w_{ia,jb}$ defined by $z_{ia}z_{jb}\in \cO_{w_{ia,jb}}$ is given by the pair of integers
\[ w_{ia,jb}=(k_{ij},l_{ab})\]
where $k_{ij}$ are the indices for $r$-orbits as in Section~\ref{sec:idemp} and $l_{ab}$ are the corresponding ones for $s$-orbits. Here, $w$ can also be characterised by ${}^{z_{11}} z_{w_{ia,jb}}= {}^{z_{ia}}z_{jb}$ and the action for the cartesian product. In this case the canonical set of relations for $\cA(\k, X\times Y, r*s)$ from Theorem~\ref{thm:main1} are
\[ F_{ia,jb}:= z_{ia}z_{jb}- z_{11}z_{w_{ia,jb}}= z_{ia}z_{jb}-z_{11}z_{k_{ij} l_{ab}},\quad \forall\ (1,1)<(i,a),\ (1,1)\le (j,b)\le (n,m). \]
As a check, it is easy to see that these vanish when we map them by $\mathfrak{s}(z_{ia})=x_i \tens y_a$ to $A\tens B$, as must be the case by Corollary~\ref{cor:seg}. Indeed, one can express the images as
\begin{equation}\label{imF}\mathfrak{s}(F_{ia,jb})=x_i x_j \tens y_a y_b  -x_1 x_{k_{ij}}\tens y_1 y_{ l_{ab}}={1\over 2}(f_{ij}\tens (y_ay_b+y_1 y_{l_{ab}})+ (x_ix_j+x_1x_{k_{ij}})\tens g_{ab})\end{equation}
where $f_{ij}$ are the canonical relations for $(X,r)$ in  Theorem~\ref{thm:main1} and $g_{ab}$ are the corresponding relations for $(Y,s)$ for $1<i,a$, while the same expressions in the cases $f_{1j}$ and $g_{1b}$ vanish automatically as $k_{1j}=j$ and $l_{1b}=b$.

\begin{thm} For left-nondegenrate idempotent braided sets $(X,r),(Y,s)$, the Segre product $\cA(\k, X, r)\circ \cA(\k, Y, s)$ can be identified as a quadratic algebra with $\cA(\k, X\times Y, r*s)$. Moreover,  it is a PBW algebra with PBW generators $\{x_i\tens y_j\ |\ 1\le i\le n,\ 1\le j\le m\}$ in lexicographic order and set of relations
\[ \Re_{A\circ B}=\{f_{ia,jb}= (x_i \tens y_a)(x_j \tens y_b)  -(x_1 \tens y_1)(x_{k_{ij}}\tens y_{ l_{ab}})\ |\ (i,a)\ne (1,1)\}\]
which form a Gr\"obner basis of $(\Re_{A\circ B})$.
 \end{thm}
\begin{proof} Since the image of the Segre morphism $\mathfrak{s}$ in Corollary~\ref{cor:seg} is the subalgebra $A\circ B$, we have
\[A*B/\ker \mathfrak{s}  \cong A\circ B.\]
Moreover, viewed as a surjection to $A\circ B$, the map $\mathfrak{s}$ respects the grading, hence
\[(A*B/\ker \mathfrak{s})_d = (A*B)_d/\ker \mathfrak{s}_d \cong (A\circ B)_d, d \geq 2.\]
But $\dim (A*B)_d = mn= \dim (A\circ B)_d$, for all $d\geq 2$, and therefore $\ker\mathfrak{s}= \{0\}$.
Hence the Segre map $\mathfrak{s}$ to $A\circ B$ is an isomorphism.  We also define $f_{ia,jb}\in J_2$ as stated, where $(J_2)=(\Re_{A\circ B})$. This is the image of $F_{ia,jb}$ in $\k\langle X\rangle_2\tens \k\langle Y\rangle_2$. Again by dimensions,
\[\dim J_2 =|X\times Y|^2 - \dim(A\circ B)_2 = (mn)(mn-1).  \]
By Theorem~\ref{thm:main1} we have exactly this many $F_{ia,jb}$  (in the present case, we excluded $(i,a)=(1,1)$). By making the same expansion as in (\ref{imF}) but this time for $f_{ia,jb}$ and working in $\k\langle X\rangle_2\tens \k\langle Y\rangle_2$, and using that $f_{ij}$ and $g_{ab}$ for $i,a>1$ are linearly independent as part of the standard bases in Theorem~\ref{thm:main1} applied to $A$ and $B$, one can show that the $f_{ia,jb}$ are likewise linearly independent. Hence they form a basis of  $J_2$ and $A\circ B$ has a standard finite presentation of the same form as that of $\cA(\k, X\times Y, r*s)$. \end{proof}

Our Segre morphism $\mathfrak{s}:\cA(\k, X\times Y,r*s)\to \cA(\k, X,r)\tens \cA(\k,Y,s)$ specialises in the involutive case to one constructed in \cite{GI23} (in other notations), but there it is has a large kernel. By contrast, we have shown in our case  that it is an isomorphism to its image $\cA(\k, X,r)\circ \cA(\k,Y,s)$, i.e. injective, which is  a similar situation to the $d$-Veronese morphism in Section~\ref{sec:vermain}.

\section{Prebraided category constructions for idempotent solutions}\label{seclin}

If  $(X,r)$ is a finite braided set, we let $V=\k X$ and extend $r$ by linearity to a map
\[ \Psi:V\tens V\to V\tens V,\quad \Psi(x\tens y)= {}^{x}y\tens x{}^{y}\]
where we identify $V\tens V=\k X\times X$ in the obvious way. On the other hand,  linear braidings when $V$ is finite-dimensional correspond to matrices as follows. If $\{x_i\}_{i=1}^n$ is a basis of a vector space $V$, we define matrices in ${\rm End}(V\tens V)$ by
\begin{equation}\label{PsiR} \Psi(x_i\tens x_j)=\sum_{a,b}x_a\tens x_b \Psi^a{}_i{}^b{}_j=\sum_{a,b}x_b\tens x_a R^a{}_i{}^b{}_j.\end{equation}
Here, the braiding matrix defined by $\Psi$ obeys the braid relations while the equivalent `R=matrix' $R^a{}_i{}^b{}_j=\Psi^b{}_i{}^a{}_j$ obeys the original YBE, see \cite{MajidQG}. As matrices in ${\rm End}(V\tens V)$, these are related by  $\Psi=P R$ for where $P^i{}_j{}^k{}_l=\delta^i{}_l\delta^k{}_j$ is the matrix for the flip map. In this section, we will only consider finite braided sets $(X,r)$ with a fixed enumeration $X=\{x_1,\cdots,x_n\}$. Then the corresponding R-matrix is
\begin{equation}\label{Rr} R^k{}_i{}^l{}_j=\delta_{x_l,{}^{x_i}x_j}\delta_{x_k, {x_i}{}^{x^j}}.\end{equation}

In this language, the YB-algebra  is just an instance of a standard R-matrix construction for `Manin-quantum planes', which we will denote $S_+(R)$, defined by
 generators $x_i$ and relations
\begin{equation}\label{xxR} x_ix_j=\sum_{a,b}x_bx_a R^a{}_i{}^b{}_j, \quad  \forall\ 1\le i,j\le n,\end{equation}
which one can check recovers $\cA(\k,X, r)$ in case (\ref{Rr}). There are, however, many other standard R-matrix constructions in the linear case and in this section we look at the main ones of these for linearizations of finite braided sets $(X,r)$, with particular interest in the idempotent case.

We will make particular use of the fact that any pair  $(V,\Psi)$ consisting of a finite-dimensional vector space and $\Psi:V\tens V\to V\tens V$ obeying the braided relations generates a prebraided category of sums and tensor products of $V$, where we write `pre' since $\Psi$ and its extensions to other objects are not assumed to be invertible. We do not include $V^*$ in this prebraided category since the braiding between $V,V^*$ would need $R$ to be `bi-invertible' \cite[Example~9.3.13]{MajidQG} and this will not be possible in the idempotent case. On the other hand, $V^*$ inherits its own braiding
\begin{equation}\label{PsiT} \Psi^T(y^i\tens y^j)=\sum_{a,b}\Psi^i{}_a{}^j{}_b y^a\tens y^b.\end{equation}
The corresponding transposed $R$-matrix is  $(R^T)^i{}_k{}^j{}_l=R^l{}_j{}^k{}_i$ such that  $P(R^T)=(PR)^T$. Hence there is a parallel theory for $V^*$.  This immediately gives a `transpose' YB-algebra associated to a set-theoretic solution

\begin{lem}\label{lem:coYB} Given a finite braided set $(X,r)$, let $R$ be the linearisation of $r$. Then $S_+(R^T)$ has generators $\{y^i\}_{i=1}^n$ and  relations
\[ \sum_{\{a,b\ |\ {}^{x_a}x_b=x_i,\ {x_a}{}^{x_b}=x_j\}} y^a y^b=y^iy^j,\quad\forall\ 1\le i,j\le n.  \]
We call this the transpose YB-algebra denoted $\cA^T(\k,X,r)$.
\end{lem}
\begin{proof} We use  (\ref{Rr}) and (\ref{Rf}) respectively in the relations $\cdot\Psi^T(y^i\tens y^j)=y^i y^j$ for all $i,j$, where $\cdot$ is the product and $\Psi^T$ is (\ref{PsiT}). \end{proof}

In the left-nondegenerate idempotent case, the sum over $b$ here gets replaced by setting $b=k_{a i}$  as in Lemma~\ref{lem:orbits}.

\begin{ex} (1) For the permutation idempotent case $(X,r_f)$, the braiding and R-matrix by linearisation are
\begin{equation}\label{Rf} \Psi(x_i\tens x_j)= f(x_j)\tens x_j,\quad R^a{}_i{}^b{}_j=\delta_{f(a),b}\delta_{a,j}\end{equation}
if one writes $f(x_j)=x_{f(j)}$ as the corresponding permutation of the indices.  Putting this into the definition of $S_+(R^T)$,  the new `transpose' YB algebra $\cA^T(\k,X,r_f)$ has the relations
\begin{equation}\label{coYBperm}  y^iy^j=0,\quad \forall i\ne f(j).\end{equation}
Thus, for  Example~\ref{ex1}, we have the 6 relations
\[ (y^i)^2=0,\quad 1\le i\le 3, \quad y^1y^2=0,\quad y^2y^3=0,\quad y^3y^1=0.  \]

(2) For Example~\ref{ex2}, we use Lemma~\ref{lem:coYB} to find for $\cA^T(\k,X,r)$ the 3 relations
\[  y^2 y^2+y^3 y^3=0,\quad y^1 y^3+y^3 y^2=0,\quad y^1 y^2+ y^2 y^3=0. \]
\end{ex}

We also note that Lemma~\ref{lem:seg} applied to Yang-Baxter algebras in the form $S_+(R)$, $S_+(S)$ gives a natural Segre morphism
\[ \mathfrak{s}: S_+(R*S)\to S_+(R)\tens S_+(S)\]
with image $S_+(R)\circ S_-(S)$, where $R*S$ is a certain `tensor product' of the matrices $R,S$ for the two quadratic algebras. Writing $\Phi, \Psi$ respectively for the corresponding braidings,  the braiding corresponding to $R*S$ is
\begin{equation}\label{tensorbraid} \Phi*\Psi= \sigma_{23}(\Phi\tens\Psi)\sigma_{23}\end{equation}
Lemma~\ref{lem:seg} in fact applies for any quadratic algebras presented in the form of $S_+(R),S_+(S)$, even if $R,S$ do not obey the Yang-Baxter relations, but when they do then it is obvious that so does $\Phi*\Psi$ (since the the proof of that proceeds independently between the two tensor factors). One can use (\ref{Rr}) and similarly for $S$ to arrive at Corollary~\ref{cor:seg} in the set-theoretic case.

\subsection{Koszul dual and Nichols algebra in the idempotent case}

As the YB-algebra is a `braided-symmetric algebra', a corresponding `braided-antisymmetric' version is $S_-(R)=S_+(-R)$ of (\ref{xxR}) with an extra minus sign. One can verify, however, that if $R$ is a linear solution of the YBE with $\Psi^2=\Psi$ then all products  $x_ix_j$ of $S_-(R)$ are zero, which is not interesting. A better candidate for this role in the R-matrix theory is the Koszul dual which in \cite{Ma91} was used to define a `fermionic quantum plane' as the Koszul dual of the usual $q$-plane in the $q$-Hecke case studied there.

\begin{lem}\label{lem:kosR} For $R$ a linear solution of the YBE with $\Psi^2=\Psi$, the Koszul dual $S_+(R)^!$  with dual basis $\{y^i\}_{i=1}^{n}$ of $V^*$ has the relations
\[  \sum_{a,b} R^j{}_a{}^i{}_b y^a y^b=0,\quad \forall\ 1\le i,j\le n.\]
In the case where $R$ is obtained from an idempotent set-theoretic solution $r$, the Koszul dual $\cA(\k,X,r)^!$ has the relations
\[ \sum_{\{a,b\ |\ {}^{x_a}x_b=x_i,\ {x_a}{}^{x_b}=x_j\}} y^a y^b=0,\quad \forall 1\le i,j\le n.\]
\end{lem}
\begin{proof}  The subspace of relations of $S_+(R)$ as a quadratic algebra is ${\rm image}(\id-\Psi)=\ker\Psi$ since $\Psi$ is idempotent. The Koszul dual therefore has the subspace of relations $(\ker\Psi)^\perp={\rm image}(\Psi^T)\subseteq V^*\tens V^*$ where $\Psi^T$ is the transposed map (\ref{PsiT}), which means the relations stated. In the linearised case,  using (\ref{Rr}) gives the relations stated. \end{proof}

\begin{ex}
 (1) For the permutation idempotent case $(X,r_f)$ the  Koszul dual of $\cA(\k,X,r_f)$ computed from (\ref{Rf}) has relations  $\sum_{a,b} y^a y^b\delta_{f(j),i}\delta_{j,b}=0$ for all $i,j$, which are empty unless $i=f(j)$, when they give the $n$ relations
\begin{equation}\label{kosperm} (\sum_a y^a)y^i=0,\quad 1\le i\le n.\end{equation}
These relations do not depend on $f$, which is possible in view of  Section~\ref{sec:perm} where we saw that all the YB algebras as isomorphic for a given cardinality of $X$.

(2) For Example~\ref{ex2} we directly use the set-theoretic result stated to obtain the Koszul dual of $\cA(\k,X,r)$ as given by the 3 relations
\[   \sum_i (y^i)^2=0,\quad  y^1 y^3+y^3y^2+y^2y^1=0,\quad y^1 y^2+ y^2 y^3+y^3 y^1=0.\]
(This is the algebra of left-invariant 1-forms of the minimal prolongation exterior algebra $\Omega_{min}$ of the finite group $S_3$ with its 3D differential calculus, see \cite[Chap.~1]{BeggsMajid}.)
\end{ex}

Next, another general construction for an object $V$ in a (pre)braided  category is the associated so-called Nichols-Woronowicz algebra,
\[ S^\pm_\Psi(V)= TV/\oplus_{m\ge 2} \ker [m,\pm \Psi]!\]
written here in the manner in which they were introduced in \cite{Ma:dbos} in terms of `braided  factorial operators' (these are equivalent to the symmetrizers in earlier work of \cite{Wor}). The new feature in \cite{Ma:dbos} was that these are always braided-Hopf algebras with the additive form of coproduct, counit and antipode $\Delta v=v\tens 1+1\tens v$, $\epsilon v=0$, $Sv=-v$ for all $v\in V$.  The braided factorials are defined inductively as
 \[   [m,\Psi]!=[m,\Psi]( [m-1,\Psi]!\tens\id);\quad   [m, \Psi]=\id + \Psi_{m-1}+ \Psi_{m-2}\Psi_{m-1}+\cdots +\Psi_1\cdots\Psi_{m-1}\]
 with $\Psi_i$ denoting the application of $\Psi$ in the $i,i+1$ copies of $V$ in $V^{\tens m}$. This is one of several factorisations in terms of different `braided integers', the one used here being the primed version in\cite{BeggsMajid}. The same construction applies with $\Psi$ replaced by $-\Psi$ and the prebraided category in this case has $-\Psi$ for the prebraiding on $V\tens V$.

 \begin{thm}\label{nichols} If $(V,\Psi)$ has $\Psi^2=\Psi$ then the Nichols algebra $S^-_\Psi(V)$ is quadratic with no new relations in higher degree.
 \end{thm}
\begin{proof} We write $[m]:=[m,-\Psi]$ for brevity and $[2]_i=\id-\Psi_i$ for this acting on the $i,i+1$ tensor factors. Let $v\in \ker[m]!$ and, noting that $[2]_1=\id-\Psi_1$ and $\Psi_1$ are complementary projections, write this (uniquely) as $v=u_1+v_1$ where $u_1\in\ker[2[_1$ and $v_1\in \ker\Psi_1$. Iterating the definition of $[m]!$, we note that this ends in $[2]_1\tens\id$ at the far right. Hence $[m]!u_1=0$, hence $[m]!v_1=0$. Now  similarly decompose $v_1=u_2+v_2$ with respect to $[2]_2, \Psi_2$ so that $u_2\in\ker[2]_2\cap\ker\Psi_1$ and $v_2\in\ker\Psi_2\cap\ker\Psi_1$. Acting on $u_2$, $[2]_1=\id$ and in the iteration of $[m]!$ the right most nontrivial factor is now $([3]\tens\id)u_1=((1-\Psi_2)+ \Psi_1\Psi_2)u_2=0$ so $[m]!u_2=0$ and hence $[m]!v_2=0$. We iterate this procedure. At the
$i$'th step we write $v_i=u_{i+1}+v_{i+1}$ where
\[ u_{i+1}\in \ker[2]_{i+1}\cap \ker\Psi_i\cdots\cap\ker\Psi_1,\quad  v_{i+1}\in \ker\Psi_{i+1}\cap \ker\Psi_i\cdots\cap\ker\Psi_1.\]
Then using the factorisation, $[m]!u_{i+1}$ will have $[i+1]\tens\id, [i]\tens\id,\cdots, [2]_1$ all acting as the identity and the first nontrivial factor will be $([i+2]\tens\id)u_{i+1}=0$ since (acting  on the  leftmost factors of $V$),
\[ [i+2]=1-\Psi_{i+1}+ \Psi_i\Psi_{i+1}+\cdots\pm  \Psi_1\cdots\Psi_i\Psi_{i+1}\]
and $\Psi_{i+1}$ acts as the identity and then $\Psi_i$ acts as zero. Hence, $[m]!u_{i+1}=0$ and hence $[m]!v_{i+1}=0$ also. The process continues to $v_{m-2}=u_{m-1}+v_{m-1}$ with $[m]!u_{m-1}=0$ as argued above and hence  $[m]!v_{m-1}=0$. But as all the $\Psi_1,\cdots,\Psi_{m-1}$ act as zero, the factorial here acts as the identity. Hence $v_{m-1}=0$. Hence we have shown that $v=u_1+u_2+\cdots+u_{m-1}$ with $u_i\in \ker[2]_i$. Hence we have proven that $\ker[m]!\subseteq \sum_i \ker[2]_i$. This says that the relations in degree $m$ already hold due to relations in degree 2.
\end{proof}

 \begin{cor} If $R$ is a linear solution of the YBE with $\Psi^2=\Psi$  then the braided-Hopf algebra $S^-_\Psi(V)$ is generated by $V={\rm span}_\k\{\theta_i\}$ with the relations
 \[ \sum_{a,b}\theta_b \theta_a R^a{}_i{}^b{}_j =0,\quad \forall\ 1\le i,j\le n,\]
 i.e. $S^-_\Psi(V)=S_+(R^T)^!$ when we identify the generators. When $R$ is obtained from $r$, we obtain  the relations
 \[ {}^{\theta_i}\theta_j  {\theta_i}^{\theta^j}=0,\quad 1\le i,j\le n,\]
where the left and right actions are as for $\{x_i\}$. We refer to this braided-Hopf algebra $\Lambda(\k,X,r)$ as the `fermionic' YB-algebra associated to an idempotent braided set $(X,r)$.
 \end{cor}
\begin{proof} The degree 2 relations of $S^-_\Psi(V)$ are to quotient by $\ker [2,-\Psi]=\ker(\id-\Psi)={\rm image}\Psi$, which are the relations as stated. Formula (\ref{Rr}) immediately gives what this looks like in the linearized case. Here, $\theta_i=x_i$ as elements of $V$ and have the corresponding left and right actions. Also, although ensured by the general theory of \cite{Ma:dbos}, it is useful to explicitly show that $\Delta \theta_a=\theta_a\tens 1+1\tens\theta_a$ extends to products as a braided Hopf algebra,
\begin{align*} \sum_{a,b}\Delta(\theta_b \theta_a) R^a{}_i{}^b{}_j&=\sum_{a,b}(\theta_b\tens 1+1\tens\theta_b)\cdot
(\theta_a\tens 1+1\tens\theta_a)R^a{}_i{}^b{}_j\\
&=\sum_{a,b}(\theta_b\theta_a\tens 1+1\tens\theta_b\theta_a+ \theta_b\tens\theta_a+
\Psi^-(\theta_b\tens\theta_a))R^a{}_i{}^b{}_j,\end{align*}
which vanishes precisely when $\Psi$ is idempotent. Here, the prebraiding in the category relevant to $S^-_\Psi(V)$ is now $\Psi^-(\theta_i\tens\theta_j)=-\sum_{a,b}\theta_b\tens \theta_a R^a{}_i{}^b{}_j$. The antipode is $S \theta_a=-\theta_a$ extended braided-antimultiplicatively with respect to $\Psi^-$, as one can similarly verify.  \end{proof}

This implies that the Koszul dual $S_+(R)^!$ in Lemma~\ref{lem:kosR} is a braided Hopf algebra for the transpose braiding $(\Psi^-)^T$ i.e. given by $\Psi^T$ as in the proof of the lemma but with an extra minus sign. Also note that, while we gave an elementary direct proof, the braided-Hopf algebra structure is a special case of a general constriction for additive `quantum braided planes' based on a pair of $R$-matrices in \cite{MajidQG}, namely $\check V(P+R, -R)$ in the notations there. Here $P+R, -R$ obey the required conditions when $R$ obeys the YBE and is $PR$ is idempotent.  By contrast, at least the quadratic algebra version of $S^{+}_\Psi(V)$ is $TV$ when $R$ is a linear solution of the YBE with $\Psi^2=\Psi$, i.e. with no relations.

\begin{ex} (1) For the permutation idempotent case, we have that $\Lambda(\k,X,r_f)$ has the $n$ relations
\begin{equation}\label{lambdaperm} f(\theta_i)\theta_i=0,\quad  \forall 1\le i\le n.\end{equation}
 Thus, for the Example~\ref{ex1} of $f$ a 3-cycle,  we have the 3 relations
\[ \theta_{2}\theta_{1}=0,\quad \theta_3\theta_2=0,\quad \theta_1\theta_3=0\]
(2) By contrast, for Example~\ref{ex2}, we have the 3 relations $\theta_i\theta_1=0$ for all $i=1,2,3$.
\end{ex}

\subsection{FRT bialgebra $A(R)$ and braided matrix bialgebra $B(R)$ in the idempotent case}

A key property of `Manin quantum planes' is that $S_+(R)$ for any $R$-matrix (usually a solution of the YBE) is necessarily a right
comodule algebra under a certain FRT bialgebra $A(R)$ via the algebra map
\begin{equation}\label{coact} x_i\mapsto \sum_a x_a\tens t^a{}_i,\end{equation}
where the FRT bialgebra has $n^2$ generators $\{t^i{}_j\}$ and relations\cite{FRT}
\[ \sum_{a,b}R^i{}_a{}^k{}_bt^a{}_j t^b{}_l=\sum_{a,b} t^k{}_b t^i{}_a R^a{}_j{}^b{}_l\]
along with a matrix coalgebra structure $\Delta t^i{}_j=\sum_a t^i{}_a\tens t^a{}_k$ and $\epsilon(t^i{}_j)=\delta_{ij}$. Because of the categorical nature of the construction,  $S^-_\Psi(V)$ is also a right comodule algebra, under
\begin{equation}\label{coactth}\theta_i\mapsto \sum_a \theta_a\tens t^a{}_i.\end{equation}
By a result of \cite[Thm.~4.1.5]{MajidQG}, $A(R)$ is always a coquasitriangular and hence its category of comodules is prebraided, and this includes the prebraided categories that we have been working with. On the $V^*$ side, we consider this a left comodule by
\[ y^i\mapsto \sum_a t^i{}_a\tens y^a\]
 making $S_+(R^T)$ and $S_+(R)^!$  left comodule algebras. Thus,  the $\{x_i\}$ transform like row vectors and $\{y^i\}$ in an opposite way like column vectors. For (bi)invertible $R$, there are general constructions for an associated Hopf algebra and an actual braided category, but  we do not have access to these when $\Psi$ is idempotent.

For a concrete example, if $(X,r_f)$ is a permutation idempotent braided set  then the  relations of the FRT bialgebra obtained from (\ref{Rf}) are
\begin{equation}\label{FRTperm}  (t^k{}_{f(l)} -\delta_{f(i),k}\sum_a t^a{}_j)t^i{}_l=0,\quad \forall\ 1\le i,j,k,l\le n.\end{equation}
For the simplest example, $|X|=2$ and $f=\id$, these amount to
\[ t^{\bar i}{}_j t^i{}_j=0,\quad (t^i{}_j)^2=(\sum_a t^a{}_{\bar j})t^i{}_j,\quad \forall\ 1\le i,j\le 2,\]
where $\bar i$ denotes the other index value to $i$.

Another important algebra in the linear R-matrix theory is the `reflection equation' or braided matrix algebra $B(R)$, see \cite[Def.~4.3.1]{MajidQG}, with a matrix of generators $u^i{}_j$ and
\[ \sum_{a,b,c,d} R^k{}_a{}^i{}_b u^b{}_c R^c{}_j{}^a{}_d u^d{}_l=\sum_{a,b,c,d}u^k{}_a R^a{}_b{}^i{}_c u^c{}_d R^d{}_j{}^b{}_l.\]
For bi-invertible R-matrices this forms a bialgebra in a braided category with matrix form of coalgebra, but when $\Psi$ is idempotent  $B(R)$ does not live in a prebraided category in an obvious way.

For the permutation idempotent case $(X,r_f)$, the braided matrix relations obtained from (\ref{Rf}) are
\begin{equation}\label{BRperm}  u^k{}_i u^i{}_l=\delta_{k,i}\sum_c u^k{}_c u^c{}_l,\quad \forall\ 1\le i,k,l\le n\end{equation}
which, like for the Koszul dual, is independent of $f$. For the simplest example  $|X|=2$, these amount to
\[ u^i{}_{\bar i} u^{\bar i}{}_j=0,\quad \forall\ 1\le i,j\le 2.\]

\subsection{Noncommutative differentials on $\cA(\k,X,r_{\id})$ and $S^-_\Psi(V)$}

In this final section of the paper, we take a first look at noncommutative differential structures on some of the quadratic algebras in earlier sections. This includes a general framework for differentials on quadratic algebras and a construction for all Nichols algebras $S^-_\Psi(V)$ in the idempotent case coming from its braided-Hopf algebra structure.

\subsubsection{Quantum differentials on quadratic algebras} The approach to quantum differentials in most approaches to noncommutative geometry, including\cite{Connes, BeggsMajid} is based on the observation that many noncommutative unital algebras $A$ do not admit sufficiently many derivations $A\to A$ to play the role of the classical
notion of partial differentials. Hence, instead, the notion of a derivation on $A$ is generalised to the following data.

\begin{dfn} \label{defcalc} Given a unital algebra $A$ over $\k$, a first order differential calculus means a pair
$(\Omega^1,\extd)$, where
\begin{enumerate}
\item $\Omega^1$ is an $A$-bimodule;
\item $\extd: A\to \Omega^1$ is a derivation in the sense $\extd(ab)=(\extd a)b+a\extd b$ for all $a,b\in A$;
\item The map $A\tens A\to \Omega^1$ sending $a\tens b\mapsto a\extd b$ is surjective.
\end{enumerate}
Here necessarily $\k.1\subseteq \ker \extd$, and $(\Omega^1,\extd)$ is called {\em connected} if $\ker\extd=\k.1$.
\end{dfn}

Given a first order calculus, there is a maximal extension to a differential graded algebra $(\Omega_{max},\extd)$, see
\cite[Lem.~1.32]{BeggsMajid}, with other differential graded algebras $(\Omega,\extd)$ over $A$ with the same $\Omega^1$ a
quotient of this. We recall that $\Omega$ here is a graded algebra with product denoted $\wedge$, $\Omega^0=A$ and $\extd$ is a
graded derivation with $\extd^2=0$.

\begin{rmk} A connected first order calculus always exists, namely there is a universal construction $\Omega^1_{uni}\subset
A\tens A $ defined as the kernel of the product with $\extd_{uni} a=1\tens a - a\tens 1$. Any other first order calculus a
quotient of this by an $A$-sub-bimodule. Also note that first order calculi are similar to the Kahler differential for
commutative algebras and have been used since the 1970s, for example in the works of Connes, Quillen and others.
\end{rmk}

\begin{lem} The Yang-Baxter algebra $\cA(\k,X,r_{f})$ of a permutation idempotent braided set $(X,r_f)$ for $|X|\ge 2$ does not admit any derivations that lower the degree by 1, other than
the zero map.
\end{lem}
\begin{proof} It suffices to take $f=\id$, and we enumerate $X=\{x_1,\cdots,x_n\}$ as usual. Let $D$ be degree lowering $D: A_i \to A_{i-1}$ and obey $D(ab)=aD(b) + D(a)b$ for all
$a, b \in A$. Then $D(x_1) = \alpha, D(x_2) = \beta$ for some $\alpha,\beta\in\k$. Hence $D(x_2x_1)= D(x_2) x_1 + x _2D(x_1) =
\beta x_1 + \alpha x_2.$ But $x_2x_1= x_1^2$ in $A$ and $D(x_1^2) = 2 \alpha x_2$, so $\beta x_1 + \alpha x_2 = 2 \alpha x_2$ and
hence $\alpha = \beta =0$ as $x_1,x_2$ are linearly independent.
 \end{proof}

We therefore do need a more general concept such as that of a first order differential calculus. For any quadratic algebra with
$n$ generators $x_1,\cdots,x_n$, we introduce a sufficient (but not necessary) construction for an $(\Omega^1,\extd)$ that reduces as expected
in the case of $\k[x_1,\cdots,x_n]$, as follows.

\begin{pro}\label{procalc} Let $A$ be a quadratic algebra on generators $\{x_i\}_{i=1}^n$ and let $\rho: A\to M_n(A)$ be an
algebra map such that
\begin{equation} \label{rho2} \sum_{i,j} r_{ij}( \rho^j{}_{ik}+ x_i\delta_{jk}) =0\quad\forall k\quad {\rm if}\quad
\sum_{i,j}r_{ij}x_ix_j=0, \end{equation}
where  $\rho^j{}_{ik}\in A$ are the matrix entries of $\rho(x_j)$ and $\delta_{jk}$ is the Kronecker $\delta$-function. Then

(1) $\Omega^1$ defined as a free left $A$-module with basis $\extd x_i$ and right module structure
\[   (a\extd x_i)b:= \sum_k (a\rho(b)_{ik})\extd x_k\]
is an $A$-bimodule.

(2) $\extd: A\to \Omega^1$ defined by  $\extd (1)=0, \extd(x_i)=\extd x_i$ extended as a derivation makes $(\Omega^1,\extd)$ into
a first order calculus.

(3) Partial derivatives $\del_i:A\to A$ defined  by
\begin{equation}\label{deli} \extd a= \sum_i (\del_i a)\extd x_i\end{equation}
for all $a\in A$ obey the twisted derivation rule
\begin{equation}\label{del2} \del_i(ab)=\sum_{j}\del_j(a) \rho(b)_{ji}+a\del_i(b)\end{equation}
for all $a,b\in A$.
\end{pro}
\begin{proof}
By definition, $\Omega^1=A\tens V$, where $V$ has basis which we denote $\{\extd x_i\}$, or equivalently $\Omega^1={\rm span}_A\{\extd
x_1,\cdots,\extd x_n\}$ with the $\extd x_i$ a left basis. The left action is by left multiplication by $A$, so $a(b\extd x_i):=
(ab)\extd x_i$. The right action stated is indeed an action as
\[ ((a\extd x_i).b).c=\sum_j ((a\rho(b)_{ij})\extd x_j).c=\sum_{j,k}(a\rho(b)_{ij}\rho(c)_{jk})\extd x_k=
\sum_{j}(a\rho(bc)_{ij})\extd x_j=(a\extd x_i).(bc).\] By construction, these  form a bimodule. Note that as $A$ here is
quadratic, an algebra map $\rho:A\to M_n(A)$ amounts to $\rho^j:=\rho(x_j)\in M_n(A)$ for $j=1,\cdots,n$ with entries
$\rho^j{}_{ik}\in A$
such that
\begin{equation} \label{rho1} \sum_{i,j} r_{ij} \rho^i \rho^j =0\quad {\rm if}\quad \sum_{i,j}r_{ij}x_ix_j=0,  \end{equation}
and the resulting bimodule is characterised by the bimodule relations
\begin{equation}\label{bimrel}  \extd x_i\ x_j:= \sum_k \rho^j{}_{ik} \extd x_k.\end{equation}

Next, we suppose (\ref{rho2}) and define $\extd: A\to \Omega^1$ as stated. This is well defined as a bimodule derivation since
\[ \extd(x_ix_j)=(\extd x_i)x_j+ x_i\extd x_j=\sum_k (\rho^j{}_{ik} + x_i\delta_{jk} )\extd x_k\]
under our assumption. As the algebra is quadratic, this implies that $\extd$ is well defined on all of $A$.

For the last part, we note that
\[ \extd (a b)=\sum_j \del_j(a)\extd x_j b+ a\sum_i\del_i(b)\extd x_i=\sum_i(\sum_{j}\del_j(a) \rho(b)_{ji}+a\del_i(b))\extd
x_i,\]
which implies the stated property of the $\del_i$ as the $\extd x_i$ are a left basis.
\end{proof}

Note that there is no implication that $\{\extd x_i\}$  are also a right basis, and they will not be in our examples below. Nor is it claimed that the construction is covariant under the FRT bialgebra when applied to the Yang-Baxter algebra of an $R$-matrix.

\subsubsection{Example of $\cA(\k,X,r_{\id})$ for $|X|=2$}\label{secA2calc}

We show how Proposition~\ref{procalc} works for the simplest nontrivial permutation idempotent example where $|X|=2$, say $X=\{x,y\}$. We also assume that $\k$ is not
characteristic 2. The Yang-Baxter algebra has relations
\[ (x-y)x=0,\quad (x-y)y=0\]
which are symmetric between the two generators.  We solve for matrices $\rho^1$ and $\rho^2$ obeying (\ref{rho1}), (\ref{rho2}) and
note that the latter implies the general form
\[ \rho^1=\begin{pmatrix} e & f \\ e+x-y & f\end{pmatrix},\quad \rho^2=\begin{pmatrix} g& h+ y-x\\ g & h\end{pmatrix}\]
for some elements $e,f,g,h\in A$. The former then lead to four independent equations among the entries,
\begin{equation*}\label{ef1} (e-g + f-h+z)f=(e-g+f-h+z)g=0,\end{equation*}
\begin{equation*}\label{ef2} fh-h^2+z h+ (e-g)(h-z)= ge-e^2- ze+ (h-f)(e+z)=0,\end{equation*}
where $z:=x-y$ is a shorthand. The simplest class of solutions to these is to assume that $e,f,g,h\in A^+$, i.e., have all their terms of strictly positive
degree (so each term has a left factor of $x$ or $y$) and that $h-e, f-g$ are each divisible by $z$ as a right factor. Then the
first two  equations are automatic as is the sum of the latter two. All that remains is their difference, which reduces to
\[ (h+e-f-g)z=0.\]
These requirements do not have a unique solution, but the lowest degree solution is to take $e,f,g,h$ to be degree 1 with
\[ f=\lambda x+ (1-\lambda)y,\quad g=\mu x+ (1-\mu)y,\quad e=\alpha x+ (1-\alpha)y,\quad h=\beta x+ (1-\beta)y\]
for parameters $\lambda,\mu,\alpha,\beta\in \k$. The result can be written compactly as
\begin{equation}\label{rhoC} \rho^i{}_{jk}=y+(\epsilon_{ij}\delta_{ik}+C_{ik})z,\quad C=\begin{pmatrix} \alpha & \lambda \\ \mu &\beta\end{pmatrix},\end{equation}
where $\epsilon_{ij}$ is antisymmetric  with $\epsilon_{12}=1=-\epsilon_{21}$ and other entries zero.  The
 bimodule relations are
\[ \extd x\ x=(y+\alpha z)\extd x+ (y+\lambda z)\extd y,\quad \extd x\ y=(y+\mu z)\extd x+ (y+(\beta-1) z)\extd y,\]
\[ \extd y\ x= (y+(\alpha+1) z)\extd x+ (y+\lambda z)\extd y,\quad \extd y\ y =(y+\mu z)\extd x+ (y+\beta z)\extd y.\]
Note that $\extd z\ x=-z\extd x$  and hence that $(\extd x-\extd y)a=0$ for any
$a\in A$ with all terms of degree $\ge 2$. Hence $\extd x,\extd y$ are never a right-basis.

Given a left basis, one can compute the associated partial derivatives, which then begins to look like more familiar differential calculus. We illustrate this for a concrete choice within our above moduli of differential structures.

\begin{ex}\label{calcsym} We take
\[ f=e=x, \quad g=h=y,\quad C=\begin{pmatrix}1 & 1\\ 0& 0\end{pmatrix}. \]
The bimodule relations are then
\[\extd x\ x=x(\extd x+\extd y) ,\quad \extd y\ x=(2x-y)\extd x + x\extd y,\quad \extd x\ y=y\extd x+ (2y-x)\extd y,\quad \extd
y\ y=y(\extd x+\extd y).\]
 Next, by iterating the bimodule relations, one finds
\[ \extd x_i \begin{cases} x^m \\ y^m\end{cases}= 2^{m-2}\begin{cases}(3 x^m-y^m)\extd x+ 2 x^m \extd y \\
2 y^m \extd x+ (3y^m-x^m)\extd y\end{cases},\quad \forall m\ge 2\]
independently of $i$. The partials can then be computed by
iterating the Leibniz rule for $\extd$ using these relations, to find
\[\del_i (1)=0,\quad \del_i (x_j)=\delta_{ij},\quad \del_i (x_i^m)=(3\ 2^{m-1}-1)x_i^{m-1}- (2^{m-2}-1)x_{\bar i}^{m-1},\quad
\del_i (x_{\bar i}^m)=(2^{m-1}-1)x_{\bar i}^{m-1},\]
for $m\ge 2$, where $x_{\bar i}$ denotes the other generator from $x_i$. As $A=\k1\oplus\k[x]^+\oplus\k[y]^+$, this specifies the
linear maps $\del_i$. They lower degree  by 1  but are not derivations. From these formulae, it follows easily that
$\del_i(a)=0$ implies $a\in \k 1$, and hence that the calculus is connected.
\end{ex}

\subsubsection{Differentials calculus on $S^-_\Psi(V)$} Since $S^-_\Psi(V)$ is a braided-Hopf algebra, there is a canonical construction in \cite[Chap.~2]{BeggsMajid} for a differential calculus where the partial derivatives are braided-derivations and that lives in the category and is in fact covariant under the FRT bialgebra. We define
\[ \Omega^1(S^-_\Psi(V))=S^-_\Psi(V)\underline{\tens} V,\quad \extd = (\id\tens \pi_1)\Delta\]
where $\pi_1$ projects to the degree 1 component of $S^-_\Psi(V)$. The braided tensor product indicates that the bimodule relations between $\extd\theta_i$ and $\theta_j$ are
\[ (\extd \theta_i)\theta_j= - \theta_b\extd\theta_a R^a{}_i{}^b{}_j,\quad \forall\ 1\le i,j\le n.\]
This extends in a natural way to an exterior algebra
\[ \Omega(S^-_\Psi(V))=S^-_\Psi(V)\underline{\tens} S_+(R),\]
where
\[ \extd \theta_i\wedge  \extd\theta_j= \extd\theta_b\wedge \extd\theta_a R^a{}_i{}^b{}_j,\quad \forall\ 1\le i,j\le n,\]
with $\extd$ extended as graded derivation. In other words, the YB algebra $S_+(R)$ now appears as the algebra of quantum differential forms $\extd\theta_i$ on  $S^-_\Psi(V)$, which is a remarkable reversal of the more familiar roles for symmetric and exterior algebras. Because of the categorical construction, the differential calculus is covariant under (\ref{coact}) (applied to $\extd \theta_i$) and (\ref{coactth}) taken together.

\begin{ex} For $(X,r_f)$ a permutation idempotent braided set,  $\Lambda(\k,X,r_f)$ in Example~\ref{lambdaperm} acquires an exterior algebra $\Omega(\Lambda(\k,X,r_f))$ generated by $\theta_i,\extd\theta_i$ with relations
\[ (\extd\theta_i)\theta_j=-\theta_{f(j)}\extd\theta_j,\quad  \extd\theta_i\wedge \extd\theta_j= \extd\theta_{f(j)}\wedge \extd\theta_j,\quad \forall\ 1\le i,j\le n.\]
These relations, along with $\extd$ a graded derivation that squares to 0, fully specifies the exterior algebra by this construction. For the simplest case of $|X|=2$ and $f=\id$, we write the 2 generators of $\Lambda(\k,X,r_{\id})$ as two coordinates $x=\theta_1,y=\theta_2$ and the relations for this and its differentials now appear as
\[  x^2=0,\quad y^2=0,\quad \extd x \wedge \extd y=\extd y\wedge \extd y,\quad \extd y\wedge \extd x= \extd x\wedge\extd x\]
\[(\extd x)x =(\extd y)x=-x\extd x,\quad (\extd x)y=(\extd y)y=-y\extd y. \]
Here, the differentials $\extd x,\extd y$  obey the Yang-Baxter algebra relations as previously for $x,y$ in Section~\ref{secA2calc}.
\end{ex}


\end{document}